\ifdefined\pdfoutput \pdfoutput=1 \fi
\documentclass[11pt]{amsart}

\usepackage[margin=1in]{geometry}

\usepackage{amsmath,amssymb,amsfonts,amsthm,mathtools}
\usepackage{mathrsfs}

\usepackage{graphicx}
\usepackage{tikz}
\usepackage{float}
\usepackage{booktabs}

\usepackage{enumitem}
\usepackage[hidelinks]{hyperref}

\usepackage{microtype}
\allowdisplaybreaks[4]

\theoremstyle{plain}
\newtheorem{theorem}{Theorem}
\newtheorem{lemma}{Lemma}

\theoremstyle{definition}

\newtheorem{problem}{Problem}

\theoremstyle{remark}
\newtheorem{remark}{Remark}

\title[Direct sampling methods for inverse elastic medium scattering]{Direct sampling methods for inverse medium scattering problems of elastic waves}

\author{Lu Zhao}
\address{College of Science, Civil Aviation University of China, Tianjin, China}
\email{zhaol@cauc.edu.cn}

\author{Yiling Li}
\address{College of Science, Civil Aviation University of China, Tianjin, China}
\email{2023061032@cauc.edu.cn}

\author{Zhiyong Cheng}
\address{School of Mathematics, Jilin University, Changchun, Jilin 130012, China}
\email{chengzy22@mails.jlu.edu.cn}
\email{chengzy97@gmail.com}
\thanks{Corresponding author: Zhiyong Cheng}

\date{}


\begin{document}

\begin{abstract}
This paper concerns the inverse elastic scattering problem of determining an unknown penetrable obstacle from far-field data. Using the Helmholtz decomposition, the coupled boundary value problem is reformulated as a coupled scalar Helmholtz system. We prove the uniqueness of the associated coupled Helmholtz boundary value problem and of the corresponding boundary integral equation system, which is then discretized by a Nystr\"{o}m method for efficient numerical computation. Leveraging the relation between the compressional and shear far-field patterns of the Navier system and those of the coupled Helmholtz system, we employ three indicators to reconstruct the location and shape of the obstacle. Furthermore, we analyze the decay properties of these indicators and establish corresponding stability estimates. Numerical experiments are presented to illustrate the effectiveness and robustness of the proposed method, even for limited-aperture data.
\end{abstract}

\maketitle

\section{Introduction}\label{sec1}

In this paper, we consider an inverse elastic scattering for penetrable obstacle from far-field data. In various fields of science and technology, such as noninvasive testing, geophysical exploration, and medical imaging, this type of inverse problem has attracted substantial interest \cite{eimag,simag,elastic}.

For the direct elastic wave scattering problem, two principal numerical methods have been extensively developed. The first is a boundary integral equation (BIE) method \cite{Bouine}, which restricts computations to the boundary of obstacle. When combined with appropriate discretization schemes, this method effectively reduces computational costs. In \cite{DongL,Drule}, the Helmholtz decomposition was used to transform the elastic wave equation to a coupled system of Helmholtz equations, and a Nystr\"{o}m discretization of the associated boundary integral equations was developed to solve the resulting system. For the multiple obstacle scattering problem, a GMRES iterative solver accelerated by the fast multipole method was proposed in \cite{Fmm} to efficiently solve the boundary integral equations. Moreover, Galerkin BIE methods \cite{2DGral,3DGral} have been successfully applied to elastic scattering in both two and three dimensions. For the three-dimensional elastic-wave scattering problem, a high-order combined-field spectral method was introduced in \cite{spect2}.
 Building on the Galerkin framework and spectral techniques, \cite{spect1} developed a fully discrete spectral scheme for the three-dimensional elastic obstacle scattering problem that reduces the evaluation of the integral operators to scalar operations, thereby greatly simplifying the numerical implementation. 
  Boundary integral formulations have also been used for elastic transmission problems. In two dimensions, coupled singular integral equations on the interface were derived and shown to be uniquely solvable \cite{martin1990}, while regularized BIE formulations with high-order Nystr\"{o}m discretizations were later developed for two-dimensional elastodynamic transmission problems \cite{DominguezTurc2022}.
 The second major class comprises artificial boundary methods \cite{Arb}, which introduce either a Dirichlet-to-Neumann (DtN) map or a perfectly matched layer (PML) truncation to reformulate the scattering problem as a boundary value problem on a bounded computational domain; the resulting problem is then solved by finite-element or finite-difference methods. In \cite{fin}, a priori and a posteriori error estimates were established for finite element discretizations coupled with the DtN map. For the elastic obstacle scattering problem, an adaptive finite element DtN method guided by a posteriori error estimators was developed in \cite{adfin}. Moreover, \cite{pml1} presented an adaptive finite-element PML method for elastic scattering by a rigid obstacle.

For the inverse elastic scattering problem, numerical methods are typically classified into two types: quantitative and qualitative methods. Quantitative methods typically rely on iterative algorithms. For instance, two Landweber iterative methods proposed in \cite{land} utilize multi-frequency scattered field data and its phaseless far-field pattern to reconstruct multiple elastic parameters of inhomogeneous medium. A gradient-descent algorithm that uses scattered field data to reconstruct penetrable elastic bodies was developed in \cite{descent}. To reconstruct rigid obstacles from either phased or phaseless far-field data generated by a single incident plane wave, \cite{DongL} proposed a reference-ball based nonlinear integral equation method. Building on this framework, \cite{cheng} employed a point source and measured the far-field pattern of the total field, thereby eliminating the need for the reference-ball device in phaseless reconstruction. In addition, based on the decomposition method, \cite{Newton} proposed a novel Newton-type approach for the inverse elastic scattering problem using scattered field data. Compared with quantitative methods, qualitative approaches typically employ indicators for reconstructing the obstacle boundary, without requiring a priori knowledge of the geometry and physical properties of obstacle. For example, the factorization method enables the reconstruction of rigid obstacle using only compressional or shear far-field patterns, thereby eliminating the requirement for full-aperture data \cite{factorm1}. The linear sampling method \cite{LSM} was proposed to reconstruct penetrable isotropic obstacles in an isotropic elastic medium. Additionally, \cite{RTM} introduced a single-frequency, weighted reverse-time migration method for imaging extended obstacles in an elastic medium. In addition to the aforementioned qualitative methods, we focus in particular on the direct sampling method developed in \cite{DSM,DSM2}. Using phased compressional and shear far-field data, \cite{DSM} introduced three indicators for the elastic scattering problem and devised data recovery procedures for limited aperture measurements cases. For phaseless case, the artificial rigid body technique makes it possible to uniquely determine a rigid obstacle from phaseless far-field data at a fixed frequency \cite{DSM2}. Further, a phase retrieval technique combined with the direct sampling methods was also proposed in \cite{DSM2} to reconstruct the rigid obstacle. For sampling-type imaging algorithms targeting unbounded rough surfaces, as well as deep-learning approaches integrated with classical numerical methods, we refer to \cite{nonsam,dp1}. Recently, theoretical developments on the inverse elastic medium scattering problem have been established in \cite{theoret1,theoret3,theory2}, while development of numerical algorithms remains relatively limited \cite{na3,na2}. As for the direct sampling method for inverse acoustic medium scattering, we refer to \cite{zou1,LXD}.

In this work, we consider the direct and inverse elastic scattering problem for bounded, penetrable obstacles embedded in a homogeneous medium. 
In contrast to existing boundary integral formulations for elastic transmission problems \cite{martin1990,DominguezTurc2022}, which are constructed directly for the vector Navier system and involve hypersingular operators with vector-valued kernels, our approach first applies the Helmholtz decomposition \cite{Drule,Djump} to reformulate the transmission problem as a coupled scalar Helmholtz system. This reduction replaces the vector traction operator by scalar boundary operators and allows us to make use of well-established tools for scalar Helmholtz problems, in particular the jump relations for the second derivatives of the single-layer potential. On this basis, we establish the uniqueness of the coupled boundary value problem and of the resulting boundary integral equation system, and solve the latter efficiently by a Nystr\"{o}m-type discretization.
For the inverse problem, building on the indicator framework in \cite{DSM}, we extend the direct sampling method from rigid to penetrable obstacles. By exploiting the relation between the compressional and shear far-field patterns of the Navier system and those of the coupled Helmholtz system, we adapt three indicator functions that reconstruct the location and shape of the obstacle without separately treating different polarization directions, and we further analyze their decay properties and establish corresponding stability estimates. Since the indicators are evaluated through direct multiplications of the far-field data with suitable test functions, the reconstruction procedure involves only scalar operations and is therefore straightforward and efficient to implement.

The goal of this work is fourfold:
\begin{enumerate}
	\item[(1)] develop a fully discrete formulation for the forward elastic scattering problem by combining the Helmholtz decomposition with a Nystr\"{o}m-type discretization;
	\item[(2)] establish the uniqueness of the associated coupled Helmholtz boundary value problem and of the resulting boundary integral equation system;
	\item[(3)] extend the indicator-based direct sampling framework from rigid obstacles to penetrable obstacles by exploiting the compressional and shear far-field patterns;
	\item[(4)]  analyze the decay properties of these indicators and derive the corresponding stability estimates.
\end{enumerate}

The paper is organized as follows. In Section 2, we introduce the problem formulation and reformulate the original model, by means of the Helmholtz decomposition, as a coupled boundary value problem for the Helmholtz equations. The uniqueness of this boundary value problem is also proved. In Section 3, we derive the corresponding boundary integral equations, prove their uniqueness, and solve them numerically using a Nystr\"{o}m-type discretization scheme.
Section 4 gives three indicators for reconstructing the location and shape of a penetrable obstacle from far-field data, analyzes their decay properties as the sampling points move away from the obstacle boundary, and establishes stability results. Numerical experiments in Section 5 validate the effectiveness of the proposed indicators. The paper concludes in Section 6.

\section{Problem Formulation}

Let \(D\subset \mathbb{R}^2\) be a bounded, simply-connected penetrable elastic obstacle with smooth boundary \(\partial D\) and mass density \(\rho_1\). The exterior region \(\mathbb{R}^2\setminus\overline{D}\)
is occupied by a homogeneous and isotropic medium of mass density \(\rho_2\). The direct scattering problem for a time-harmonic elastic wave is to find the displacement fields
\(
\pmb{U}_1 \in (C^2(D))^2 \cap (C^1(\overline{D}))^2,~ 
\pmb{U}_2 \in (C^2(\mathbb{R}^2\setminus\overline{D}))^2 \cap (C^1(\mathbb{R}^2\setminus D))^2, 
\)
satisfying the two-dimensional Navier equations 
\begin{equation}\label{2.1}
	\begin{cases}
		\mu_{1} \Delta\pmb{U}_1 + (\lambda_{1}+\mu_{1}) \nabla(\nabla\cdot\pmb{U}_1)
		+ \rho_{1} \omega^{2} \pmb{U}_1 = 0
		&\text{in }D,\\[6pt]
		\mu_{2} \Delta\pmb{U}_2 + (\lambda_{2}+\mu_{2}) \nabla(\nabla\cdot\pmb{U}_2)
		+ \rho_{2} \omega^{2} \pmb{U}_2 = 0
		&\text{in }\mathbb{R}^2\setminus\overline{D},
	\end{cases}
\end{equation}
where \(\omega>0\) denotes the angular frequency, and \(\lambda_1\), \(\mu_1\) and \(\lambda_2\), \(\mu_2\) represent the Lam\'{e} constants for the obstacle \(D\) and exterior domain \(\mathbb{R}^2\setminus\overline{D}\), respectively, satisfying \(\mu_i>0,\lambda_i+\mu_i>0\), \(i=1,2\). 

The continuity of displacement and traction across the boundary \(\partial D\) implies
\begin{align}\label{bc}
	\pmb{U}_1 = \pmb{U}_2,\qquad
	T_{\pmb \nu,1}(\pmb{U}_1) = T_{\pmb \nu,2}(\pmb{U}_2),
\end{align}
where the traction operators are defined by
\[
T_{\pmb \nu,i}(\pmb{U}) = 2\mu_i\partial_{\pmb{\nu}}\pmb{U}
+ \lambda_i\pmb{\nu}(\nabla\cdot\pmb{U})
- \mu_i\pmb{\tau}\operatorname{curl}\pmb{U},\quad i=1,2. 
\]
Here, \(\pmb{\nu}=(\nu_1,\nu_2)^\top\) and \(\pmb{\tau}=\pmb{\nu}^{\perp}=(-\nu_2,\nu_1)^\top\) are the unit exterior normal and tangent vectors on \(\partial D\).  The scalar curl of a vector field \(\pmb{v}=(v_1,v_2)^\top\)  is \(
\operatorname{curl}\pmb{v} = \partial_{x_1}v_2 - \partial_{x_2}v_1
\), whereas the vector curl of a scalar field 
$v$ is given by \(
\pmb{\operatorname{curl}}v = \bigl(\partial_{x_2}v,-\partial_{x_1}v\bigr)^\top
\).

Let the incident field be the superposition of compressional and shear plane waves
\[
\pmb{u} = a_{p}\pmb{u}_{p}  +  a_{s}\pmb{u}_{s}.
\]
Here, \((a_{p},a_{s})\in\mathbb{R}^2\) are the incident field coefficients, and the plane waves are defined by
\[
\pmb{u}_{p}(\pmb{x})
= \pmb{d}{\rm e}^{\mathrm{i}\kappa_{p}\pmb{x}\cdot\pmb{d}}, 
\quad
\pmb{u}_{s}(\pmb{x})
= \pmb{d}^{\perp}{\rm e}^{\mathrm{i}\kappa_{s}\pmb{x}\cdot\pmb{d}},
\]
where \(\pmb{d}=(\cos\theta,\sin\theta)^{\top}\) is the unit propagation direction for an angle \(\theta\in[0,2\pi)\), and \(\pmb{d}^{\perp}=(-\sin\theta,\cos\theta)^{\top}\) is the corresponding orthogonal unit vector. The respective wavenumbers $\kappa_{p}$ and $\kappa_{s}$ are given by
\[
\kappa_{p} = \omega\sqrt{\frac{\rho_2}{\lambda_{2}+2\mu_{2}}},\qquad
\kappa_{s} = \omega\sqrt{\frac{\rho_2}{\mu_{2}}}.
\]

Moreover, the scattered field \(\pmb{v}=\pmb{U}_{2}-\pmb{u}\) must satisfy the Kupradze-Sommerfeld radiation condition
\begin{align*}
	\lim_{r\to\infty}\sqrt{r}\bigl(\partial_{r}\pmb{v}_{p}
	- \mathrm{i}\kappa_{p}\pmb{v}_{p}\bigr) = 0,
	\quad
	\lim_{r\to\infty}\sqrt{r}\bigl(\partial_{r}\pmb{v}_{s}
	- \mathrm{i}\kappa_{s}\pmb{v}_{s}\bigr) = 0,
	\quad
	r = \lvert\pmb{x}\rvert,
\end{align*}
where $\pmb{v}_{p}$ and $\pmb{v}_{s}$ denote the compressional and shear components of \(\pmb{v}\), defined respectively as
\[
\pmb{v}_{p}
= -\frac{1}{\kappa_{p}^{2}}
\nabla(\nabla  \cdot  \pmb{v}),
\qquad
\pmb{v}_{s}
= \frac{1}{\kappa_{s}^{2}}
\pmb{\operatorname{curl}}\bigl(\operatorname{curl}\pmb{v}\bigr).
\]

The uniqueness of the elastic transmission problem \eqref{2.1}--\eqref{bc}, subject to the Kupradze--Sommerfeld radiation condition, has been established in \cite{martin1990,descent}.

Following \cite{DongL}, we apply the Helmholtz decomposition to \(\pmb{U}_1\) and \(\pmb{v}\)
\begin{equation}\label{HelmDec}
	\pmb{U}_1 = \nabla\phi_1 + \pmb{\operatorname{curl}}\psi_1,\quad
	\pmb{v}    = \nabla\phi_2 + \pmb{\operatorname{curl}}\psi_2,
\end{equation}
where \(\phi_1,\psi_1\) and \(\phi_2,\psi_2\) are scalar potential functions. Substituting \eqref{HelmDec} into \eqref{2.1} yields the following Helmholtz equations
\[
\begin{cases}
	\Delta\phi_{1} + \kappa_{a}^{2}\phi_{1} = 0,\quad
	\Delta\psi_{1} + \kappa_{b}^{2}\psi_{1} = 0 \quad {\rm in}~D,\\[4pt]
	\Delta\phi_{2} + \kappa_{p}^{2}\phi_{2} = 0,\quad
	\Delta\psi_{2} + \kappa_{s}^{2}\psi_{2} = 0\quad {\rm in}~\mathbb{R}^2\setminus\overline{D}, 
\end{cases}
\]
where the compressional and shear wavenumbers $\kappa_a,\kappa_b$ are given by
\[
\kappa_{\mathrm{a}} = \omega\sqrt{\frac{\rho_1}{\lambda_{1}+2\mu_{1}}},\qquad
\kappa_{b} = \omega\sqrt{\frac{\rho_1}{\mu_{1}}}.
\]
Additionally, \(\phi_2\) and \(\psi_2\) are required to satisfy the Sommerfeld radiation conditions
\[
\lim_{r\to\infty} r^{\tfrac12}(\partial_{r}\phi_2 - \mathrm{i}\kappa_{p}\phi_2)=0,\quad
\lim_{r\to\infty} r^{\tfrac12}(\partial_{r}\psi_2 - \mathrm{i}\kappa_{s}\psi_2)=0,
\quad r = \left|\pmb{x}\right|.
\]
Substituting the decomposition \eqref{HelmDec} into \eqref{bc} and taking inner products with \(\pmb{\nu}\) and \(\pmb{\tau}\), respectively, we obtain the following boundary conditions on \(\partial D\)
\[
\begin{cases}
	\partial_{\pmb{\nu}}\phi_{1} + \partial_{\pmb{\tau}}\psi_{1}
	- \partial_{\pmb{\nu}}\phi_{2} - \partial_{\pmb{\tau}}\psi_{2}
	= f_{1},\\[3pt]
	\partial_{\pmb{\tau}}\phi_{1} - \partial_{\pmb{\nu}}\psi_{1}
	- \partial_{\pmb{\tau}}\phi_{2} + \partial_{\pmb{\nu}}\psi_{2}
	= f_{2},\\[3pt]
	2\bigl(\mu_{1}\pmb{\nu}  \cdot  \partial_{\pmb{\nu}}\nabla\phi_{1}
	+\mu_{1}\pmb{\nu}  \cdot  \partial_{\pmb{\nu}}\pmb{\operatorname{curl}}\psi_{1}\bigr)
	- \lambda_{1}\kappa_{a}^{2}\phi_{1}\\
	\qquad-2\bigl(\mu_{2}\pmb{\nu}  \cdot  \partial_{\pmb{\nu}}\nabla\phi_{2}
	+\mu_{2}\pmb{\nu}  \cdot  \partial_{\pmb{\nu}}\pmb{\operatorname{curl}}\psi_{2}\bigr)
	+ \lambda_{2}\kappa_{p}^{2}\phi_{2}
	= f_{3},\\[3pt]
	2\bigl(\mu_{1}\pmb{\tau}  \cdot  \partial_{\pmb{\nu}}\nabla\phi_{1}
	+\mu_{1}\pmb{\tau}  \cdot  \partial_{\pmb{\nu}}\pmb{\operatorname{curl}}\psi_{1}\bigr)
	- \mu_{1}\kappa_{b}^{2}\psi_{1}\\
	\qquad-2\bigl(\mu_{2}\pmb{\tau}  \cdot  \partial_{\pmb{\nu}}\nabla\phi_{2}
	+\mu_{2}\pmb{\tau}  \cdot  \partial_{\pmb{\nu}}\pmb{\operatorname{curl}}\psi_{2}\bigr)
	+ \mu_{2}\kappa_{s}^{2}\psi_{2}
	= f_{4},
\end{cases}
\]
with the terms $f_1$ to $f_4$ given by
\[
f_{1} = \pmb{\nu}  \cdot  \pmb{u},\quad
f_{2} = \pmb{\tau}  \cdot  \pmb{u},\quad
f_{3} = \pmb{\nu}  \cdot  T_{\pmb \nu,2}(\pmb{u}),\quad
f_{4} = \pmb{\tau}  \cdot  T_{\pmb \nu,2}(\pmb{u}).
\]

In summary, the scalar potential functions \(\phi_1\), \(\psi_1\) and \(\phi_2\),  \(\psi_2\) are governed by the following coupled boundary value problem
\begin{equation}\label{2.5}
	\left\{
	\begin{aligned}
		&\Delta\phi_1+\kappa_{a}^2\phi_1=0,\quad \Delta\psi_1+\kappa_{b}^2\psi_1=0
		&&\text{in } D,\\
		&\Delta\phi_2+\kappa_{p}^2\phi_2=0,\quad \Delta\psi_2+\kappa_{s}^2\psi_2=0
		&&\text{in } \mathbb{R}^2\setminus\overline{D},\\
		&\partial_{\boldsymbol{\nu}}\phi_1+\partial_{\boldsymbol{\tau}}\psi_1-\partial_{\boldsymbol{\nu}}\phi_2-\partial_{\boldsymbol{\tau}}\psi_2=f_1
		&&\text{on } \partial D,\\
		&\partial_{\boldsymbol{\tau}}\phi_1-\partial_{\boldsymbol{\nu}}\psi_1-\partial_{\boldsymbol{\tau}}\phi_2+\partial_{\boldsymbol{\nu}}\psi_2=f_2
		&&\text{on } \partial D,\\
		&\begin{aligned}[b]
			&2(\mu_1\boldsymbol{\nu}\cdot\partial_{\boldsymbol{\nu}}\nabla\phi_1+\mu_1\boldsymbol{\nu}\cdot\partial_{\boldsymbol{\nu}}\boldsymbol{\operatorname{curl}}\psi_1)-\lambda_1\kappa_{a}^2\phi_1\\
			&\quad-2(\mu_2\boldsymbol{\nu}\cdot\partial_{\boldsymbol{\nu}}\nabla\phi_2+\mu_2\boldsymbol{\nu}\cdot\partial_{\boldsymbol{\nu}}\boldsymbol{\operatorname{curl}}\psi_2)+\lambda_2\kappa_{p}^2\phi_2=f_3
		\end{aligned}
		&&\text{on } \partial D,\\
		&\begin{aligned}[b]
			&2(\mu_1\boldsymbol{\tau}\cdot\partial_{\boldsymbol{\nu}}\nabla\phi_1+\mu_1\boldsymbol{\tau}\cdot\partial_{\boldsymbol{\nu}}\boldsymbol{\operatorname{curl}}\psi_1)-\mu_1\kappa_{b}^2\psi_1\\
			&\quad-2(\mu_2\boldsymbol{\tau}\cdot\partial_{\boldsymbol{\nu}}\nabla\phi_2+\mu_2\boldsymbol{\tau}\cdot\partial_{\boldsymbol{\nu}}\boldsymbol{\operatorname{curl}}\psi_2)+\mu_2\kappa_{s}^2\psi_2=f_4
		\end{aligned}
		&&\text{on } \partial D,\\
		&\lim_{\rho\to\infty}\rho^{\frac{1}{2}}(\partial_{\rho}\phi_2-\mathrm{i}\kappa_{p}\phi_2)=0,\quad
		\lim_{\rho\to\infty}\rho^{\frac{1}{2}}(\partial_{\rho}\psi_2-\mathrm{i}\kappa_{s}\psi_2)=0
		&&\rho=|\boldsymbol{x}|.
	\end{aligned}
	\right.
\end{equation}

Compared with the rigid-obstacle case, the coupled problem \eqref{2.5} is mathematically more complex. The transmission conditions require the simultaneous determination of the interior potentials $\phi_1$, $\psi_1$ and the exterior radiating potentials $\phi_2$, $\psi_2$. Moreover, the first two boundary conditions couple the normal and tangential derivatives of the compressional and shear potentials, while the last two, arising from the continuity of traction, involve second-order boundary terms. These features make both the following uniqueness analysis and the derivation of the boundary integral equations in Section~3 more delicate than in the rigid-obstacle setting.

\begin{theorem}\label{uniqueness of coupled system}
The coupled boundary value problem \eqref{2.5} has at most one solution.
\end{theorem}

\begin{proof}
It suffices to prove that $\phi_1=\psi_1=0$ in $D$ and $\phi_2=\psi_2=0$ in $\mathbb{R}^2\setminus\overline{D}$ when $f_1=f_2=f_3=f_4=0$. 
Define the scalar boundary traces
\[
q_1:=\partial_{\pmb\nu}\phi_1+\partial_{\pmb\tau}\psi_1,\qquad
r_1:=\partial_{\pmb\tau}\phi_1-\partial_{\pmb\nu}\psi_1,
\]
\[
q_2:=\partial_{\pmb\nu}\phi_2+\partial_{\pmb\tau}\psi_2,\qquad
r_2:=\partial_{\pmb\tau}\phi_2-\partial_{\pmb\nu}\psi_2,
\]
and
\[
m_1:=2\bigl(\mu_1\pmb\nu\cdot\partial_{\pmb\nu}\nabla\phi_1
+\mu_1\pmb\nu\cdot\partial_{\pmb\nu}\boldsymbol{\operatorname{curl}}\psi_1\bigr)
-\lambda_1\kappa_a^2\phi_1,
\]
\[
n_1:=2\bigl(\mu_1\pmb\tau\cdot\partial_{\pmb\nu}\nabla\phi_1
+\mu_1\pmb\tau\cdot\partial_{\pmb\nu}\boldsymbol{\operatorname{curl}}\psi_1\bigr)
-\mu_1\kappa_b^2\psi_1,
\]
\[
m_2:=2\bigl(\mu_2\pmb\nu\cdot\partial_{\pmb\nu}\nabla\phi_2
+\mu_2\pmb\nu\cdot\partial_{\pmb\nu}\boldsymbol{\operatorname{curl}}\psi_2\bigr)
-\lambda_2\kappa_p^2\phi_2,
\]
\[
n_2:=2\bigl(\mu_2\pmb\tau\cdot\partial_{\pmb\nu}\nabla\phi_2
+\mu_2\pmb\tau\cdot\partial_{\pmb\nu}\boldsymbol{\operatorname{curl}}\psi_2\bigr)
-\mu_2\kappa_s^2\psi_2.
\]
Then the homogeneous boundary conditions in \eqref{2.5} can be written as
\begin{equation}\label{boundary condition bc}
q_1=q_2,\qquad r_1=r_2,\qquad m_1=m_2,\qquad n_1=n_2
\qquad\text{on }~\partial D.
\end{equation}

Applying the first Betti formula to
$\nabla\phi_1+\boldsymbol{\operatorname{curl}}\psi_1$ yields
\begin{align}
&\int_D
\Bigl(
\lambda_1\bigl|\nabla\cdot(\nabla\phi_1+\boldsymbol{\operatorname{curl}}\psi_1)\bigr|^2
+2\mu_1\bigl|\varepsilon(\nabla\phi_1+\boldsymbol{\operatorname{curl}}\psi_1)\bigr|^2
-\rho_1\omega^2\bigl|\nabla\phi_1+\boldsymbol{\operatorname{curl}}\psi_1\bigr|^2
\Bigr)dx \notag\\
=&
\int_{\partial D}
\Bigl(
m_1\overline{q_1}+n_1\overline{r_1}
\Bigr)ds,
\label{betti_inside}
\end{align}
where $\varepsilon(\pmb w):=\frac12(\nabla\pmb w+(\nabla\pmb w)^\top)$.
Since the volume integral in \eqref{betti_inside} is real-valued, we have
\begin{equation}\label{imag part}
\Im\int_{\partial D}\bigl(m_1\overline{q_1}+n_1\overline{r_1}\bigr)ds=0.
\end{equation}

Let $B_R$ be a disk such that $\overline D\subset B_R$, and set
\(
\Omega_R:=B_R\setminus\overline D.
\)
Applying the same Betti formula in $\Omega_R$ to
$\nabla\phi_2+\boldsymbol{\operatorname{curl}}\psi_2$ yields
\begin{align}
&\int_{\Omega_R}
\Bigl(
\lambda_2\bigl|\nabla\cdot(\nabla\phi_2+\boldsymbol{\operatorname{curl}}\psi_2)\bigr|^2
+2\mu_2\bigl|\varepsilon(\nabla\phi_2+\boldsymbol{\operatorname{curl}}\psi_2)\bigr|^2
-\rho_2\omega^2\bigl|\nabla\phi_2+\boldsymbol{\operatorname{curl}}\psi_2\bigr|^2
\Bigr)dx \notag\\
=&
\int_{\partial B_R}\bigl(m_R\overline{q_R}+n_R\overline{r_R}\bigr)ds
-\int_{\partial D}\bigl(m_2\overline{q_2}+n_2\overline{r_2}\bigr)ds,
\label{betti_outside}
\end{align}
where on $\partial B_R$,
\[
q_R:=\partial_r\phi_2+\partial_{\pmb\tau}\psi_2,\qquad
r_R:=\partial_{\pmb\tau}\phi_2-\partial_r\psi_2,
\]
\[
m_R:=2\bigl(\mu_2\hat {\pmb x}\cdot\partial_r\nabla\phi_2
+\mu_2 \hat {\pmb x}\cdot\partial_r\boldsymbol{\operatorname{curl}}\psi_2\bigr)
-\lambda_2\kappa_p^2\phi_2,
\]
\[
n_R:=2\bigl(\mu_2\hat {\pmb x}^\perp\cdot\partial_r\nabla\phi_2
+\mu_2\hat {\pmb x}^\perp\cdot\partial_r\boldsymbol{\operatorname{curl}}\psi_2\bigr)
-\mu_2\kappa_s^2\psi_2,
\]
with $\hat {\pmb x}=\frac{\pmb x}{|\pmb x|}$ on $\partial B_R$.
Taking imaginary parts in \eqref{betti_outside} and using
\eqref{boundary condition bc} and \eqref{imag part}, we get
\begin{equation}\label{imag part BR}
\Im\int_{\partial B_R}\bigl(m_R\overline{q_R}+n_R\overline{r_R}\bigr)ds=0.
\end{equation}

Now we use the Sommerfeld radiation conditions for $\phi_2$ and $\psi_2$.
By the standard asymptotic expansions of outgoing solutions to the
two-dimensional Helmholtz equation,
\[
\phi_2(\pmb x)=\frac{e^{i\kappa_p r}}{r^{1/2}}
\bigl(\phi_\infty(\hat {\pmb x})+\mathcal O(r^{-1})\bigr),
\qquad
\psi_2(\pmb x)=\frac{e^{i\kappa_s r}}{r^{1/2}}
\bigl(\psi_\infty(\hat {\pmb x})+\mathcal O(r^{-1})\bigr),
\qquad r=\left|\pmb x \right|\to\infty,
\]
uniformly in $\hat {\pmb x}\in\mathbb S$. Hence, on $\partial B_R$,
\[
\partial_r\phi_2=i\kappa_p\phi_2+\mathcal O(R^{-3/2}),\qquad
\partial_r\psi_2=i\kappa_s\psi_2+\mathcal O(R^{-3/2}),
\]
\[
\partial_{\pmb\tau}\phi_2=\mathcal O(R^{-3/2}),\qquad
\partial_{\pmb\tau}\psi_2=\mathcal O(R^{-3/2}),
\]
\[
\partial_{rr}\phi_2=-\kappa_p^2\phi_2+\mathcal O(R^{-3/2}),\qquad
\partial_{rr}\psi_2=-\kappa_s^2\psi_2+\mathcal O(R^{-3/2}),
\]
\[
\partial_{r\pmb \tau}\phi_2=\mathcal O(R^{-3/2}),\qquad
\partial_{r\pmb \tau}\psi_2=\mathcal O(R^{-3/2}).
\]
Moreover, on $\partial B_R$,
\[
\hat {\pmb x}\cdot\partial_r\nabla\phi_2=\partial_{rr}\phi_2,\qquad
\hat {\pmb x}\cdot\partial_r\boldsymbol{\operatorname{curl}}\psi_2=\partial_{r\pmb\tau}\psi_2,
\]
\[
\hat {\pmb x}^\perp\cdot\partial_r\nabla\phi_2=\partial_{r\pmb\tau}\phi_2,\qquad
\hat {\pmb x}^\perp\cdot\partial_r\boldsymbol{\operatorname{curl}}\psi_2=-\partial_{rr}\psi_2.
\]
Therefore,
\[
q_R=i\kappa_p\phi_2+\mathcal O(R^{-3/2}),\qquad
r_R=-i\kappa_s\psi_2+\mathcal O(R^{-3/2}),
\]
\[
m_R=-(\lambda_2+2\mu_2)\kappa_p^2\phi_2+\mathcal O(R^{-3/2})
=-\rho_2\omega^2\phi_2+\mathcal O(R^{-3/2}),
\]
\[
n_R=\mu_2\kappa_s^2\psi_2+\mathcal O(R^{-3/2})
=\rho_2\omega^2\psi_2+\mathcal O(R^{-3/2}).
\]
Substituting these estimates into \eqref{imag part BR}, we obtain
\begin{equation}\label{farfield_identity}
\Im\int_{\partial B_R}\bigl(m_R\overline{q_R}+n_R\overline{r_R}\bigr)ds
=
\rho_2\omega^2
\int_{\partial B_R}\bigl(\kappa_p\left| \phi_2\right|^2+\kappa_s\left|\psi_2\right|^2\bigr)ds
+o(1)
\quad\text{as }R\to\infty.
\end{equation}
Combining \eqref{imag part BR} and \eqref{farfield_identity}, we get
\[
\lim_{R\to\infty}
\int_{\partial B_R}\bigl(\kappa_p\left|\phi_2\right|^2+\kappa_s\left|\psi_2\right|^2\bigr)ds=0.
\]
Since $\kappa_p>0$ and $\kappa_s>0$, it follows that
\[
\lim_{R\to\infty}\int_{\partial B_R}\left|\phi_2\right|^2ds=0,
\qquad
\lim_{R\to\infty}\int_{\partial B_R}\left|\psi_2\right|^2ds=0.
\]
By Rellich's lemma for the scalar Helmholtz equation together with the
Sommerfeld radiation conditions, we conclude that
\begin{equation*}
\phi_2=\psi_2=0\qquad\text{in }\mathbb R^2\setminus\overline D.
\end{equation*}

Hence the boundary conditions on $\partial D$ reduce to
\begin{equation}\label{reduced_bc_1}
\partial_{\pmb\nu}\phi_1+\partial_{\pmb\tau}\psi_1=0,\qquad
\partial_{\pmb\tau}\phi_1-\partial_{\pmb\nu}\psi_1=0,
\end{equation}
and
\begin{equation}\label{reduced_bc_2}
2\bigl(\mu_1\pmb\nu\cdot\partial_{\pmb\nu}\nabla\phi_1
+\mu_1\pmb\nu\cdot\partial_{\pmb\nu}\boldsymbol{\operatorname{curl}}\psi_1\bigr)
-\lambda_1\kappa_a^2\phi_1=0,
\end{equation}
\begin{equation}\label{reduced_bc_3}
2\bigl(\mu_1\pmb\tau\cdot\partial_{\pmb\nu}\nabla\phi_1
+\mu_1\pmb\tau\cdot\partial_{\pmb\nu}\boldsymbol{\operatorname{curl}}\psi_1\bigr)
-\mu_1\kappa_b^2\psi_1=0.
\end{equation}
From \eqref{reduced_bc_1}, we have
\begin{equation}\label{boundary condition zero}
\nabla\phi_1+\boldsymbol{\operatorname{curl}}\psi_1=0
\qquad\text{on }\partial D.
\end{equation}
Differentiating \eqref{boundary condition zero} tangentially gives
\[
\partial_{\pmb\tau}\bigl(\nabla\phi_1+\boldsymbol{\operatorname{curl}}\psi_1\bigr)=0
\qquad\text{on }\partial D.
\]
Using
\[
\nabla\cdot\bigl(\nabla\phi_1+\boldsymbol{\operatorname{curl}}\psi_1\bigr)
=
\pmb\nu\cdot\partial_{\pmb\nu}\bigl(\nabla\phi_1+\boldsymbol{\operatorname{curl}}\psi_1\bigr)
+\pmb\tau\cdot\partial_{\pmb\tau}\bigl(\nabla\phi_1+\boldsymbol{\operatorname{curl}}\psi_1\bigr),
\]
\[
\operatorname{curl}\bigl(\nabla\phi_1+\boldsymbol{\operatorname{curl}}\psi_1\bigr)
=
\pmb\tau\cdot\partial_{\pmb\nu}\bigl(\nabla\phi_1+\boldsymbol{\operatorname{curl}}\psi_1\bigr)
-\pmb\nu\cdot\partial_{\pmb\tau}\bigl(\nabla\phi_1+\boldsymbol{\operatorname{curl}}\psi_1\bigr),
\]
we have
\[
\pmb\nu\cdot\partial_{\pmb\nu}\bigl(\nabla\phi_1+\boldsymbol{\operatorname{curl}}\psi_1\bigr)
=
\nabla\cdot\bigl(\nabla\phi_1+\boldsymbol{\operatorname{curl}}\psi_1\bigr)
=
\Delta\phi_1
=
-\kappa_a^2\phi_1
\qquad\text{on }\partial D,
\]
and
\[
\pmb\tau\cdot\partial_{\pmb\nu}\bigl(\nabla\phi_1+\boldsymbol{\operatorname{curl}}\psi_1\bigr)
=
\operatorname{curl}\bigl(\nabla\phi_1+\boldsymbol{\operatorname{curl}}\psi_1\bigr)
=
-\Delta\psi_1
=
\kappa_b^2\psi_1
\qquad\text{on }\partial D.
\]
Substituting these two relations into \eqref{reduced_bc_2} and
\eqref{reduced_bc_3}, we obtain
\[
-(\lambda_1+2\mu_1)\kappa_a^2\phi_1=0,\qquad
\mu_1\kappa_b^2\psi_1=0
\qquad\text{on }\partial D.
\]
Hence,
\[
\phi_1=0,\qquad \psi_1=0 \qquad\text{on }\partial D.
\]
Therefore,
\[
\partial_{\pmb\tau}\phi_1=0,\qquad \partial_{\pmb\tau}\psi_1=0
\qquad\text{on }\partial D,
\]
and then \eqref{reduced_bc_1} yields
\[
\partial_{\pmb\nu}\phi_1=0,\qquad \partial_{\pmb\nu}\psi_1=0
\qquad\text{on }\partial D.
\]
Thus $\phi_1$ and $\psi_1$ satisfy Helmholtz equations in $D$ with vanishing
Cauchy data on $\partial D$. By Holmgren's theorem for the
Helmholtz equation \cite[Theorem 2.3]{Kress}, it follows that
\[
\phi_1=\psi_1=0 \qquad\text{in }D,
\]
which completes the proof.
\end{proof}

The scattered field $\pmb{v}$ admits the asymptotic expansion \cite{TArens}
\[
\pmb{v}(\pmb{x})
= \frac{  {\rm e}^{\mathrm{i}\kappa_{p}|\pmb{x}|}}{\sqrt{|\pmb{x}|}}
v_{p}^{\infty}(\hat{\pmb{x}})\hat{\pmb{x}}
+ \frac{  {\rm e}^{\mathrm{i}\kappa_{s}|\pmb{x}|}}{\sqrt{|\pmb{x}|}}
v_{s}^{\infty}(\hat{\pmb{x}})\hat{\pmb{x}}^{\perp}
+ \mathcal O\bigl(|\pmb{x}|^{-\tfrac{3}{2}}\bigr),\quad\hat{\pmb{x}}=\frac{\pmb{x}}{\left|\pmb{x}\right|},
\quad \left|\pmb{x}\right|\to\infty,
\]
defining the compressional and shear far-field patterns \(v_{p}^{\infty}\) and \(v_{s}^{\infty}\) 
on the unit circle $\mathbb{S}$. Furthermore, it follows from \cite[Theorem 3.1]{DongL} that
\begin{equation}\label{far}
	v_{p}^{\infty}(\hat{\pmb{x}}) = \mathrm{i}\kappa_{p}
	\phi_2^{\infty}(\hat{\pmb{x}}),
	\qquad
	v_{s}^{\infty}(\hat{\pmb{x}}) = -\mathrm{i}\kappa_{s}
	\psi_2^{\infty}(\hat{\pmb{x}}).
\end{equation}
Here, \(\phi_2^{\infty}\) and \(\psi_2^{\infty}\) denote the far-field patterns of \(\phi_2\) and \(\psi_2\), respectively.

This work addresses two fundamental problems in elastic scattering: the direct problem and the inverse problem associated with a penetrable obstacle. They are formally stated as follows
\begin{problem}
	\textbf{(Direct Problem)} Give the incident field $\pmb{u}$ and the information of penetrable elastic obstacle $\partial D$, determine the compressional and shear far-field patterns \(v_{p}^{\infty}\) and \(v_{s}^{\infty}\) on the unit circle \(\mathbb{S}\).
\end{problem}
\begin{problem}
	\textbf{(Inverse Problem)} Given the incident field \(\pmb{u}\) and the far-field patterns \(v_{p}^{\infty}\) and \(v_{s}^{\infty}\) measured on \(\mathbb{S}\), reconstruct the location and shape of the penetrable elastic obstacle \(D\).
\end{problem}

\section{The direct scattering problem}
This section develops the boundary integral equations for the coupled problem \eqref{2.5} and presents a Nystr\"{o}m-type discretization method. The last two conditions in \eqref{2.5} involve second-order boundary traces after the Helmholtz decomposition. Consequently, the resulting boundary integral formulation is a coupled four-density system and relies essentially on the jump relations for the second derivatives of the single-layer potential. For the evaluation of the singular integrals arising in this system, we follow the approaches in \cite{Drule} and \cite{Kress}, which successfully implement Nystr\"{o}m-type methods for acoustic and elastic scattering problems involving hypersingular integral equations.

\subsection{Boundary Integral Equations}

The single-layer integral operator and the corresponding far-field integral operator are defined by
\begin{equation*}
	\begin{aligned}
		S_{\kappa}[g](\pmb{x})
		&=\int_{\partial D}\varPhi(\pmb{x},\pmb{y};\kappa)g(\pmb{y})\mathrm{d}s(\pmb{y}), 
		\quad \pmb{x}\in\partial D,\\
		S_{\kappa}^{\infty}[g](\hat{\pmb{x}})
		&=\gamma_{\kappa}
		\int_{\partial D}{\rm e}^{-\mathrm{i}\kappa\hat{\pmb{x}}\cdot\pmb{y}}
		g(\pmb{y})\mathrm{d}s(\pmb{y}), 
		\quad \hat{\pmb{x}}\in\mathbb{S},
	\end{aligned}
\end{equation*}
where 
$
\gamma_{\kappa} = {\rm e}^{\mathrm{i}\pi/4}/\sqrt{8\pi\kappa}, 
$ and 
$
\varPhi(\pmb{x},\pmb{y};\kappa)
=\frac{\mathrm{i}}{4}H_{0}^{(1)}  \bigl(\kappa\lvert\pmb{x}-\pmb{y}\rvert\bigr),
~ \pmb{x}\neq\pmb{y},
$
is the fundamental solution of the two-dimensional Helmholtz equation. Here, \(H_{0}^{(1)}\) denotes the Hankel function of the first kind with order zero. 
We note that the quantity inside the bracket \([\cdot]\) is associated with the variable $\pmb{y}$, while the resulting function depends on the variable outside the bracket. This notational convention will be consistently employed throughout this paper.

Additionally, the boundary integral operators associated with the normal and tangential derivatives of the single-layer potential are defined as
\begin{align*}
	N_{\kappa}[g](\pmb{x})
	&=\int_{\partial D}\frac{\partial}{\partial\pmb{\nu}(\pmb{x})}
	\varPhi(\pmb{x},\pmb{y};\kappa)g(\pmb{y})\mathrm{d}s(\pmb{y}),
	\quad
	\pmb{x}\in\partial D,\\
	T_{\kappa}[g](\pmb{x})
	&=\int_{\partial D}
	\frac{\partial}{\partial\pmb{\tau}(\pmb{x})}
	\varPhi(\pmb{x},\pmb{y};\kappa)g(\pmb{y})\mathrm{d}s(\pmb{y}),
	\quad
	\pmb{x}\in\partial D.
\end{align*}

Assume that the solutions to the system \eqref{2.5} admit the following single-layer potential representations

	\begin{align}\label{single1}
		&\phi_1(\pmb{x})
		= \int_{\partial D} \varPhi(\pmb{x},\pmb{y};\kappa_{a})g_1(\pmb{y})\mathrm{d}s(\pmb{y}), 
		~~
		\psi_1(\pmb{x})
		= \int_{\partial D} \varPhi(\pmb{x},\pmb{y};\kappa_{b})g_2(\pmb{y})\mathrm{d}s(\pmb{y}),
		 ~~\pmb{x}\in D,\\
		\label{single2}
		&\phi_2(\pmb{x})
		= \int_{\partial D} \varPhi(\pmb{x},\pmb{y};\kappa_{p})g_3(\pmb{y})\mathrm{d}s(\pmb{y}), 
		~~
		\psi_2(\pmb{x})
		= \int_{\partial D} \varPhi(\pmb{x},\pmb{y};\kappa_{s})g_4(\pmb{y})\mathrm{d}s(\pmb{y}),
		~~ \pmb{x}\in\mathbb{R}^2\setminus\overline{D},
\end{align}where the densities \(g_i\in C^{1,\alpha}(\partial D)\) for \(i=1,\cdots,4\) and $0<\alpha<1$.

By taking the limit in \eqref{single1} as $\pmb{x}$ approaches $\partial D$ from inside $D$, and in \eqref{single2} as $\pmb{x}$ approaches $\partial D$ from $\mathbb{R}^2 \setminus \overline{D}$, and then incorporating the jump relations \cite[Corollary 3.5]{Djump} with the transmission conditions \eqref{bc}, we derive the following system of boundary integral equations on \(\partial D\)
\begin{align}
	f_1 &= g_1/2 + N_{\kappa_{a}}[g_1] + T_{\kappa_{b}}[g_2] + g_3/2 - N_{\kappa_{p}}[g_3] - T_{\kappa_{s}}[g_4], \label{3.3a}\\
	f_2 &= T_{\kappa_{a}}[g_1] - g_2/2 - N_{\kappa_{b}}[g_2] - T_{\kappa_{p}}[g_3] - g_4/2 + N_{\kappa_{s}}[g_4], \label{3.3b}\\
	f_3 &= 2\bigl( -\mu_1\kappa_{a}^2\pmb{\nu}^\top S_{\kappa_{a}}[\pmb{\nu}\pmb{\nu}^\top g_1]\pmb{\nu}
	+ \mu_1\pmb{\nu}^\top N_{\kappa_{a}}[\pmb{\tau}\partial_{\pmb{\tau}} g_1 + g_1\partial_{\pmb{\tau}}\pmb{\tau}]
	- \mu_1\pmb{\nu}^\top T_{\kappa_{a}}[\pmb{\nu}\partial_{\pmb{\tau}} g_1 + g_1\partial_{\pmb{\tau}}\pmb{\nu}] \notag\\
	&\quad + \mu_1\kappa_{b}^2\pmb{\nu}^\top S_{\kappa_{b}}[\pmb{\tau}\pmb{\nu}^\top g_2]\pmb{\nu}
	+ \mu_1\pmb{\nu}^\top N_{\kappa_{b}}[\pmb{\nu}\partial_{\pmb{\tau}} g_2 + g_2\partial_{\pmb{\tau}}\pmb{\nu}]
	+ \mu_1\pmb{\nu}^\top T_{\kappa_{b}}[\pmb{\tau}\partial_{\pmb{\tau}} g_2 + g_2\partial_{\pmb{\tau}}\pmb{\tau}] \notag\\
	&\quad + \mu_2\kappa_{p}^2\pmb{\nu}^\top S_{\kappa_{p}}[\pmb{\nu}\pmb{\nu}^\top g_3]\pmb{\nu}
	- \mu_2\pmb{\nu}^\top N_{\kappa_{p}}[\pmb{\tau}\partial_{\pmb{\tau}} g_3 + g_3\partial_{\pmb{\tau}}\pmb{\tau}]
	+ \mu_2\pmb{\nu}^\top T_{\kappa_{p}}[\pmb{\nu}\partial_{\pmb{\tau}} g_3 + g_3\partial_{\pmb{\tau}}\pmb{\nu}] \notag\\
	&\quad - \mu_2\kappa_{s}^2\pmb{\nu}^\top S_{\kappa_{s}}[\pmb{\tau}\pmb{\nu}^\top g_4]\pmb{\nu}
	- \mu_2\pmb{\nu}^\top N_{\kappa_{s}}[\pmb{\nu}\partial_{\pmb{\tau}} g_4 + g_4\partial_{\pmb{\tau}}\pmb{\nu}]
	- \mu_2\pmb{\nu}^\top T_{\kappa_{s}}[\pmb{\tau}\partial_{\pmb{\tau}} g_4 + g_4\partial_{\pmb{\tau}}\pmb{\tau}] \bigr) \notag\\
	&\quad - \lambda_1\kappa_{a}^2 S_{\kappa_{a}}[g_1] + \lambda_2\kappa_{p}^2 S_{\kappa_{p}}[g_3]
	+ \mu_1(\pmb{\nu} \cdot \partial_{\pmb{\tau}}\pmb{\tau}) g_1 + \mu_1(\pmb{\nu} \cdot \partial_{\pmb{\tau}}\pmb{\nu}) g_2 + \mu_1\partial_{\pmb{\tau}} g_2 \notag\\
	&\quad + \mu_2(\pmb{\nu} \cdot \partial_{\pmb{\tau}}\pmb{\tau}) g_3 + \mu_2(\pmb{\nu} \cdot \partial_{\pmb{\tau}}\pmb{\nu}) g_4 + \mu_2\partial_{\pmb{\tau}} g_4, \label{3.3c}\\
	f_4 &= 2\bigl( -\mu_1\kappa_{a}^2\pmb{\tau}^\top S_{\kappa_{a}}[\pmb{\nu}\pmb{\nu}^\top g_1]\pmb{\nu}
	+ \mu_1\pmb{\tau}^\top N_{\kappa_{a}}[\pmb{\tau}\partial_{\pmb{\tau}} g_1 + g_1\partial_{\pmb{\tau}}\pmb{\tau}]
	- \mu_1\pmb{\tau}^\top T_{\kappa_{a}}[\pmb{\nu}\partial_{\pmb{\tau}} g_1 + g_1\partial_{\pmb{\tau}}\pmb{\nu}] \notag\\
	&\quad + \mu_1\kappa_{b}^2\pmb{\tau}^\top S_{\kappa_{b}}[\pmb{\tau}\pmb{\nu}^\top g_2]\pmb{\nu}
	+ \mu_1\pmb{\tau}^\top N_{\kappa_{b}}[\pmb{\nu}\partial_{\pmb{\tau}} g_2 + g_2\partial_{\pmb{\tau}}\pmb{\nu}]
	+ \mu_1\pmb{\tau}^\top T_{\kappa_{b}}[\pmb{\tau}\partial_{\pmb{\tau}} g_2 + g_2\partial_{\pmb{\tau}}\pmb{\tau}] \notag\\
	&\quad + \mu_2\kappa_{p}^2\pmb{\tau}^\top S_{\kappa_{p}}[\pmb{\nu}\pmb{\nu}^\top g_3]\pmb{\nu}
	- \mu_2\pmb{\tau}^\top N_{\kappa_{p}}[\pmb{\tau}\partial_{\pmb{\tau}} g_3 + g_3\partial_{\pmb{\tau}}\pmb{\tau}]
	+ \mu_2\pmb{\tau}^\top T_{\kappa_{p}}[\pmb{\nu}\partial_{\pmb{\tau}} g_3 + g_3\partial_{\pmb{\tau}}\pmb{\nu}] \notag\\
	&\quad - \mu_2\kappa_{s}^2\pmb{\tau}^\top S_{\kappa_{s}}[\pmb{\tau}\pmb{\nu}^\top g_4]\pmb{\nu}
	- \mu_2\pmb{\tau}^\top N_{\kappa_{s}}[\pmb{\nu}\partial_{\pmb{\tau}} g_4 + g_4\partial_{\pmb{\tau}}\pmb{\nu}]
	- \mu_2\pmb{\tau}^\top T_{\kappa_{s}}[\pmb{\tau}\partial_{\pmb{\tau}} g_4 + g_4\partial_{\pmb{\tau}}\pmb{\tau}] \bigr) \notag\\
	&\quad - \mu_1\kappa_{b}^2 S_{\kappa_{b}}[g_2] + \mu_2\kappa_{s}^2 S_{\kappa_{s}}[g_4]
	+ \mu_1(\pmb{\tau} \cdot \partial_{\pmb{\tau}}\pmb{\tau}) g_1 + \mu_1\partial_{\pmb{\tau}} g_1 + \mu_1(\pmb{\tau} \cdot \partial_{\pmb{\tau}}\pmb{\nu}) g_2 \notag\\
	&\quad + \mu_2(\pmb{\tau} \cdot \partial_{\pmb{\tau}}\pmb{\tau}) g_3 + \mu_2\partial_{\pmb{\tau}} g_3 + \mu_2(\pmb{\tau} \cdot \partial_{\pmb{\tau}}\pmb{\nu}) g_4. \label{3.3d}
\end{align}

\begin{theorem}
	Assume that $\kappa_p$ and $\kappa_s$ are not interior Dirichlet eigenvalues for the Helmholtz equation in $D$. Then the system of boundary integral equations \eqref{3.3a}--\eqref{3.3d} has at most one solution in $C^{1,\alpha}(\partial D)^4$.
\end{theorem}

\begin{proof}
	It suffices to show $g_1=g_2=g_3=g_4=0$ in $ C^{1,\alpha}(\partial D)$ when $f_1=f_2=f_3=f_4=0$.
	 Define the potentials by the single-layer representations \eqref{single1}--\eqref{single2}.
	By construction, the jump relations imply that these potentials satisfy the homogeneous coupled boundary value problem \eqref{2.5}. Hence, by Theorem~\ref{uniqueness of coupled system},
	\[
	\phi_1=\psi_1=0\quad\text{in }~D,
	\qquad
	\phi_2=\psi_2=0\quad\text{in }~\mathbb{R}^2\setminus\overline D.
	\]
	
    Define the global single-layer fields
	\[
	W_1(\pmb{x})=\int_{\partial D}\varPhi(\pmb{x},\pmb{y};\kappa_a)g_1(\pmb{y})ds(\pmb{y}),
	\qquad
	W_2(\pmb{x})=\int_{\partial D}\varPhi(\pmb{x},\pmb{y};\kappa_b)g_2(\pmb{y})ds(\pmb{y}),\quad  \pmb{x}\in\mathbb{R}^2\setminus\partial D.
	\]
	Their interior restrictions are precisely $\phi_1$ and $\psi_1$, so
	\[
	W_1^-=0,
	\qquad
	W_2^-=0
	\qquad\text{in }~D.
	\]
	Since the single layer potential is continuous on $\partial D$, we have
	\[
	W_1=0,
	\qquad
	W_2=0
	\qquad\text{on }\partial D.
	\]
	Now $W_1^+$ and $W_2^+$ are radiating exterior Helmholtz solutions with zero Dirichlet boundary data. By uniqueness of the exterior Helmholtz Dirichlet problem,
	\[
	W_1^+=0,
	\qquad
	W_2^+=0
	\qquad\text{in }~\mathbb{R}^2\setminus\overline D.
	\]
	Therefore their normal derivatives vanish on both sides of $\partial D$, and the jump relation for the normal derivative of the single layer potential yields
	\[
	\partial_{\pmb{\nu}}^-W_1-\partial_{\pmb{\nu}}^+W_1=g_1,
	\qquad
	\partial_{\pmb{\nu}}^-W_2-\partial_{\pmb{\nu}}^+W_2=g_2.
	\]
	Hence
	\[
	g_1=0,
	\qquad
	g_2=0\qquad \text{on }~\partial D.
	\]
	
	Similarly, define
	\[
	W_3(\pmb{x})=\int_{\partial D}\varPhi(\pmb{x},\pmb{y};\kappa_p)g_3(\pmb{y})ds(\pmb{y}),
	\qquad
	W_4(\pmb{x})=\int_{\partial D}\varPhi(\pmb{x},\pmb{y};\kappa_s)g_4(\pmb{y})ds(\pmb{y}).
	\]
	Their exterior restrictions are $\phi_2$ and $\psi_2$, hence
	\[
	W_3^+=0,
	\qquad
	W_4^+=0
	\qquad\text{in }~\mathbb{R}^2\setminus\overline D.
	\]
	By continuity of the single layer potential,
	\[
	W_3=0,
	\qquad
	W_4=0
	\qquad\text{on }~\partial D.
	\]
	Since $\kappa_p$ and $\kappa_s$ are not interior Dirichlet eigenvalues for the Helmholtz equation, the corresponding interior Dirichlet problems are uniquely solvable. Therefore,
	\[
	W_3^-=0,
	\qquad
	W_4^-=0
	\qquad\text{in }~D.
	\]
	Again the normal derivatives vanish on both sides, so the jump relation gives
	\[
	g_3=0,
	\qquad
	g_4=0\qquad \text{on }~\partial D.
	\]
	Thus $g_1=g_2=g_3=g_4=0$, and the proof is complete.
\end{proof}

\begin{remark}
The above boundary integral formulation can be extended to multi-obstacle configurations with appropriate boundary or transmission conditions imposed on each component. Due to multiple scattering among the obstacles, the total far-field pattern cannot be obtained by simply adding the far-field patterns of the corresponding isolated obstacles. As in the single-obstacle case, one is led to a coupled system of boundary integral equations for the boundary unknowns on all components, which captures the interactions among the obstacles. As a representative example, Appendix \ref{appendix a} presents the coupled boundary integral equations for a two-obstacle configuration consisting of one penetrable obstacle and one impenetrable obstacle subject to the Dirichlet boundary condition.
\end{remark}

\subsection{Parametrization}
We employ a parametric representation of the boundary \(\partial D\) as an analytic curve
\[
\partial D : =  \{\pmb p(\xi)=(p_1(\xi),p_2(\xi))^\top
: \xi\in [0,2\pi) \}.
\]
The boundary integral operators \(S_\kappa\), \(S_\kappa^\infty\), \(N_\kappa\) and \(T_\kappa\) are expressed in paramet-rized form as
\begin{align*}
	(S_\kappa[J_{\pmb{p}}\varrho])(\xi) 
	&= \int_{0}^{2\pi} \widetilde {S}(\xi,\varsigma;\kappa)
	J_{\pmb{p}}(\varsigma)\varrho(\varsigma)\mathrm{d}\varsigma,\\
	(S_\kappa^\infty[J_{\pmb{p}}\varrho])(\xi)
	&= \gamma_\kappa
	\int_{0}^{2\pi}
	{\rm e}^{-\mathrm{i}\kappa\pmb {\hat x}(\xi)  \cdot  \pmb p(\varsigma)}
	J_{\pmb{p}}(\varsigma)\varrho(\varsigma)\mathrm{d}\varsigma,\\
	(N_\kappa[J_{\pmb{p}}\varrho])(\xi)
	&= \frac{1}{J_{\pmb{p}}(\xi)}
	\int_{0}^{2\pi} \widetilde N(\xi,\varsigma;\kappa)
	J_{\pmb{p}}(\varsigma)\varrho(\varsigma)\mathrm{d}\varsigma,\\
	(T_\kappa[J_{\pmb{p}}\varrho])(\xi)
	&= \frac{1}{J_{\pmb{p}}(\xi)}
	\int_{0}^{2\pi} \widetilde T(\xi,\varsigma;\kappa)
	J_{\pmb{p}}(\varsigma)\varrho(\varsigma)\mathrm{d}\varsigma,
\end{align*}
where 
\(
\varrho=(g\circ \pmb{p}), ~
J_{\pmb{p}}(\varsigma)=\lvert \pmb{p}'(\varsigma)\rvert
\)
is the Jacobian of the transformation, and the kernel functions are given by
\begin{align*}
	\widetilde {S}(\xi,\varsigma;\kappa)
	&= \frac{\mathrm{i}}{4}
	H_{0}^{(1)}  \bigl(\kappa\lvert \pmb p(\xi)-\pmb p(\varsigma)\rvert\bigr),\\
	\widetilde N(\xi,\varsigma;\kappa)
	&= \frac{\mathrm{i}\kappa}{4}
	\pmb{n}(\xi)  \cdot  \bigl(\pmb p(\varsigma)-\pmb p(\xi)\bigr)
	\frac{H_{1}^{(1)}  \bigl(\kappa\lvert \pmb p(\xi)-\pmb p(\varsigma)\rvert\bigr)}
	{\lvert \pmb p(\xi)-\pmb p(\varsigma)\rvert},\\
	\widetilde T(\xi,\varsigma;\kappa)
	&= \frac{\mathrm{i}\kappa}{4}
	\pmb{n}(\xi)^{\perp}  \cdot  \bigl(\pmb p(\varsigma)-\pmb p(\xi)\bigr)
	\frac{H_{1}^{(1)}  \bigl(\kappa\lvert \pmb p(\xi)-\pmb p(\varsigma)\rvert\bigr)}
	{\lvert \pmb p(\xi)-\pmb p(\varsigma)\rvert},
\end{align*}
with $H_{1}^{(1)}$ denoting the Hankel function of the first kind of order one. Furthermore,
\[
\pmb{n}(\xi)
= \tilde{\pmb{\nu}}(\xi)\lvert \pmb p'(\xi)\rvert
= \bigl(p'_2(\xi),-p'_1(\xi)\bigr)^\top,
\quad
\pmb{n}(\xi)^\perp
= \tilde{\pmb{\tau}}(\xi)\lvert \pmb p'(\xi)\rvert
= \bigl(p'_1(\xi),p'_2(\xi)\bigr)^\top,
\]
and \(\tilde{\pmb{\nu}}=\pmb {\nu}\circ \pmb p\), \(\tilde{\pmb{\tau}}=\pmb{\tau}\circ \pmb p\).

Consequently, the boundary integral equations \eqref{3.3a}-\eqref{3.3d} can be reformulated in parametrized form as follows
\begin{align}
	w_1=&\tilde{\varrho_1}J_{\pmb{p}}/2+N_{\kappa_{a}}[\tilde{\varrho_1}J_{\pmb{p}}^2]+T_{\kappa_{b}}[\tilde{\varrho_2}J_{\pmb{p}}^2]+\tilde{\varrho_3}J_{\pmb{p}}/2-N_{\kappa_{p}}[\tilde{\varrho_3}J_{\pmb{p}}^2]-T_{\kappa_{s}}[\tilde{\varrho_4}J_{\pmb{p}}^2], \label{par1}\\
	w_2=&T_{\kappa_{a}}[\tilde{\varrho_1}J_{\pmb{p}}^2]-\tilde{\varrho_2}J_{\pmb{p}}/2-N_{\kappa_{b}}[\tilde{\varrho_2}J_{\pmb{p}}^2]-T_{\kappa_{p}}[\tilde{\varrho_3}J_{\pmb{p}}^2]-\tilde{\varrho_4}J_{\pmb{p}}/2+N_{\kappa_{s}}[\tilde{\varrho_4}J_{\pmb{p}}^2],\label{par2}\\
	w_3=&2(-\mu_1\kappa_{a}^2\tilde{\pmb{\nu}}^{\top}S_{\kappa_{a}}[\pmb{n}\pmb{n}^{\top}\tilde{\varrho_1}]\tilde{\pmb{\nu}}+\mu_1\tilde{\pmb{\nu}}^{\top}N_{\kappa_{a}}[\pmb{n}^{\perp}\tilde{\varrho_1}'+\pmb{n}^{\perp '}\tilde{\varrho_1}]-\mu_1\tilde{\pmb{\nu}}^{\top}T_{\kappa_{a}}[\pmb{n}\tilde{\varrho_1}'+\pmb{n}'\tilde{\varrho_1}]\notag\\
	&+\mu_1\kappa_{b}^2\tilde{\pmb{\nu}}^{\top}S_{\kappa_{b}}[\pmb{n}^{\perp}\pmb{n}^{\top}\tilde{\varrho_2}]\tilde{\pmb{\nu}}+\mu_1\tilde{\pmb{\nu}}^{\top}N_{\kappa_{b}}[\pmb{n}\tilde{\varrho_2}'+\pmb{n}'\tilde{\varrho_2}]+\mu_1\tilde{\pmb{\nu}}^{\top}T_{\kappa_{b}}[\pmb{n}^{\perp}\tilde{\varrho_2}'+\pmb{n}^{\perp '}\tilde{\varrho_2}]\notag\\
	&+\mu_2\kappa_{p}^2\tilde{\pmb{\nu}}^{\top}S_{\kappa_{p}}[\pmb{n}\pmb{n}^{\top}\tilde{\varrho_3}]\tilde{\pmb{\nu}}-\mu_2\tilde{\pmb{\nu}}^{\top}N_{\kappa_{p}}[\pmb{n}^{\perp}\tilde{\varrho_3}'+\pmb{n}^{\perp '}\tilde{\varrho_3}]+\mu_2\tilde{\pmb{\nu}}^{\top}T_{\kappa_{p}}[\pmb{n}\tilde{\varrho_3}'+\pmb{n}'\tilde{\varrho_3}]\notag\\
	&-\mu_2\kappa_{s}^2\tilde{\pmb{\nu}}^{\top}S_{\kappa_{s}}[\pmb{n}^{\perp}\pmb{n}^{\top}\tilde{\varrho_4}]\tilde{\pmb{\nu}}-\mu_2\tilde{\pmb{\nu}}^{\top}N_{\kappa_{s}}[\pmb{n}\tilde{\varrho_4}'+\pmb{n}'\tilde{\varrho_4}]-\mu_2\tilde{\pmb{\nu}}^{\top}T_{\kappa_{s}}[\pmb{n}^{\perp}\tilde{\varrho_4}'+\pmb{n}^{\perp '}\tilde{\varrho_4}])\notag\\		
	&-\lambda_1\kappa_{a}^2S_{\kappa_{a}}[\tilde{\varrho_1}J_{\pmb{p}}^2]+\lambda_2\kappa_{p}^2S_{\kappa_{p}}[\tilde{\varrho_3}J_{\pmb{p}}^2]+\mu_1(\tilde{\pmb{\nu}}\cdot \pmb{n}^{\perp '})\tilde{\varrho_1}/J_{\pmb{p}}+\mu_1(\tilde{\pmb{\nu}}\cdot \pmb{n}')\tilde{\varrho_2}/J_{\pmb{p}}\notag\\
	&+\mu_1\tilde{\varrho_2}'+\mu_2(\tilde{\pmb{\nu}}\cdot \pmb{n}^{\perp '})\tilde{\varrho_3}/J_{\pmb{p}}+\mu_2(\tilde{\pmb{\nu}}\cdot \pmb{n}')\tilde{\varrho_4}/J_{\pmb{p}}+\mu_2\tilde{\varrho_4}',\label{par3}
\\
	w_4=&2(-\mu_1\kappa_{a}^2\tilde{\pmb{\tau}}^{\top}S_{\kappa_{a}}[\pmb{n}\pmb{n}^{\top}\tilde{\varrho_1}]\tilde{\pmb{\nu}}+\mu_1\tilde{\pmb{\tau}}^{\top}N_{\kappa_{a}}[\pmb{n}^{\perp}\tilde{\varrho_1}'+\pmb{n}^{\perp '}\tilde{\varrho_1}]-\mu_1\tilde{\pmb{\tau}}^{\top}T_{\kappa_{a}}[\pmb{n}\tilde{\varrho_1}'+\pmb{n}'\tilde{\varrho_1}]\notag\\
	&+\mu_1\kappa_{b}^2\tilde{\pmb{\tau}}^{\top}S_{\kappa_{b}}[\pmb{n}^{\perp}\pmb{n}^{\top}\tilde{\varrho_2}]\tilde{\pmb{\nu}}+\mu_1\tilde{\pmb{\tau}}^{\top}N_{\kappa_{b}}[\pmb{n}\tilde{\varrho_2}'+\pmb{n}'\tilde{\varrho_2}]+\mu_1\tilde{\pmb{\tau}}^{\top}T_{\kappa_{b}}[\pmb{n}^{\perp}\tilde{\varrho_2}'+\pmb{n}^{\perp '}\tilde{\varrho_2}]\notag\\
	&+\mu_2\kappa_{p}^2\tilde{\pmb{\tau}}^{\top}S_{\kappa_{p}}[\pmb{n}\pmb{n}^{\top}\tilde{\varrho_3}]\tilde{\pmb{\nu}}-\mu_2\tilde{\pmb{\tau}}^{\top}N_{\kappa_{p}}[\pmb{n}^{\perp}\tilde{\varrho_3}'+\pmb{n}^{\perp '}\tilde{\varrho_3}]+\mu_2\tilde{\pmb{\tau}}^{\top}T_{\kappa_{p}}[\pmb{n}\tilde{\varrho_3}'+\pmb{n}'\tilde{\varrho_3}]\notag\\
	&-\mu_2\kappa_{s}^2\tilde{\pmb{\tau}}^{\top}S_{\kappa_{s}}[\pmb{n}^{\perp}\pmb{n}^{\top}\tilde{\varrho_4}]\tilde{\pmb{\nu}}-\mu_2\tilde{\pmb{\tau}}^{\top}N_{\kappa_{s}}[\pmb{n}\tilde{\varrho_4}'+\pmb{n}'\tilde{\varrho_4}]-\mu_2\tilde{\pmb{\tau}}^{\top}T_{\kappa_{s}}[\pmb{n}^{\perp}\tilde{\varrho_4}'+\pmb{n}^{\perp '}\tilde{\varrho_4}])\notag\\
	&-\mu_1\kappa_{b}^2S_{\kappa_{b}}[\tilde{\varrho_2}J_{\pmb{p}}^2]+\mu_2\kappa_{s}^2S_{\kappa_{s}}[\tilde{\varrho_4}J_{\pmb{p}}^2]+\mu_1(\tilde{\pmb{\tau}}\cdot \pmb{n}^{\perp '})\tilde{\varrho_1}/J_{\pmb{p}}+\mu_1\tilde{\varrho_1}'\notag\\
	&+\mu_1(\tilde{\pmb{\tau}}\cdot \pmb{n}')\tilde{\varrho_2}/J_{\pmb{p}}+\mu_2(\tilde{\pmb{\tau}}\cdot \pmb{n}^{\perp '})\tilde{\varrho_3}/J_{\pmb{p}}+\mu_2\tilde{\varrho_3}'+\mu_2(\tilde{\pmb{\tau}}\cdot \pmb{n}')\tilde{\varrho_4}/J_{\pmb{p}}.\label{par4}
\end{align}
where $w_{j}=(f_{j}\circ \pmb p)$, $\tilde{\varrho_{j}}=\varrho_{j}/J_{\pmb{p}}$, $\tilde{\varrho_{j}}'=(\varrho_{j}/J_{\pmb{p}})'$, $j=1,\ldots,4$ and $\pmb{n}'=(p''_2, -p''_1)^{\top}$, $\pmb{n}^{\perp '}=(p''_1,p''_2)^{\top}$.

A fully discretized scheme for equation \eqref{par1}-\eqref{par4}, which employs the Nystr\"{o}m-type method, is presented in the Appendix \ref{appendix b}.

\subsection{Numerical results}

To validate the effectiveness of the boundary integral equation method, we construct the following exact solutions
\begin{align*}
	\phi_{1}^{*} &= e^{\mathrm{i}\kappa_a\pmb{x}\cdot\pmb{d}},         & \psi_{1}^{*} &= e^{\mathrm{i}\kappa_b\pmb{x}\cdot\pmb{d}},         & \pmb{x} &\in D, \\
	\phi_{2}^{*} &= \frac{\mathrm{i}}{4}H_{0}^{(1)}\bigl(\kappa_p\lvert\pmb{x}-\pmb{y} \rvert\bigr), & \psi_{2}^{*} &= \frac{\mathrm{i}}{4}H_{0}^{(1)}\bigl(\kappa_s\lvert\pmb{x}-\pmb{y} \rvert\bigr), & \pmb{x} &\in \mathbb{R}^2\setminus\overline{D},
\end{align*}
which satisfy the coupled boundary value problem \eqref{2.5}. Here, $\pmb{d}=(1,0)^{\top}\in\mathbb{S}$ and $\pmb{y}=(0.2,0)^{\top}\in D$. The corresponding exact far-field patterns of $\phi_{2}^{*}$ and $\psi_{2}^{*}$ are given by
$$\phi_{2}^{\infty,*}=\gamma_{\kappa_p}e^{-\mathrm{i}\kappa_p\hat{\pmb{x}}\cdot\pmb{y}},\quad\psi_{2}^{\infty,*}=\gamma_{\kappa_s}e^{-\mathrm{i}\kappa_s\hat{\pmb{x}}\cdot\pmb{y}},\quad \hat{\pmb{x}}\in\mathbb{S}.
$$

The numerical solution is obtained by imposing the following boundary conditions on \(\partial D\)
\[
\begin{cases}
	\partial_{\pmb{\nu}}\phi_{1}^{*} + \partial_{\pmb{\tau}}\psi_{1}^{*}
	- \partial_{\pmb{\nu}}\phi_{2}^{*} - \partial_{\pmb{\tau}}\psi_{2}^{*}
	= f_{1},\\[3pt]
	\partial_{\pmb{\tau}}\phi_{1}^{*} - \partial_{\pmb{\nu}}\psi_{1}^{*}
	- \partial_{\pmb{\tau}}\phi_{2}^{*} + \partial_{\pmb{\nu}}\psi_{2}^{*}
	= f_{2},\\[3pt]
	2\bigl(\mu_{1}\pmb{\nu}  \cdot  \partial_{\pmb{\nu}}\nabla\phi_{1}^{*}
	+\mu_{1}\pmb{\nu}  \cdot  \partial_{\pmb{\nu}}\pmb{\operatorname{curl}}\psi_{1}^{*}\bigr)
	- \lambda_{1}\kappa_{a}^{2}\phi_{1}^{*}\\
	\qquad-2\bigl(\mu_{2}\pmb{\nu}  \cdot  \partial_{\pmb{\nu}}\nabla\phi_{2}^{*}
	+\mu_{2}\pmb{\nu}  \cdot  \partial_{\pmb{\nu}}\pmb{\operatorname{curl}}\psi_{2}^{*}\bigr)
	+ \lambda_{2}\kappa_{p}^{2}\phi_{2}^{*}
	= f_{3},\\[3pt]
	2\bigl(\mu_{1}\pmb{\tau}  \cdot  \partial_{\pmb{\nu}}\nabla\phi_{1}^{*}
	+\mu_{1}\pmb{\tau}  \cdot  \partial_{\pmb{\nu}}\pmb{\operatorname{curl}}\psi_{1}^{*}\bigr)
	- \mu_{1}\kappa_{b}^{2}\psi_{1}^{*}\\
	\qquad-2\bigl(\mu_{2}\pmb{\tau}  \cdot  \partial_{\pmb{\nu}}\nabla\phi_{2}^{*}
	+\mu_{2}\pmb{\tau}  \cdot  \partial_{\pmb{\nu}}\pmb{\operatorname{curl}}\psi_{2}^{*}\bigr)
	+ \mu_{2}\kappa_{s}^{2}\psi_{2}^{*}
	= f_{4}.
\end{cases}
\]
The computed values of \(f_{1}\), \(f_{2}\), \(f_{3}\), and \(f_{4}\) are then substituted into the original problem \eqref{2.5}. The Lam\'{e} constants are set to \(\lambda_1 = 5.5\), \(\mu_1 = 2.6\), \(\lambda_2 = 1.5\), and \(\mu_2 = 1\), with mass densities \(\rho_1 = 1\) for \(D\) and \(\rho_2 = 1.5\) for the exterior region \(\mathbb{R}^2 \setminus \overline{D}\). The number of observation points is set to 32 on both the measurement curve \(\partial B = \{\pmb{x} \in \mathbb{R}^2 : |\pmb{x}| = 0.2\}\) and the unit circle \(\mathbb{S}\). This specific set of Lam\'{e} constants and mass densities will also be used in the inverse scattering experiments in subsequent sections. All numerical tests are implemented by using Matlab 2025b on a personal laptop with a 48GB RAM, 4.50GHz Apple M4pro processor.

For the penetrable pear-shaped obstacle, we compute the relative errors
\begin{align*}
	E^{\mathrm{in}}(\varphi)=\frac{\|\varphi^{*}-\varphi^{(n)}\|_{L^2(\partial B)}}{\|\varphi^{*}\|_{L^2(\partial B)}},\quad
	E^{\infty}(\varphi)=\frac{\|\varphi^{\infty,*} - \varphi^{\infty,(n)}\|_{L^2(\mathbb{S})}}{\|\varphi^{\infty,*}\|_{L^2(\mathbb{S})}},
\end{align*}
for \(\omega = \pi\) and \(\omega = 10\pi\), as shown in Tables \ref{table2} and \ref{table3}. 
\begin{table}[h]
	\centering
	\caption{Relative errors for the pear-shaped penetrable obstacle with $\omega=\pi$}\label{table2}
	\begin{tabular}{@{}llllll@{}}
		\toprule
		$n$ & $E^{\mathrm{in}}(\phi_1)$ & $E^{\mathrm{in}}(\psi_1)$ & $E^{\infty}(\phi_2)$ & $E^{\infty}(\psi_2)$ & time \\
		\midrule
		8 & 0.0032& 0.0166 & 0.0071 & 0.0265 & 0.006s \\
		16 & 1.6252e-06& 1.7249e-06 &8.8402e-07 & 5.7974e-06 & 0.008s \\
		24 & 4.9990e-09& 7.2987e-09 & 2.5772e-10 & 1.7125e-09 & 0.011s \\
		32 & 2.9872e-13& 2.4718e-13 & 6.5495e-14 & 1.6336e-13 & 0.026s \\
		64 & 4.1614e-14 & 8.1235e-14 & 4.1397e-13 & 3.7103e-13 & 0.038s \\
		128 &  3.7783e-13 & 3.8529e-13 & 1.2112e-12 &  1.0340e-12 & 0.138s \\
		256 & 7.8033e-13 &  2.8276e-12 & 2.8269e-12 & 3.5567e-12 & 0.289s \\
		512 & 1.4381e-12 & 3.0792e-12 & 6.7556e-12 & 1.0955e-11 & 1.188s\\
		\bottomrule
	\end{tabular}
\end{table}
The results indicate that increasing the number of integration nodes \(n\) significantly improves accuracy at the cost of longer computation times. However, due to the accumulation of round-off errors, the relative error increases when an excessive number of nodes is used. Furthermore, the high-frequency case (\(\omega = 10\pi\)) requires more integration nodes than the low-frequency case (\(\omega = \pi\)) to achieve comparable accuracy.

\begin{table}[h]
	\centering
	\caption{Relative errors for the pear-shaped penetrable obstacle with $\omega=10\pi$}\label{table3}
	\begin{tabular}{@{}llllll@{}}
		\toprule
		$n$ & $E^{\mathrm{in}}(\phi_1)$ & $E^{\mathrm{in}}(\psi_1)$ & $E^{\infty}(\phi_2)$ & $E^{\infty}(\psi_2)$ & time \\
		\midrule
		24 & 0.0088& 0.0242 & 0.5239 & 2.0557 & 0.012s \\
		32 & 1.3471e-04& 4.6395e-04 & 0.0320 & 0.1117 & 0.028s \\
		48 & 8.5545e-08& 2.8629e-07 & 1.08419e-04 & 3.9380e-04 & 0.043s \\
		64 & 1.1347e-13 & 4.4910e-13 & 1.2301e-10 & 2.2972e-09 & 0.044s \\
		128 &  4.2999e-14 & 8.2160e-14 & 8.3902e-13 &  1.4661e-12 & 0.104s \\
		256 & 3.3419e-13 &  4.1357e-13 & 4.2040e-12 & 6.7444e-12 & 0.300s \\
		512 & 9.8709e-13 & 7.7154e-13 & 1.4891e-11 & 2.3863e-11 & 1.300s\\
		\bottomrule
	\end{tabular}
\end{table}

\section{The inverse scattering problem}
In this section, we employ three indicators to reconstruct both the location and shape of the penetrable obstacle from far-field data. We also analyze the decay behavior of these indicators as sampling points move away from the obstacle boundary and establish corresponding stability estimates.

The compressional and shear far-field patterns of the scattered field 
\(\pmb{v}\) are denoted by $v_{pp}^{\infty}(\hat{\pmb{x}},\pmb{d})$ and $v_{ps}^{\infty}(\hat{\pmb{x}},\pmb{d})$ for an incident P-wave $ \pmb{u}_p $, and by $v_{sp}^{\infty}(\hat{\pmb{x}},\pmb{d})$ and $v_{ss}^{\infty}(\hat{\pmb{x}},\pmb{d})$ for an incident S-wave $\pmb{u}_s$, with $\hat{\pmb{x}}$ and $\pmb{d}$ being the observation and incident directions, respectively.

For any sampling point \(\pmb{z}\in\mathbb{R}^2\), define the test functions
\[
\Xi_{\pmb{z}}^{p}(\boldsymbol{\eta})
= {\rm e}^{-\mathrm{i}\kappa_{p}\pmb{z}  \cdot  \boldsymbol{\eta}},
\quad
\Xi_{\pmb{z}}^{s}(\boldsymbol{\eta})
= {\rm e}^{-\mathrm{i}\kappa_{s}\pmb{z}  \cdot  \boldsymbol{\eta}},
\quad
\boldsymbol{\eta}\in\mathbb{S}
\]
with \(\Xi_{\pmb{z}}^{p}, \Xi_{\pmb{z}}^{s} \in L^2(\mathbb{S})\). We then formulate the following three inverse elastic scattering problems

\begin{enumerate}
	\item \(\pmb{I}_F\) (Full-aperture problem): Given the full far-field data
	\[
	\left\{ v_{pp}^\infty(\hat{\pmb{x}},\pmb{d}),\ v_{ps}^\infty(\hat{\pmb{x}},\pmb{d}),v_{sp}^\infty(\hat{\pmb{x}},\pmb{d}),\ v_{ss}^\infty(\hat{\pmb{x}},\pmb{d}) \right\}
	\]
	for all \(\hat{\pmb{x}},\pmb{d} \in \mathbb{S}\), reconstruct the location and shape of the medium \(D\).
	\item \(\pmb{I}_P\) (P-wave problem): Given only the compressional far-field data
	\[
	\left\{ v_{pp}^\infty(\hat{\pmb{x}},\pmb{d}) \right\}
	\]
	for all \(\hat{\pmb{x}},\pmb{d} \in \mathbb{S}\), reconstruct the location and shape of \(D\).
	
	\item \(\pmb{I}_S\) (S-wave problem): Given only the shear far-field data
	\[
	\left\{ v_{ss}^\infty(\hat{\pmb{x}},\pmb{d}) \right\}
	\]
	for all \(\hat{\pmb{x}},\pmb{d} \in \mathbb{S}\), reconstruct the location and shape of \(D\).
\end{enumerate}
To address these inverse problems, we introduce the following indicators
\begin{align*}
	\pmb{I}_F(\pmb{z}) &= \left| \int_{\mathbb{S}} \begin{pmatrix} \Xi_{\pmb{z}}^p(\pmb{d}) \\ \Xi_{\pmb{z}}^s(\pmb{d}) \end{pmatrix}^{\top} 
	\int_{\mathbb{S}} 
	\begin{pmatrix}
		v_{pp}^{\infty}(\hat{\pmb{x}},\pmb{d}) & v_{ps}^{\infty}(\hat{\pmb{x}},\pmb{d}) \\
		v_{sp}^{\infty}(\hat{\pmb{x}},\pmb{d}) & v_{ss}^{\infty}(\hat{\pmb{x}},\pmb{d})
	\end{pmatrix}
	\begin{pmatrix}
		\overline{\Xi_{\pmb{z}}^p(\hat{\pmb{x}})} \\[2pt]
		\overline{\Xi_{\pmb{z}}^s(\hat{\pmb{x}})}
	\end{pmatrix}
	\mathrm{d}s(\hat{\pmb{x}}) \mathrm{d}s(\pmb{d}) \right|, 
\end{align*}
\begin{align*}
	\pmb{I}_P(\pmb{z}) &= \left| \int_{\mathbb{S}} \Xi_{\pmb{z}}^p(\pmb{d}) 
	\int_{\mathbb{S}} v_{pp}^{\infty}(\hat{\pmb{x}},\pmb{d}) 
	\overline{\Xi_{\pmb{z}}^p(\hat{\pmb{x}})} \mathrm{d}s(\hat{\pmb{x}}) \mathrm{d}s(\pmb{d}) \right|, \\
	\pmb{I}_S(\pmb{z}) &= \left| \int_{\mathbb{S}} \Xi_{\pmb{z}}^s(\pmb{d}) 
	\int_{\mathbb{S}} v_{ss}^{\infty}(\hat{\pmb{x}},\pmb{d}) 
	\overline{\Xi_{\pmb{z}}^s(\hat{\pmb{x}})} \mathrm{d}s(\hat{\pmb{x}}) \mathrm{d}s(\pmb{d}) \right|.
\end{align*}
The analysis of the behavior of these indicators for sampling points away from the obstacle boundary requires the following auxiliary functions, defined for each \(\pmb{z} \in \mathbb{R}^2\):

\[
\varPsi_{mn}(\pmb{z},\pmb{d}) := \int_{\mathbb{S}} v_{mn}^{\infty}(\hat{\pmb{x}},\pmb{d}) 
\overline{\Xi_{\pmb{z}}^{n}(\hat{\pmb{x}})} \mathrm{d}s(\hat{\pmb{x}}), \quad m,n \in \{p,s\},
\]
For each fixed sampling point $\pmb{z}$, the function \(\varPsi_{mn}(\pmb{z},\cdot)\in L^{2}(\mathbb{S})\). The indicators \(\pmb{I}_F\), \(\pmb{I}_P\), and \(\pmb{I}_S\) can be expressed in terms of
\[
\int_{\mathbb{S}} \varPsi_{mn}(\pmb{z},\pmb{d}) \Xi_{\pmb{z}}^{m}(\pmb{d}) \mathrm{d}s(\pmb{d}), \quad m,n \in \{p,s\}.
\]

\begin{lemma}[{\cite[Lemma 2.7]{LXD}}]
	\label{lemma}
	In two dimensions, for any \(\pmb{p} \in \mathbb{R}^2\), the following identities hold
	\begin{equation*}
		\begin{aligned}
			\int_{\mathbb{S}}   {\rm e}^{-\mathrm{i}\kappa\hat{\pmb{x}}\cdot\pmb{p}} 
			\mathrm{d}s(\hat{\pmb{x}})
			&=2\pi J_{0}  \bigl(\kappa\lvert\pmb{p}\rvert\bigr),\\
			\int_{\mathbb{S}}\hat{\pmb{x}}
			{\rm e}^{-\mathrm{i}\kappa \hat{\pmb{x}}\cdot\pmb{p}} 
			\mathrm{d}s(\hat{\pmb{x}})
			&=
			\begin{cases}
				\pmb{0}, & \pmb{p}=\pmb{0},\\
				\displaystyle \frac{2\pi}{\mathrm{i}} 
				\hat{\pmb{p}} J_{1}  \bigl(\kappa\lvert\pmb{p}\rvert\bigr),
				& \pmb{p}\neq\pmb{0},
			\end{cases}
		\end{aligned}
	\end{equation*}
	where \(\hat{\pmb{p}} = \pmb{p} / \lvert \pmb{p} \rvert\).
\end{lemma}

\begin{theorem}\label{decay}
	For any incidence direction \(\pmb{d} \in \mathbb{S}\), the auxiliary function
	\[
	\varPsi_{pp}(\pmb{z},\pmb{d})
	= \int_{\mathbb{S}} v_{pp}^{\infty}(\hat{\pmb{x}},\pmb{d})  \overline{\Xi_{\pmb{z}}^{p}(\hat{\pmb{x}})}  \mathrm{d}s(\hat{\pmb{x}})
	\]
	decays as the sampling point \(\pmb{z} \in \mathbb{R}^2\) moves away from the obstacle boundary \(\partial D\).
\end{theorem}
\begin{proof}
	
	Based on the far-field pattern relations \eqref{far} and the integral representation
	\[
	\phi_{2}^{\infty}(\pmb{\hat{x}},\pmb{d})
	= \gamma_{\kappa_{p}}
	\int_{\partial D}  \Bigl\{\phi_{2}(\pmb{y},\pmb{d}) \partial_{\pmb{\nu}(\pmb{y})}
	{\rm e}^{-\mathrm{i} \kappa_{p} \pmb{\hat{x}}\cdot\pmb{y}}
	- (\partial_{\pmb{\nu}}\phi_{2})(\pmb{y},\pmb{d})   {\rm e}^{-\mathrm{i} \kappa_{p} \pmb{\hat{x}}\cdot\pmb{y}}
	\Bigr\} ds(\pmb{y}),
	\]
	we deduce that
	\begin{align*}
		v_{pp}^{\infty}(\pmb{\hat{x}},\pmb{d})
		&= \mathrm{i} \kappa_{p} \phi_{2,p}^{\infty}(\pmb{\hat{x}},\pmb{d})
		= \mathrm{i} \kappa_{p} \gamma_{\kappa_{p}}
		\int_{\partial D}  
		\Bigl\{\phi_{2,p}(\pmb{y},\pmb{d}) \partial_{\pmb{\nu}(\pmb y)}
		{\rm e}^{-\mathrm{i} \kappa_{p} \pmb{\hat{x}}\cdot\pmb{y}}- (\partial_{\pmb{\nu}}\phi_{2,p})(\pmb{y},\pmb{d})   {\rm e}^{-\mathrm{i} \kappa_{p} \pmb{\hat{x}}\cdot\pmb{y}}
		\Bigr\} ds(\pmb{y}),
	\end{align*}
	where \(\phi_{2,p}^{\infty}\) denotes the far-field pattern of the scalar potential \(\phi_{2}\)  under compressional plane wave incidence. Consequently, the auxiliary function becomes
	\\
	\begin{align*}
		\varPsi_{pp}(\pmb{z},\pmb{d})&=\int_{\mathbb{S}}v_{pp}^{\infty}(\pmb{\hat{x}},\pmb{d})\overline{\Xi_{\pmb{z}}^p(\pmb{\hat{x}})}\mathrm{d}s(\pmb{\hat{x}})\\
		&=\mathrm{i}\kappa_{p}\gamma_{\kappa_{p}}\int_{\mathbb{S}}\int_{\partial D}\Big\{\phi_{2,p}(\pmb{y},\pmb{d})\partial_{\pmb{\nu}(\pmb{y})}\mathrm{e}^{\mathrm{-i}\kappa_{p}\pmb{\hat{x}}\cdot (\pmb{y}-\pmb{z})} -(\partial_{\pmb{\nu}}\phi_{2,p})(\pmb{y},\pmb{d})\mathrm{e}^{\mathrm{-i}\kappa_{p}\pmb{\hat{x}}\cdot (\pmb{y}-\pmb{z})}\Big\}\mathrm{d}s(\pmb{y})\mathrm{d}s(\pmb{\hat{x}})\\
		&=\mathrm{i}\kappa_{p}\gamma_{\kappa_{p}}\int_{\partial D}\Big\{-\mathrm{i}\kappa_p\phi_{2,p}(\pmb{y},\pmb{d})\pmb{\nu}(\pmb{y})\cdot\int_{\mathbb{S}}\pmb{\hat{x}}\mathrm{e}^{\mathrm{-i}\kappa_{p}\pmb{\hat{x}}\cdot (\pmb{y}-\pmb{z})}\mathrm{d}s(\pmb{\hat{x}})\\
		&\qquad\qquad\qquad\qquad\qquad-(\partial_{\pmb{\nu}}\phi_{2,p})(\pmb{y},\pmb{d})\int_{\mathbb{S}}\mathrm{e}^{\mathrm{-i}\kappa_{p}\pmb{\hat{x}}\cdot (\pmb{y}-\pmb{z})}\mathrm{d}s(\pmb{\hat{x}})\Big\}\mathrm{d}s(\pmb{y}).
	\end{align*}
	Applying Lemma \ref{lemma} yields
	\begin{align*}
		\varPsi_{pp}(\pmb{z},\pmb{d})
		&= \mathrm{i} \kappa_{p} \gamma_{\kappa_{p}}
		\int_{\partial D}
		\Bigl\{
		-\mathrm{i} \kappa_{p} \frac{2\pi}{\mathrm{i}} 
		\phi_{2,p}(\pmb{y},\pmb{d}) 
		\pmb{\nu}(\pmb{y})  \cdot  \frac{\pmb{y}-\pmb{z}}{\lvert\pmb{y}-\pmb{z}\rvert} 
		J_{1}(\kappa_{p}\lvert\pmb{y}-\pmb{z}\rvert)\\
		&\qquad\qquad\qquad\qquad\qquad\qquad-2\pi (\partial_{\pmb{\nu}}\phi_{2,p})(\pmb{y},\pmb{d}) 
		J_{0}(\kappa_{p}\lvert\pmb{y}-\pmb{z}\rvert)
		\Bigr\} ds(\pmb{y}).
	\end{align*}
	Therefore, \(\varPsi_{pp}(\pmb{z},\pmb{d})\) is expressed as a superposition of Bessel functions \(J_1\) and \(J_0\), implying that it decays asymptotically at the rate of a Bessel function as the sampling point \(\pmb{z}\) moves away from \(\partial D\).
\end{proof}

As established in Theorem \ref{decay}, the indicator $\pmb{I}_P$
decays as the sampling point moves away from the obstacle boundary. This behavior is confirmed by the numerical experiments in Section \ref{numex}. Below, we present a stability result for the indicator $\pmb{I}_P$, which ensures the reliability of our inversion method even in the presence of measurement errors.

\begin{theorem}\label{the}
	Let \(\pmb{I}_P\) and \(\pmb{I}_{P}^\delta\) be the indicators constructed from the exact and noisy far-field data \(v_{pp}^{\infty}\) and \(v_{pp,\delta}^{\infty}\), respectively. Defining the auxiliary functions
	$$
	\varPsi_{pp}(\pmb{z},\pmb{d})=\int_{\mathbb{S}}v_{pp}^{\infty}(\pmb{\hat{x}},\pmb{d})\overline{\Xi_{\pmb{z}}^p(\pmb{\hat{x}})}\mathrm{d}s(\pmb{\hat{x}}),\quad \varPsi^{\delta}_{pp}(\pmb{z},\pmb{d})=\int_{\mathbb{S}}v_{pp,\delta}^{\infty}(\pmb{\hat{x}},\pmb{d})\overline{\Xi_{\pmb{z}}^p(\pmb{\hat{x}})}\mathrm{d}s(\pmb{\hat{x}}),
	$$
	for $\pmb{\hat{x}},\pmb{d}\in\mathbb{S}$.
	Then there exists a constant \(C > 0\), independent of the sampling point \(\pmb{z} \in \mathbb{R}^2\), such that
	\[
	\bigl| \pmb{I}_P(\pmb{z}) - \pmb{I}_P^\delta(\pmb{z}) \bigr|
	\le C  \bigl\| v_{pp}^{\infty} - v_{pp,\delta}^{\infty} \bigr\|_{L^2(\mathbb{S}\times\mathbb{S})}.
	\]
\end{theorem}
\begin{proof}
	Applying the reverse triangle inequality and the Cauchy-Schwarz inequality, we derive
	\begin{align*}
		& \bigl| \pmb{I}_P(\pmb{z}) - \pmb{I}_P^\delta(\pmb{z}) \bigr| \\
		=& \left| 
		\left| \int_{\mathbb{S}} \Xi_{\pmb{z}}^p(\pmb{d}) 
		\varPsi_{pp}(\pmb{z},\pmb{d})  \mathrm{d}s(\pmb{d}) \right|
		- \left| \int_{\mathbb{S}} \Xi_{\pmb{z}}^p(\pmb{d}) 	\varPsi^{\delta}_{pp}(\pmb{z},\pmb{d})  \mathrm{d}s(\pmb{d}) \right|
		\right| \\
		\le& \left| \int_{\mathbb{S}} \Xi_{\pmb{z}}^p(\pmb{d}) \int_{\mathbb{S}} \left[ v_{pp}^{\infty}(\hat{\pmb{x}},\pmb{d}) - v_{pp,\delta}^{\infty}(\hat{\pmb{x}},\pmb{d}) \right] \overline{\Xi_{\pmb{z}}^p(\hat{\pmb{x}})}  \mathrm{d}s(\hat{\pmb{x}}) \mathrm{d}s(\pmb{d}) \right| \\
		\le& C  \left\| v_{pp}^{\infty} - v_{pp,\delta}^{\infty} \right\|_{L^2(\mathbb{S}\times\mathbb{S})}.
	\end{align*}
\end{proof}

From the preceding analysis, the auxiliary functions \(\varPsi_{ps}(\pmb{z},\pmb{d})\), \(\varPsi_{sp}(\pmb{z},\pmb{d})\), and \(\varPsi_{ss}(\pmb{z},\pmb{d})\) can also be expressed as superpositions of Bessel functions. Consequently, the indicators \(\pmb{I}_F\) and \(\pmb{I}_S\) exhibit analogous asymptotic decay properties. Moreover, the stability of \(\pmb{I}_F\) and \(\pmb{I}_S\) can be established by arguments analogous to those in Theorem \ref{the}.

In particular, we note that if the far-field patterns \(\phi_{2}^{\infty}\) and \(\psi_{2}^{\infty}\) are available, then the quantities \(v_{pp}^\infty\), \(v_{ps}^\infty\), \(v_{sp}^\infty\) and \(v_{ss}^\infty\) required by the indicators can be directly computed. Thus, our indicators are fundamentally constructed from these scalar far-field patterns.

\section{Numerical experiments}\label{numex}

To evaluate the effectiveness of the indicators, a series of numerical experiments are performed in this section. The far-field patterns \(v_{pp}^{\infty}(\hat{\pmb{x}}_i, \pmb{d}_j)\), \(v_{ps}^{\infty}(\hat{\pmb{x}}_i, \pmb{d}_j)\), \(v_{sp}^{\infty}(\hat{\pmb{x}}_i, \pmb{d}_j)\), and \(v_{ss}^{\infty}(\hat{\pmb{x}}_i, \pmb{d}_j)\) are computed for uniformly distributed observation directions \(\hat{\pmb{x}}_i = (\cos(2\pi i / N), \sin(2\pi i / N))^{\top}\) with \(i = 1, 2, \ldots, N\), and incident directions \(\pmb{d}_j = (\cos(2\pi j / M),\sin(2\pi j / M))^{\top}\) with \(j = 1, 2, \ldots, M\), where \(N = M = 144\). The resulting far-field data are stored in matrices \(\pmb{U}_{F} \in \mathbb{C}^{N \times M}\) for \(F \in \{pp, ps, sp, ss\}\).
\begin{figure}[!htpb]
	\centering
	\resizebox{0.85\textwidth}{!}{
		\begin{tikzpicture}[
			transform shape,
			label/.style={font=\normalsize, anchor=center, align=center},
			leftlabel/.style={font=\normalsize, anchor=east, align=right}
			]
			\node[label] at (-3,0) {Noise \\ Level};
			\node[label] at (-1,0) {Ground \\ Truth};
			\node[label] at (1.9,0) {$\pmb{I}_P$};
			\node[label] at (4.9,0) {$\pmb{I}_S$};
			\node[label] at (7.9,0) {$\pmb{I}_F$};
			\node[label] at (-3,-2) {0\%};
			\node[label] at (-3,-4.5) {30\%};
			\node[label] at (-1,-3.25) {\includegraphics[width=3.4cm, keepaspectratio]{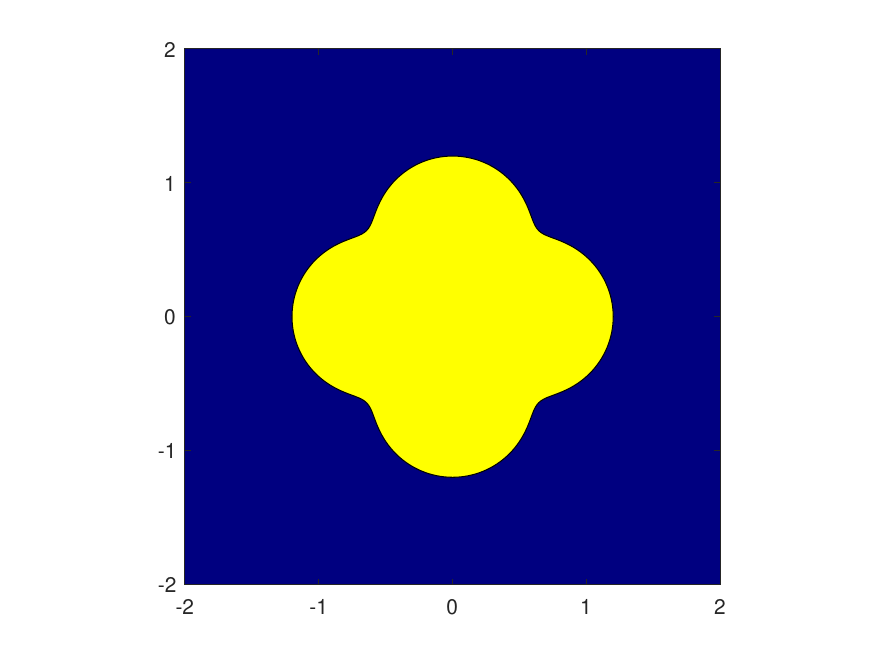}};
			\node[label] at (1.9,-2) {\includegraphics[width=3.4cm, keepaspectratio]{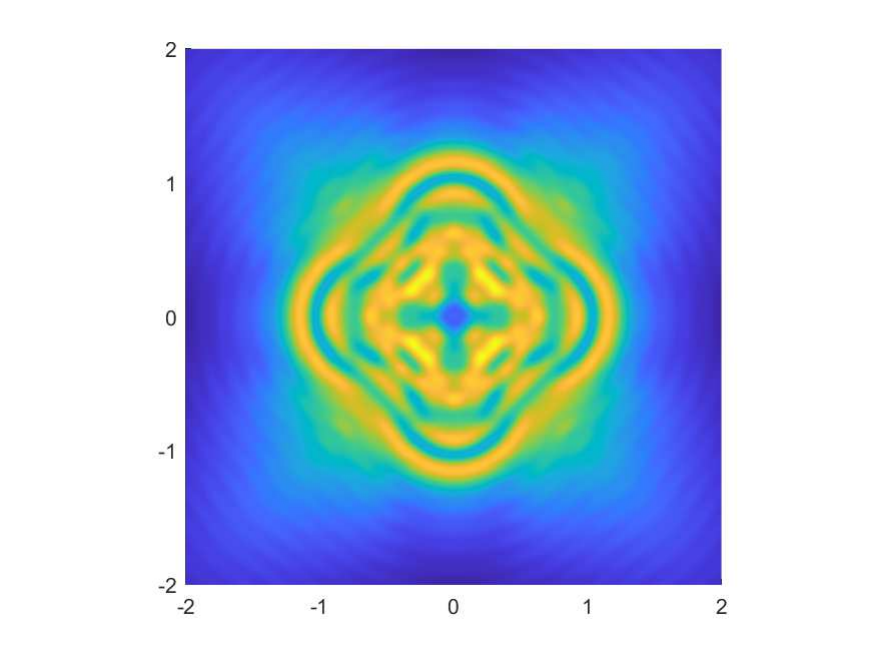}};
			\node[label] at (4.9,-2) {\includegraphics[width=3.4cm, keepaspectratio]{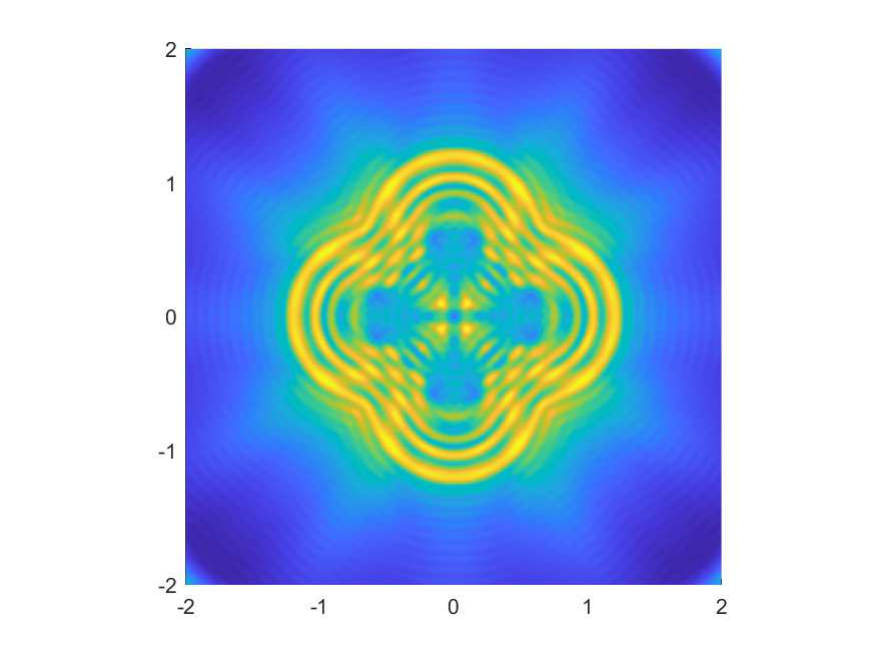}};
			\node[label] at (7.9,-2) {\includegraphics[width=3.4cm, keepaspectratio]{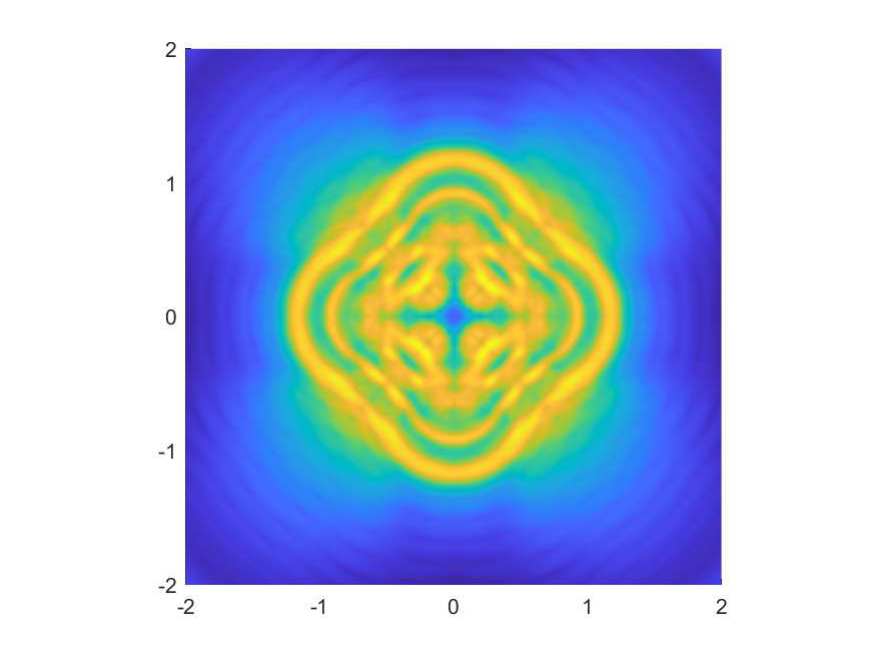}};
			
			\node[label] at (1.9,-4.5) {\includegraphics[width=3.4cm, keepaspectratio]{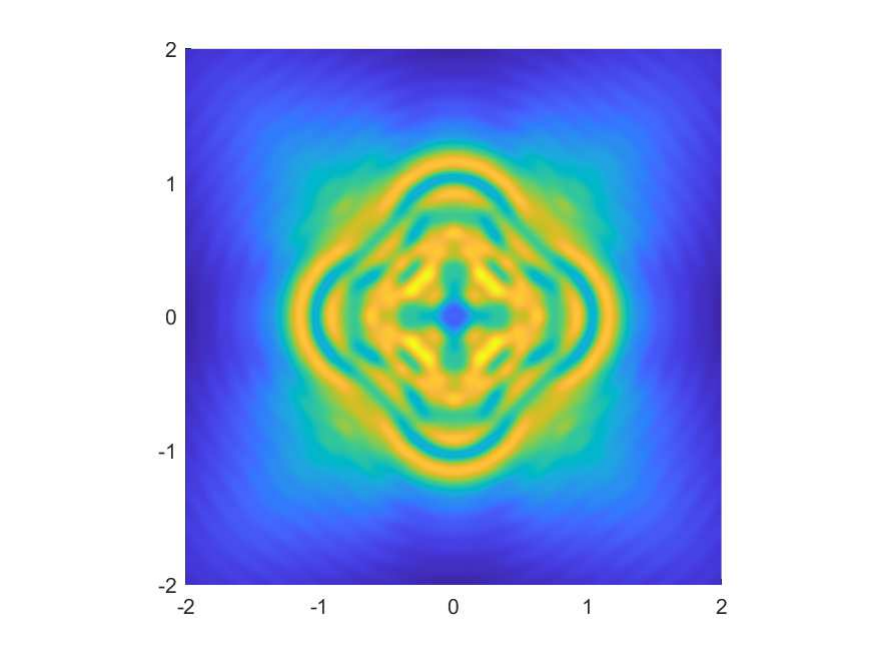}};
			\node[label] at (4.9,-4.5) {\includegraphics[width=3.4cm, keepaspectratio]{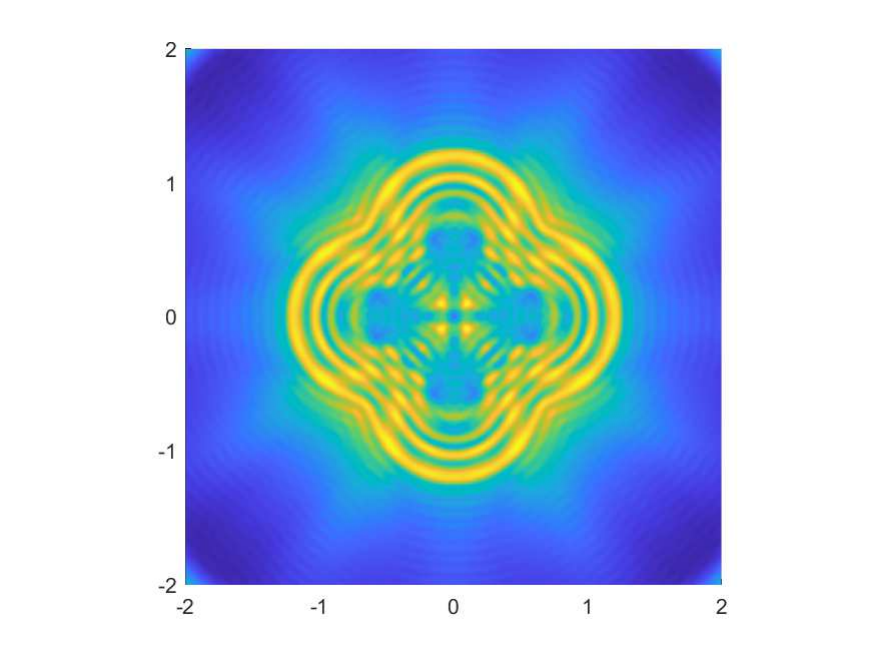}};
			\node[label] at (7.9,-4.5) {\includegraphics[width=3.4cm, keepaspectratio]{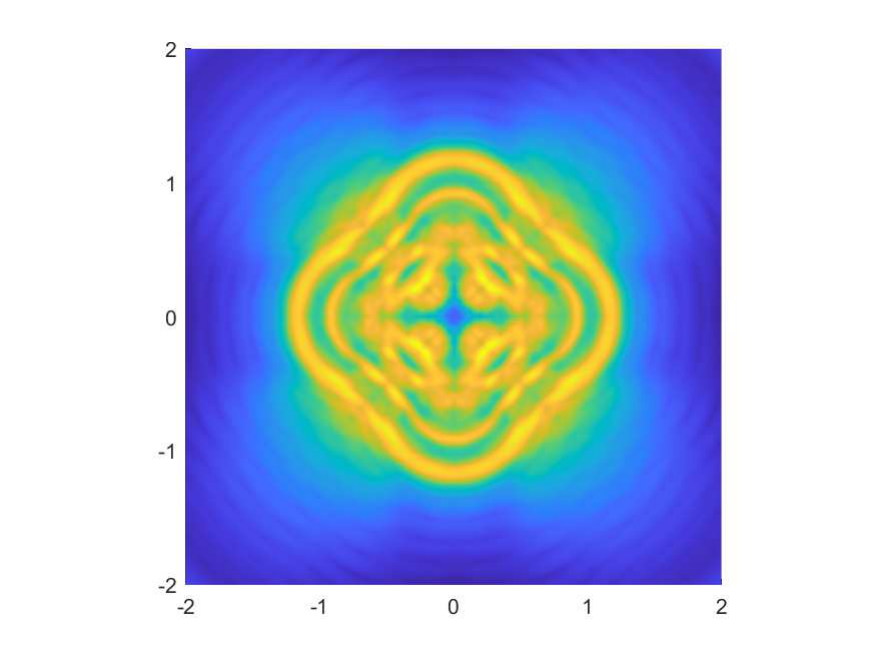}};
			
		\end{tikzpicture}
	}
	\caption{Reconstruction of flower-shaped penetrable obstacle using indicators $\pmb{I}_P$, $\pmb{I}_S$ and $\pmb{I}_F$ with  \(0\%\) and \(30\%\) noise level.}\label{figure1}
\end{figure}
\begin{figure}[!htbp]
	\centering
	\resizebox{0.85\textwidth}{!}{
		\begin{tikzpicture}[
			transform shape,
			label/.style={font=\normalsize, anchor=center, align=center},
			leftlabel/.style={font=\normalsize, anchor=east, align=right}
			]
			\node[label] at (-3,0) {Noise \\ Level};
			\node[label] at (-1,0) {Ground \\ Truth};
			\node[label] at (1.9,0) {$\pmb{I}_P$};
			\node[label] at (4.9,0) {$\pmb{I}_S$};
			\node[label] at (7.9,0) {$\pmb{I}_F$};
			
			\node[label] at (-3,-2) {0\%};
			\node[label] at (-3,-4.5) {30\%};
			\node[label] at (-1,-3.25) {\includegraphics[width=3.4cm]{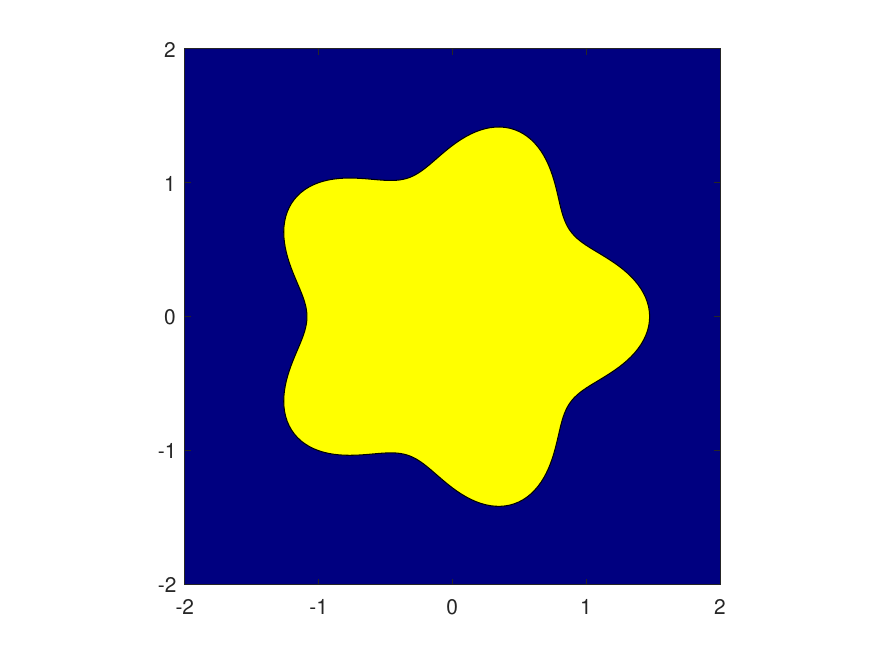}};
			\node[label] at (1.9,-2) {\includegraphics[width=3.4cm]{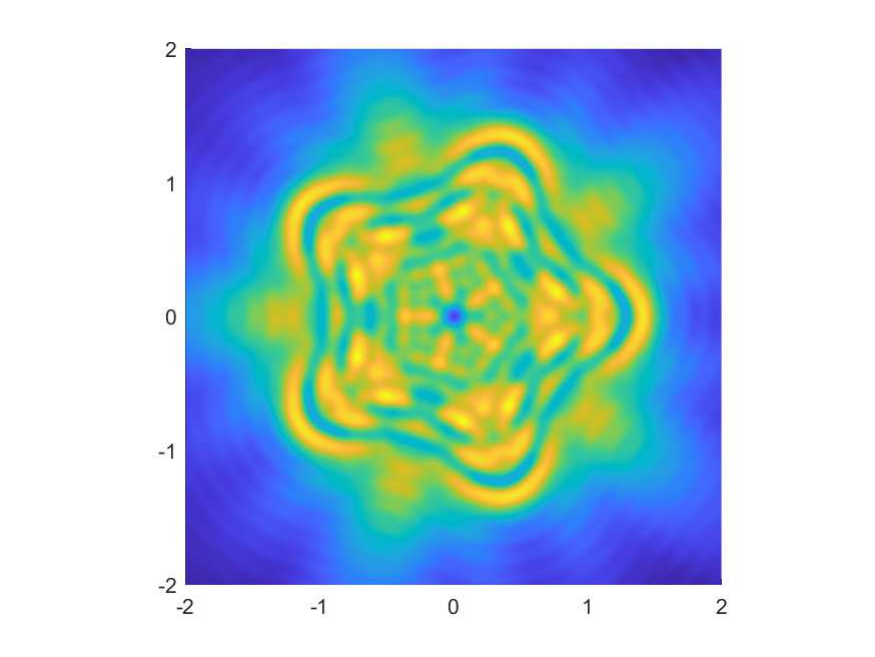}};
			\node[label] at (4.9,-2) {\includegraphics[width=3.4cm]{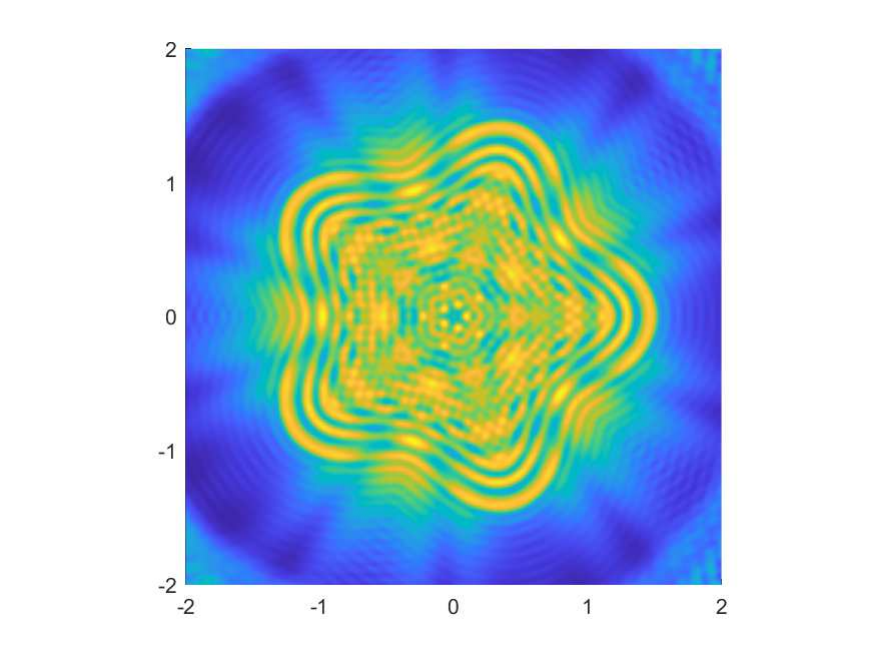}};
			\node[label] at (7.9,-2) {\includegraphics[width=3.4cm]{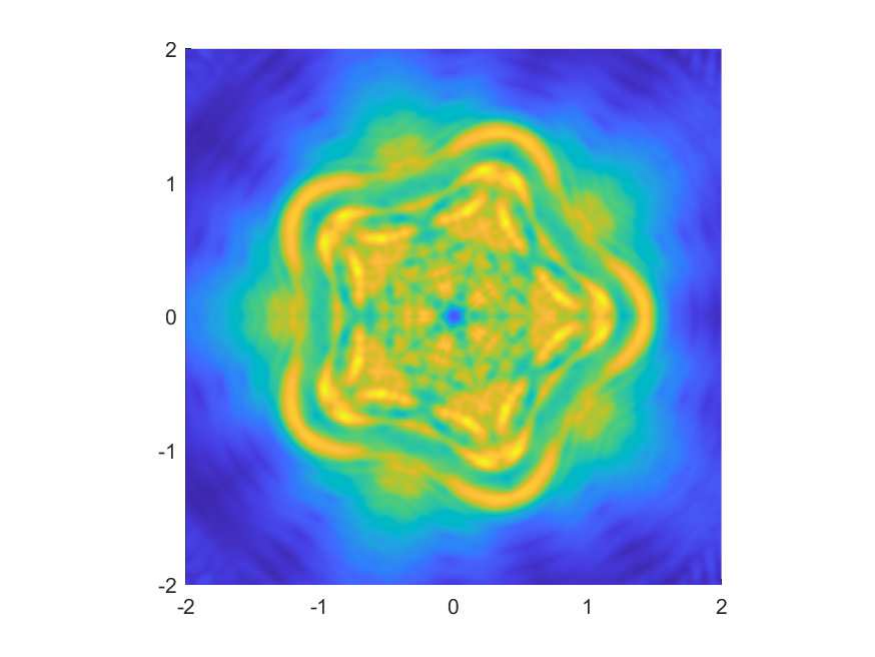}};
			\node[label] at (1.9,-4.5) {\includegraphics[width=3.4cm]{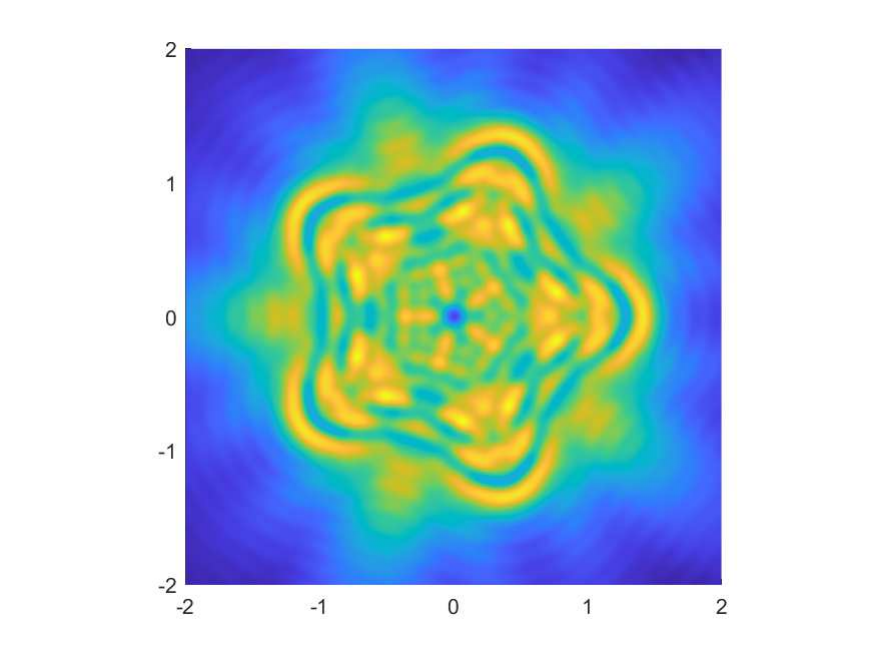}};
			\node[label] at (4.9,-4.5) {\includegraphics[width=3.4cm]{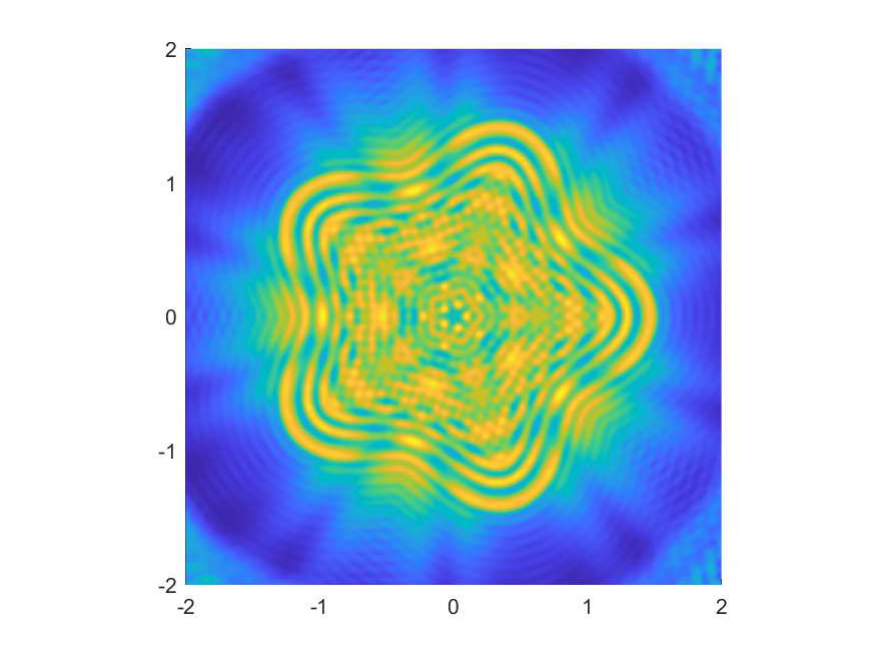}};
			\node[label] at (7.9,-4.5) {\includegraphics[width=3.4cm]{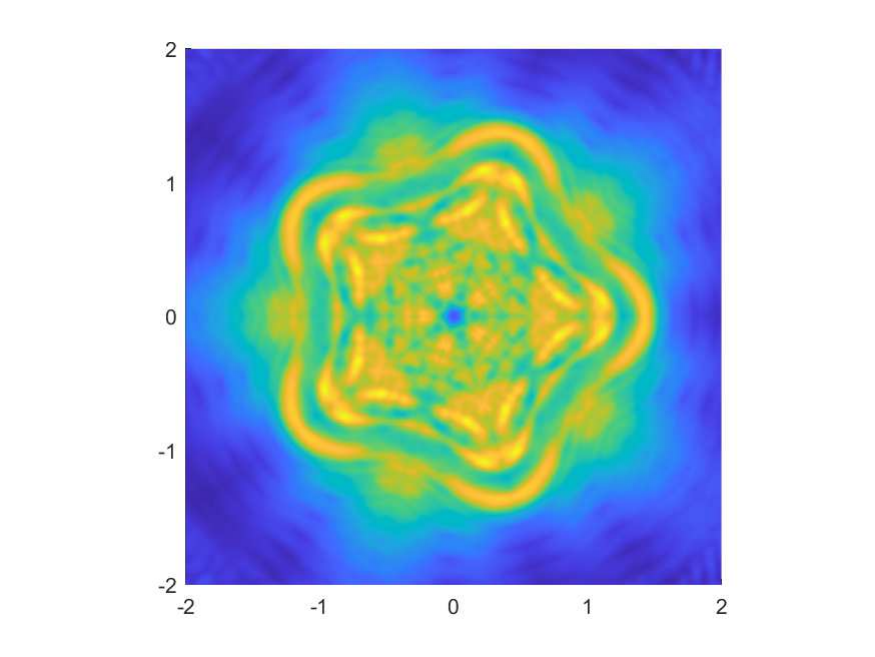}};
		\end{tikzpicture}
	} 
	\caption{Reconstruction of 5-leaf-shaped penetrable obstacle using indicators $\pmb{I}_P$, $\pmb{I}_S$ and $\pmb{I}_F$ with 0\% and 30\% noise level.}
	\label{figureleaf}
\end{figure}

To examine the stability of the indicators, we introduce a perturbation term to the computed far-field data, yielding the noisy data
\[
\pmb{U}_{F}^{\delta} = \pmb{U}_{F} + \delta \| \pmb{U}_{F} \| \frac{\pmb{R}_1 + \mathrm{i} \pmb{R}_2}{\| \pmb{R}_1 + \mathrm{i} \pmb{R}_2 \|},
\]
where \(\pmb{R}_1\) and \(\pmb{R}_2\) are \(N \times M\) matrices with normally distributed entries in \([-1, 1]\), and \(\delta > 0\) represents the relative noise level.

For the limited-aperture case, let \(\mathbb{S}_0 \subset \mathbb{S}\) be a non-empty subset of directions, and define the product set \(\mathbb{S}_1 \times \mathbb{S}_2 \in \left\{ \mathbb{S} \times \mathbb{S}_0,~ \mathbb{S}_0 \times \mathbb{S} \right\}\). The corresponding indicators are defined as:
\begin{align*}
	\pmb I_{F,L}(\pmb{z})
	&= \left| \int_{\mathbb{S}_1}
	\begin{pmatrix} \Xi_{\pmb{z}}^p(\pmb{d}) \\ \Xi_{\pmb{z}}^s(\pmb{d}) \end{pmatrix}^{\top}
	\int_{\mathbb{S}_2}
	\begin{pmatrix}
		v_{pp}^{\infty}(\hat{\pmb{x}},\pmb{d}) & v_{ps}^{\infty}(\hat{\pmb{x}},\pmb{d}) \\
		v_{sp}^{\infty}(\hat{\pmb{x}},\pmb{d}) & v_{ss}^{\infty}(\hat{\pmb{x}},\pmb{d})
	\end{pmatrix}
	\begin{pmatrix}
		\overline{\Xi_{\pmb{z}}^p(\hat{\pmb{x}})} \\
		\overline{\Xi_{\pmb{z}}^s(\hat{\pmb{x}})}
	\end{pmatrix}
	\mathrm{d}s(\hat{\pmb{x}})
	\mathrm{d}s(\pmb{d}) \right|, \\
	\pmb I_{P,L}(\pmb{z})
	&= \left| \int_{\mathbb{S}_1} \Xi_{\pmb{z}}^p(\pmb{d})
	\int_{\mathbb{S}_2} v_{pp}^{\infty}(\hat{\pmb{x}},\pmb{d})
	\overline{\Xi_{\pmb{z}}^p(\hat{\pmb{x}})}
	\mathrm{d}s(\hat{\pmb{x}})
	\mathrm{d}s(\pmb{d}) \right|, \\
	\pmb I_{S,L}(\pmb{z})
	&= \left| \int_{\mathbb{S}_1} \Xi_{\pmb{z}}^s(\pmb{d})
	\int_{\mathbb{S}_2} v_{ss}^{\infty}(\hat{\pmb{x}},\pmb{d})
	\overline{\Xi_{\pmb{z}}^s(\hat{\pmb{x}})}
	\mathrm{d}s(\hat{\pmb{x}})
	\mathrm{d}s(\pmb{d}) \right|.
\end{align*}
This formulation corresponds to scenarios involving either partial incidence or partial observation of the far-field data.

\begin{table}[h]
	\centering
	\caption{Parametric form for the exact boundary curves}	\label{table1}
	\begin{tabular}{@{}ll@{}}
		\toprule
		Type           &Parametrization\\
		\midrule
		pear-shaped    & 
		$\pmb{x}(t)=\pmb c + \alpha(2 + 0.3 \cos (3t)) (\cos t,\sin t)^\top,\quad t\in [0,2\pi)$\\
		~\\
		flower-shaped    & 
		$\pmb{x}(t)=\pmb c + \alpha(\cos^{10}t+\sin^{10}t)^{\tfrac{1}{10}} (\cos t,\sin t)^\top,\quad t\in [0,2\pi)$\\
		~\\
		5-leaf-shaped    & 
		$\pmb{x}(t)=\pmb c + \alpha(4 + 0.6 \cos (5t)) (\cos t,\sin t)^\top,\quad t\in [0,2\pi)$\\
		~\\
		kite-shaped    & 
		$\pmb{x}(t)=\pmb c + \alpha(\cos t+0.65\cos(2t)-0.65,1.5\sin t)^\top,\quad t\in [0,2\pi)$\\
		~\\
		peanut-shaped  &
		$\pmb{x}(t)=\pmb c +\alpha\sqrt{3\cos ^{2}t+1}(\cos t,\sin t)^\top,\quad t\in [0,2\pi)$\\
		~\\
		apple-shaped  &
		$\pmb{x}(t)=\pmb c +\alpha\frac{2(1+0.9\cos t)+0.1\sin (2t)}{1+0.75\cos t}(\cos t,\sin t)^\top,\quad t\in [0,2\pi)$\\
		~\\
		circle-shaped  &
		$\pmb{x}(t)=\pmb c + \alpha (\cos t,\sin  t)^\top,\quad t\in [0,2\pi)$\\
		\bottomrule
	\end{tabular}
	\par\smallskip\noindent{\footnotesize Note: $\pmb c$ and $\alpha\in\mathbb{R}$ represent the position and scale of the obstacle, respectively.}
\end{table}

\begin{figure}[!htpb]
	\centering 
	\resizebox{0.85\textwidth}{!}{
		\begin{tikzpicture}[
			transform shape, 
			label/.style={font=\normalsize, anchor=center, align=center},
			leftlabel/.style={font=\normalsize, anchor=east, align=right}
			]
			\node[label] at (-3,0) {Noise \\ Level};
			\node[label] at (-1,0) {Ground \\ Truth};
			\node[label] at (1.9,0) {$\pmb{I}_P$};
			\node[label] at (4.9,0) {$\pmb{I}_S$};
			\node[label] at (7.9,0) {$\pmb{I}_F$};

			\node[label] at (-3,-2) {0\%};
			\node[label] at (-3,-4.5) {30\%};
			
			\node[label] at (-1,-3.25) {\includegraphics[width=3.4cm, keepaspectratio]{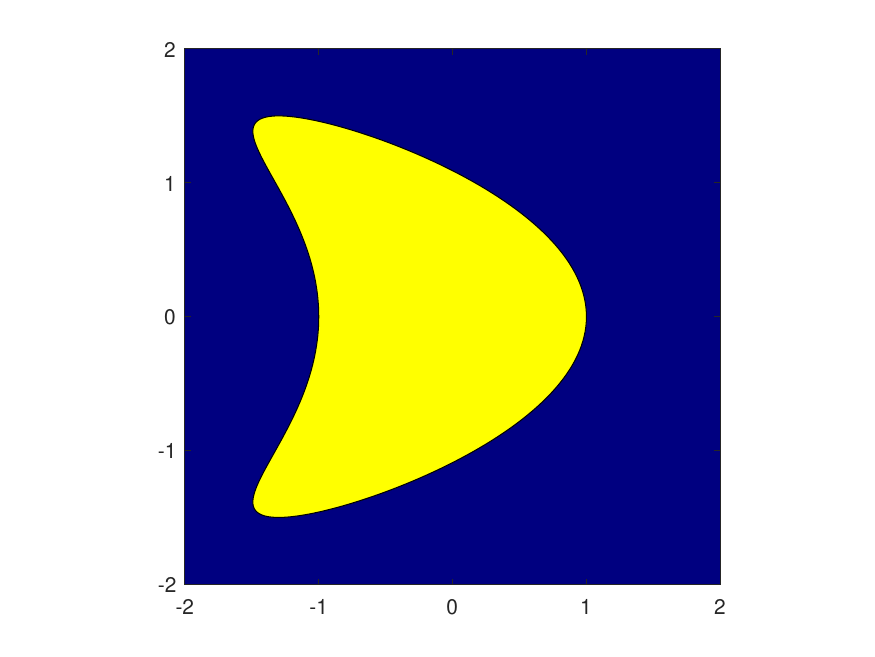}};
			\node[label] at (1.9,-2) {\includegraphics[width=3.4cm, keepaspectratio]{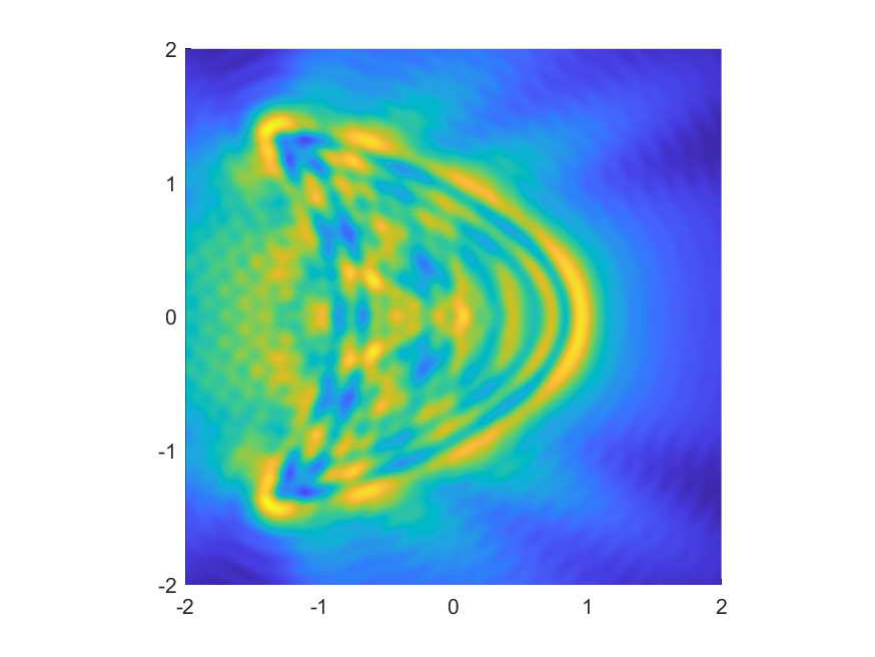}};
			\node[label] at (4.9,-2) {\includegraphics[width=3.4cm, keepaspectratio]{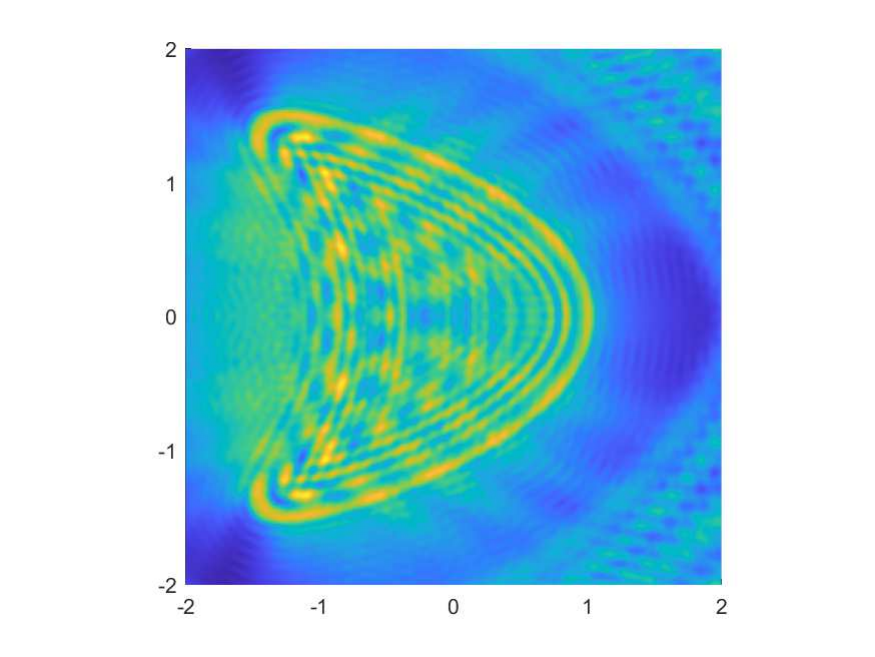}};
			\node[label] at (7.9,-2) {\includegraphics[width=3.4cm, keepaspectratio]{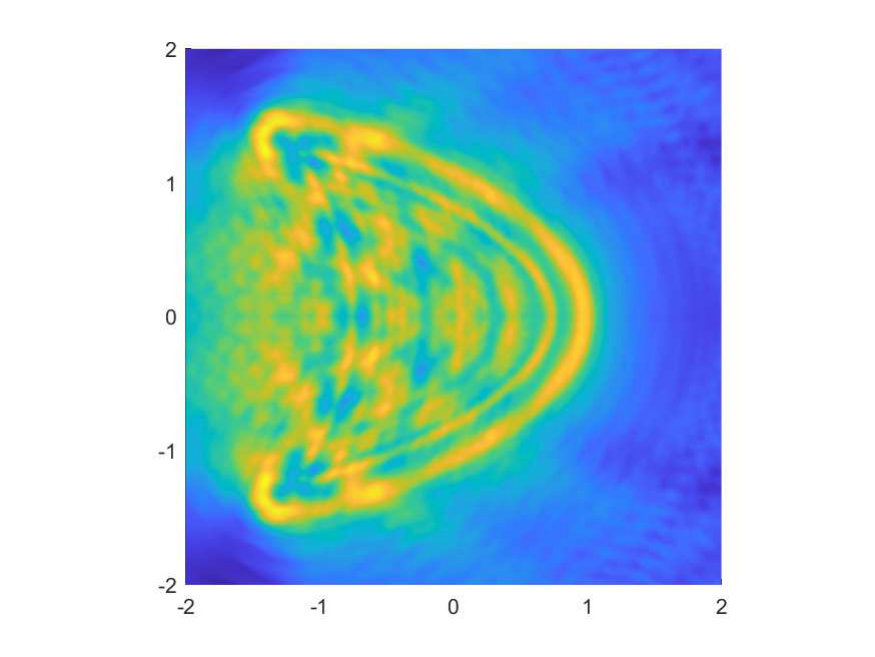}};
			
			\node[label] at (1.9,-4.5) {\includegraphics[width=3.4cm, keepaspectratio]{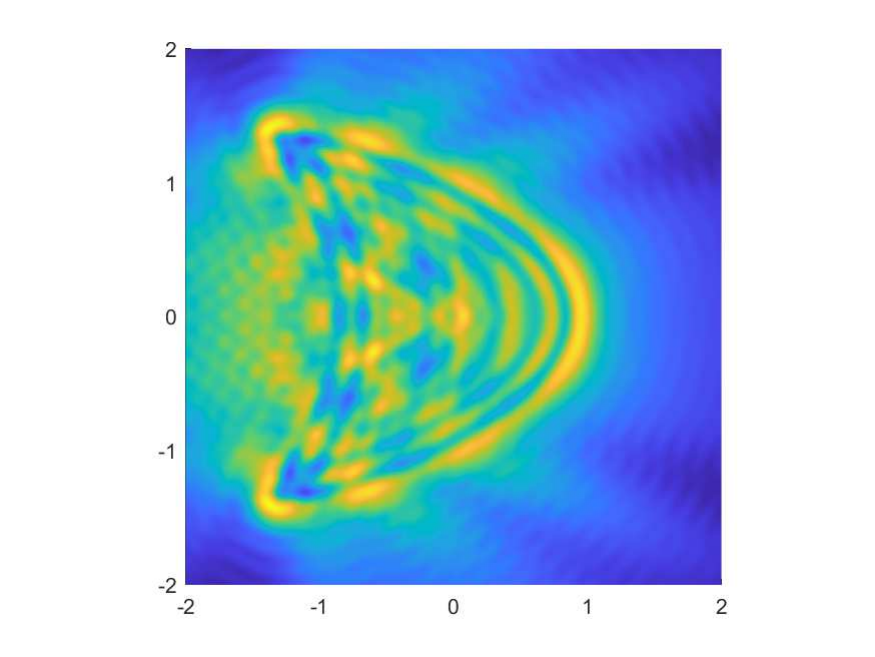}};
			\node[label] at (4.9,-4.5) {\includegraphics[width=3.4cm, keepaspectratio]{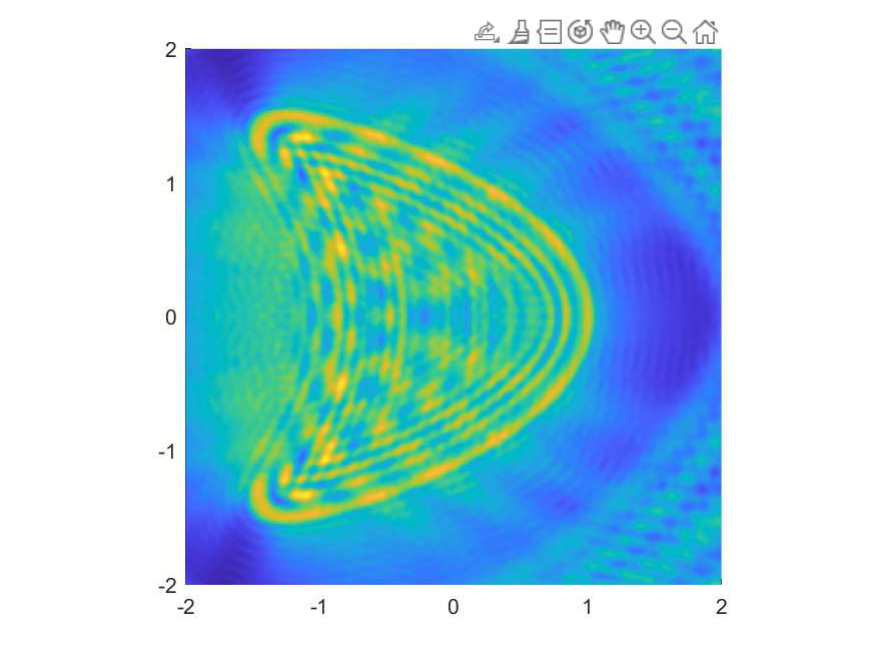}};
			\node[label] at (7.9,-4.5) {\includegraphics[width=3.4cm, keepaspectratio]{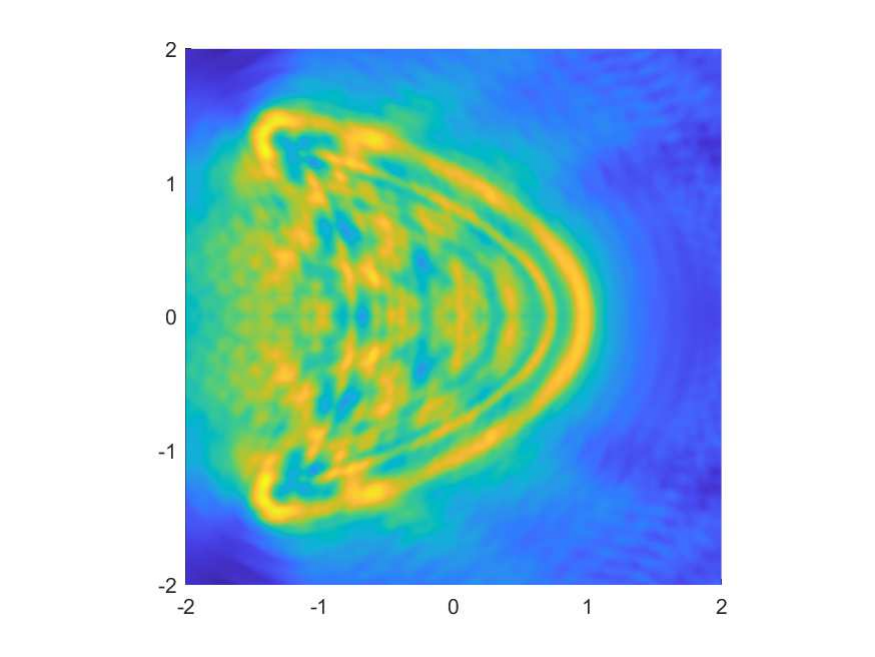}};
			
	\end{tikzpicture}}
	\caption{Reconstruction of kite-shaped penetrable obstacle using indicators $\pmb{I}_P$, $\pmb{I}_S$ and $\pmb{I}_F$ with  \(0\%\) and \(30\%\) noise level.}\label{figurekite}
\end{figure}
\begin{figure}[!htpb]
	\centering 
	\resizebox{0.85\textwidth}{!}{
		\begin{tikzpicture}[
			transform shape, 
			label/.style={font=\normalsize, anchor=center, align=center},
			leftlabel/.style={font=\normalsize, anchor=east, align=right}
			]
			\node[label] at (-3,0) {Noise \\ Level};
			\node[label] at (-1,0) {Ground \\ Truth};
			\node[label] at (1.9,0) {$\pmb{I}_P$};
			\node[label] at (4.9,0) {$\pmb{I}_S$};
			\node[label] at (7.9,0) {$\pmb{I}_F$};
			
			\node[label] at (-3,-2) {0\%};
			\node[label] at (-3,-4.5) {30\%};
			
			\node[label] at (-1,-3.25) {\includegraphics[width=3.4cm, keepaspectratio]{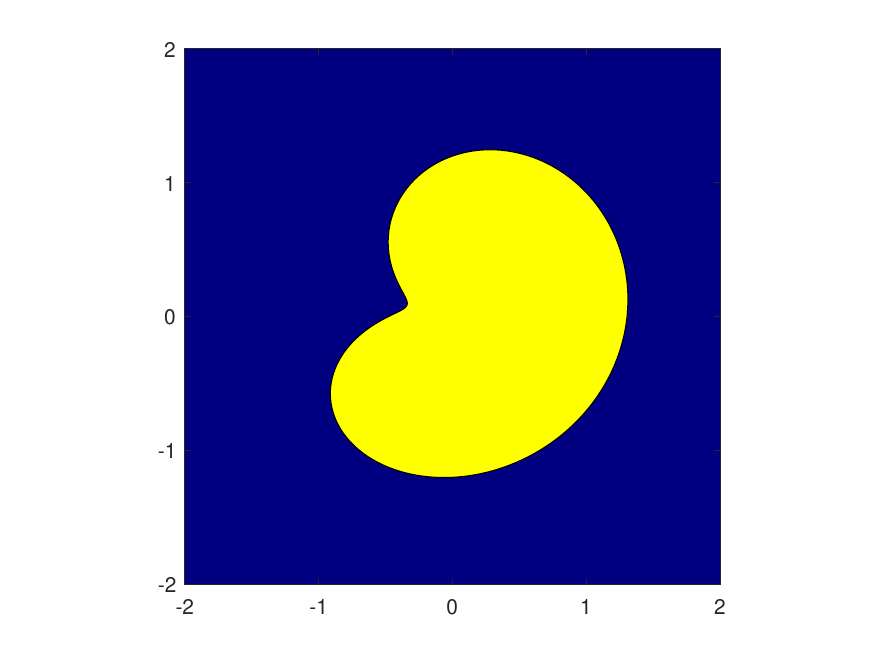}};
			\node[label] at (1.9,-2) {\includegraphics[width=3.4cm, keepaspectratio]{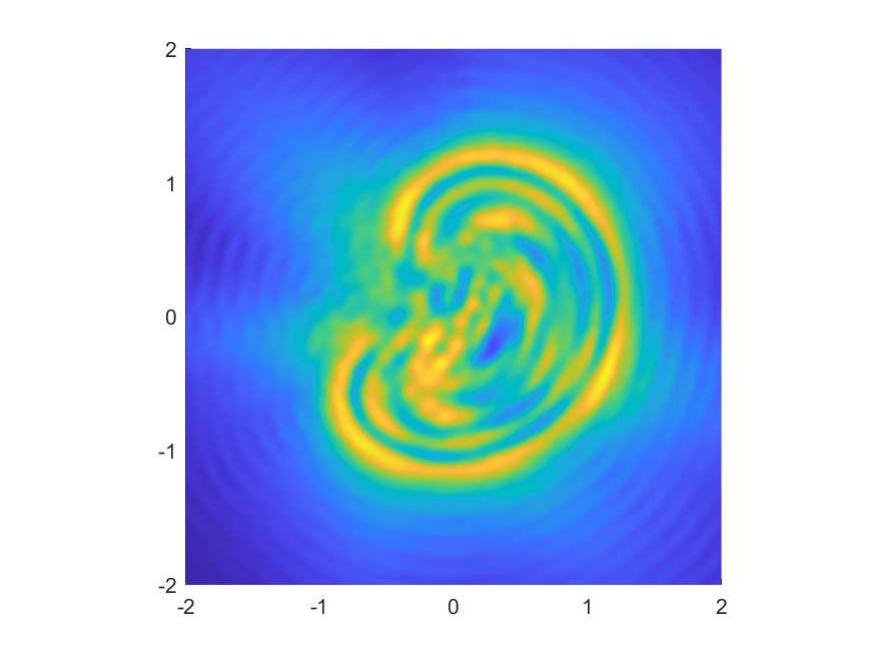}};
			\node[label] at (4.9,-2) {\includegraphics[width=3.4cm, keepaspectratio]{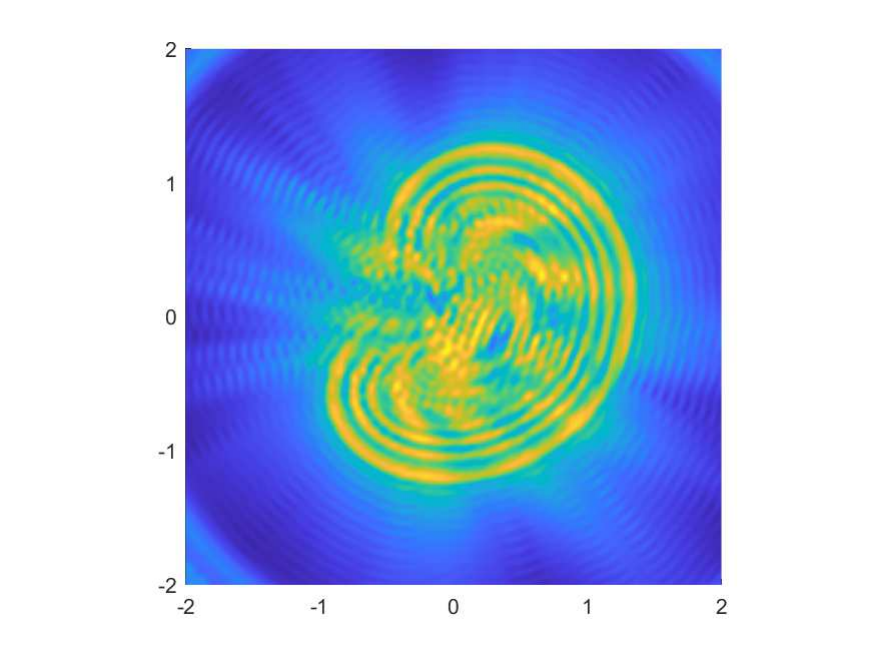}};
			\node[label] at (7.9,-2) {\includegraphics[width=3.4cm, keepaspectratio]{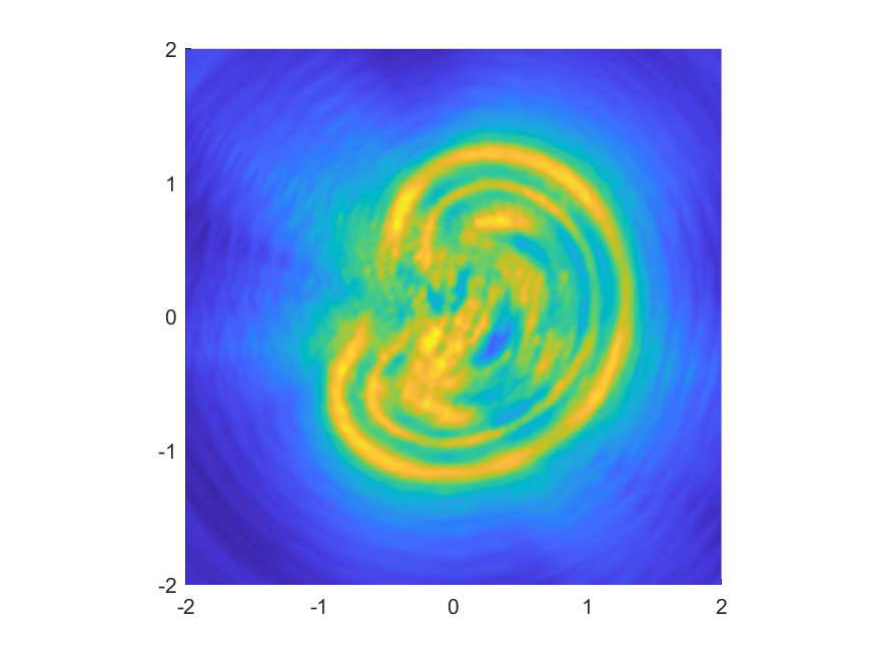}};
			
			\node[label] at (1.9,-4.5) {\includegraphics[width=3.4cm, keepaspectratio]{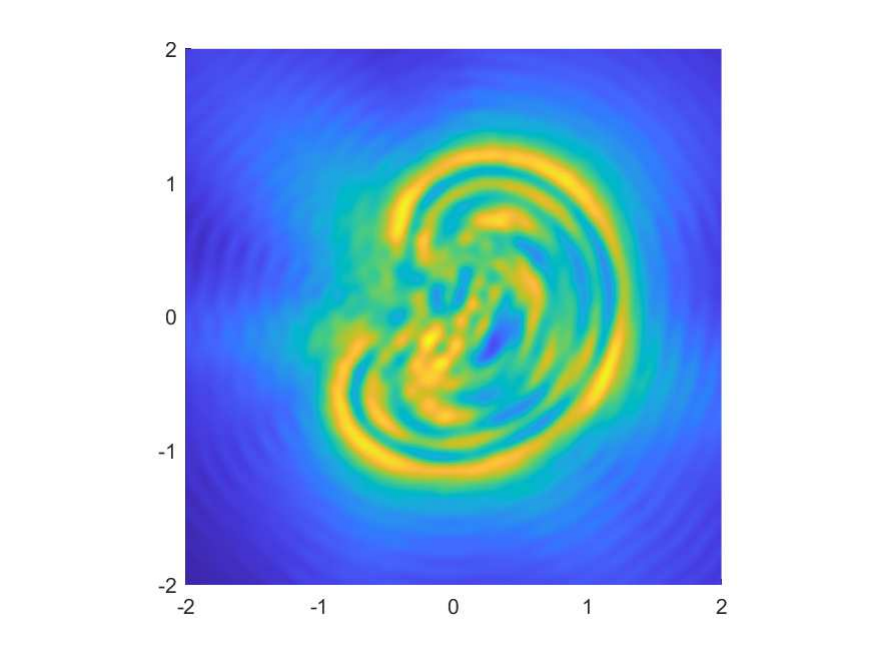}};
			\node[label] at (4.9,-4.5) {\includegraphics[width=3.4cm, keepaspectratio]{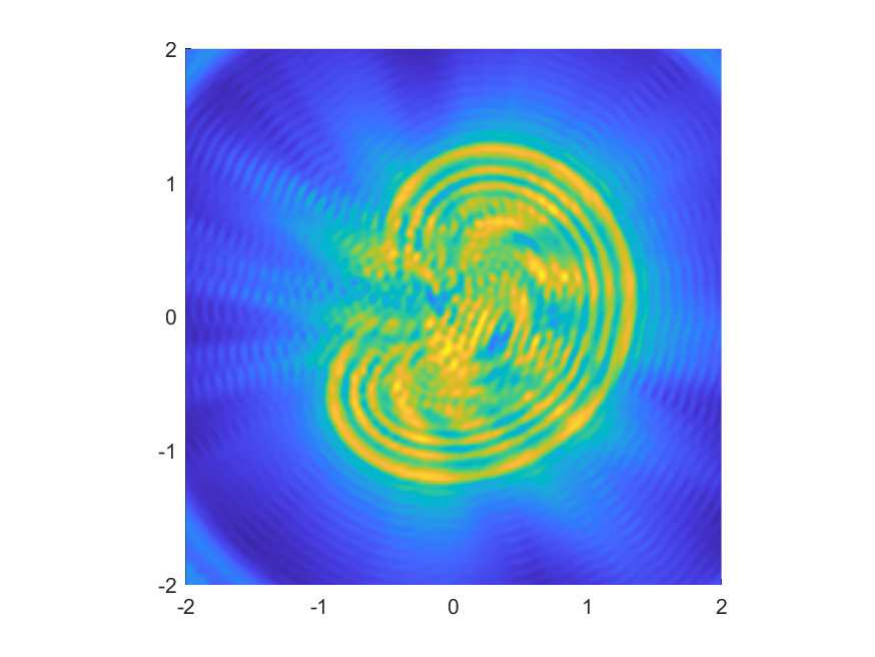}};
			\node[label] at (7.9,-4.5) {\includegraphics[width=3.4cm, keepaspectratio]{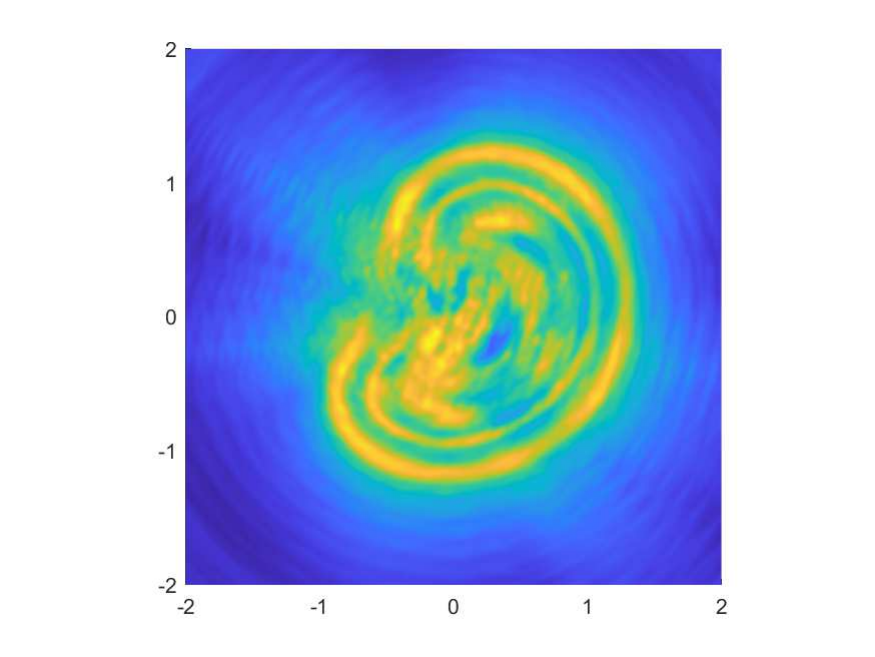}};
			
	\end{tikzpicture}}
	\caption{Reconstruction of apple-shaped penetrable obstacle using indicators $\pmb{I}_P$, $\pmb{I}_S$ and $\pmb{I}_F$ with  \(0\%\) and \(30\%\) noise level.}\label{figure3}
\end{figure}
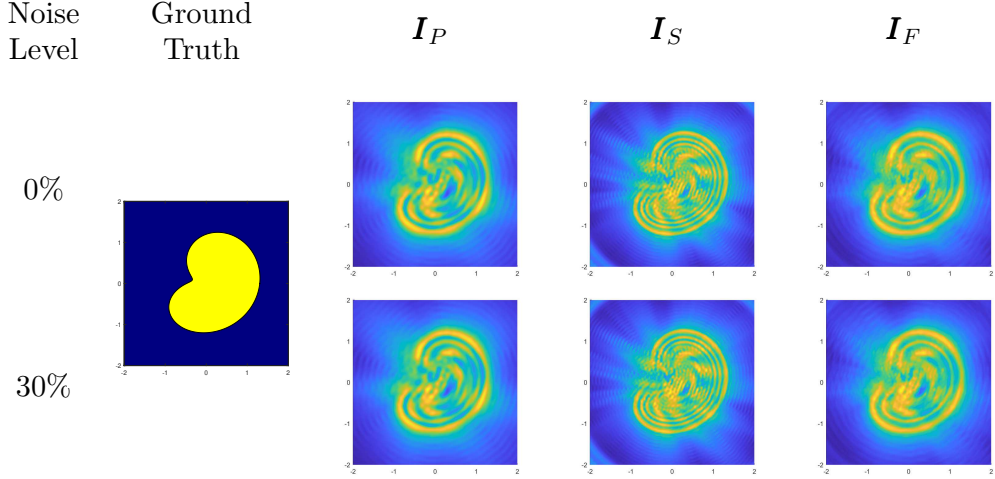

The following parameter settings are used throughout all numerical experiments unless otherwise specified. The mass densities are set to \(\rho_1 = 1\) inside the obstacle \(D\) and \(\rho_2 = 1.5\) in the exterior domain \(\mathbb{R}^2 \setminus \overline{D}\). The Lam\'{e} constants are chosen as \(\lambda_1 = 5.5\), \(\mu_1 = 2.6\) for the obstacle and \(\lambda_2 = 1.5\), \(\mu_2 = 1\) for the background medium. The angular frequency is fixed at \(\omega = 10\pi\), and a relative noise level of \(\delta = 30\%\) is added to the far-field data. The sampling region is the square domain \([-2, 2] \times [-2, 2]\), discretized using a \(361 \times 361\) uniform grid. The boundary \(\partial D\) is discretized with 256 quadrature nodes (i.e., \(n = 128\)).

To evaluate the method, we present reconstruction results for multiple shapes, whose parameterized boundary descriptions are provided in Table \ref{table1}.

\begin{figure}[!htpb]
	\centering 
	\resizebox{0.95\textwidth}{!}{
		\begin{tikzpicture}[
			transform shape, 
			label/.style={font=\normalsize, anchor=center, align=center},
			leftlabel/.style={font=\normalsize, anchor=east, align=right}
			]
			\node[label] at (-3.4,0) {Type};
			\node[label] at (-1,0) {Ground \\ Truth};
			\node[label] at (1.9,0) {$\pmb{I}_P$};
			\node[label] at (4.9,0) {$\pmb{I}_S$};
			\node[label] at (7.9,0) {$\pmb{I}_F$};
			
			\node[label] at (-3.4,-2) {Pear};
			\node[label] at (-3.5,-4.5) {Rectangle};
			\node[label] at (-3.5,-7) {Apple};
			
			\node[label] at (-1,-2) {\includegraphics[width=3.4cm, keepaspectratio]{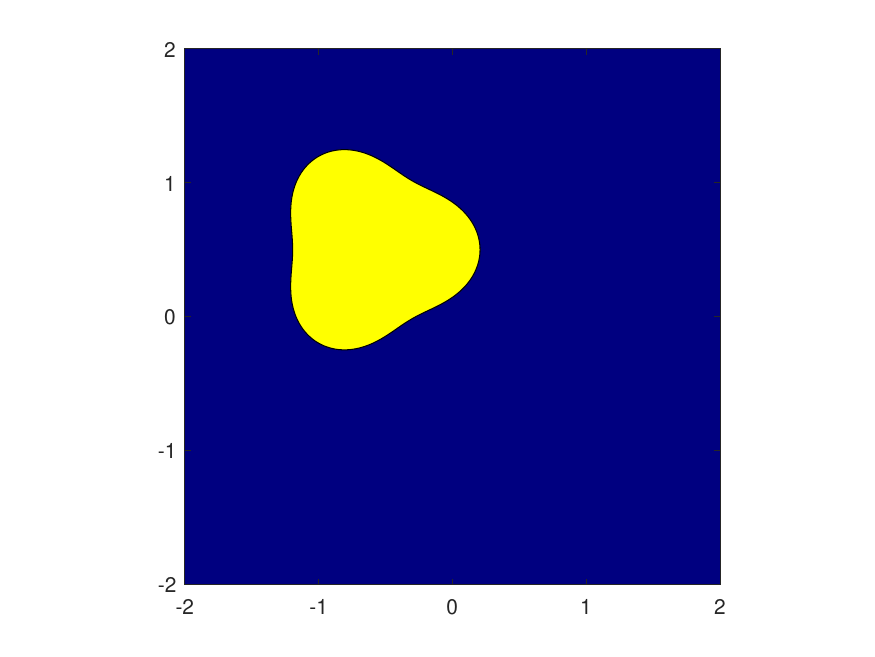}};
			\node[label] at (1.8,-2) {\includegraphics[width=3.4cm, keepaspectratio]{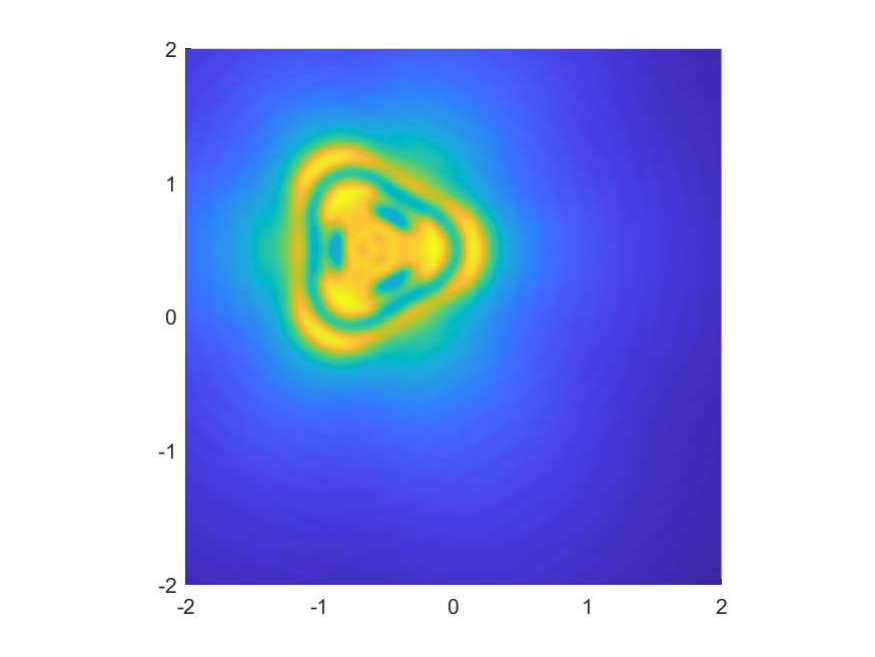}};
			\node[label] at (4.6,-2) {\includegraphics[width=3.4cm, keepaspectratio]{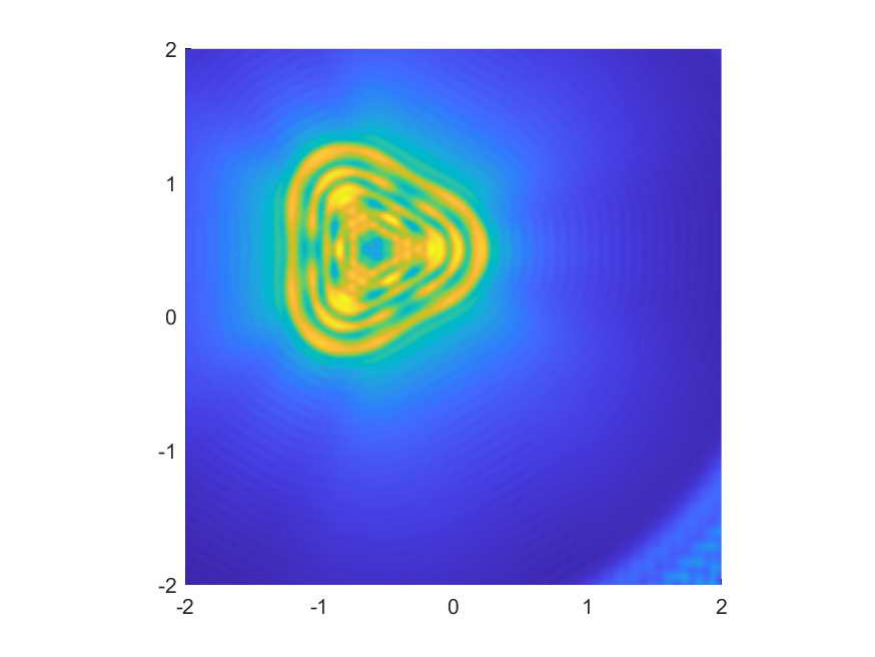}};
			\node[label] at (7.4,-2) {\includegraphics[width=3.4cm, keepaspectratio]{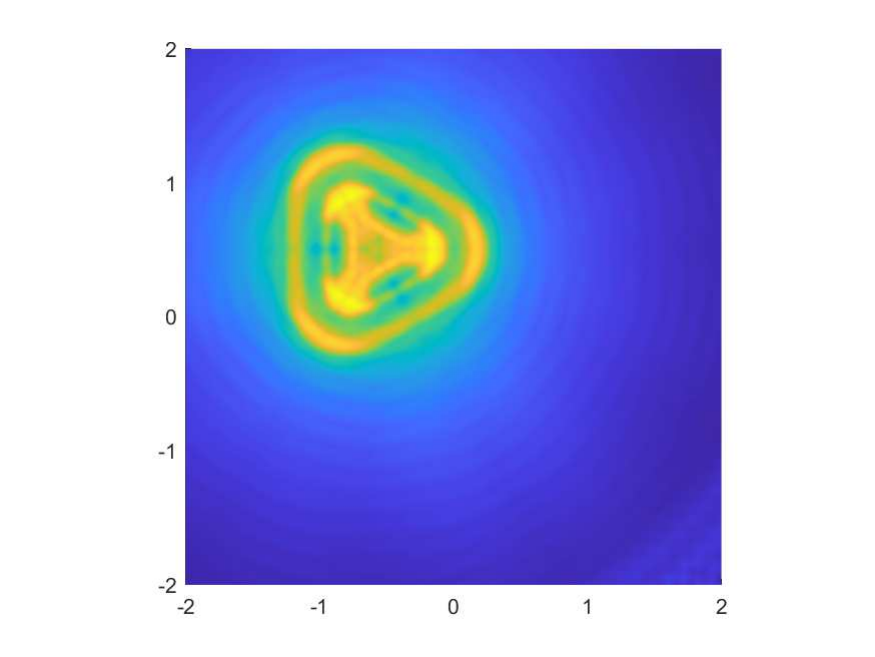}};

			\node[label] at (-1,-4.5) {\includegraphics[width=3.4cm, keepaspectratio]{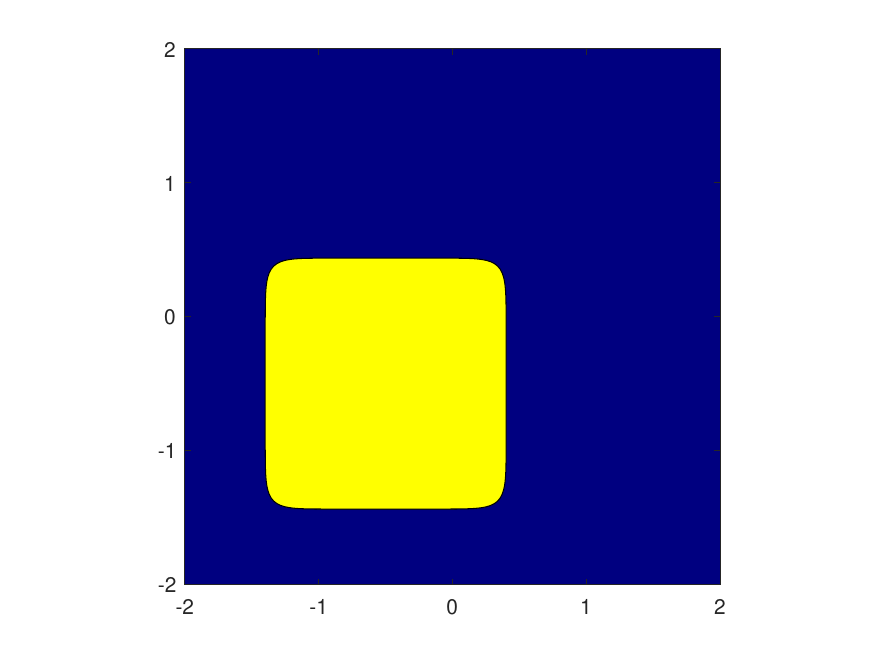}};
			\node[label] at (1.8,-4.5) {\includegraphics[width=3.4cm, keepaspectratio]{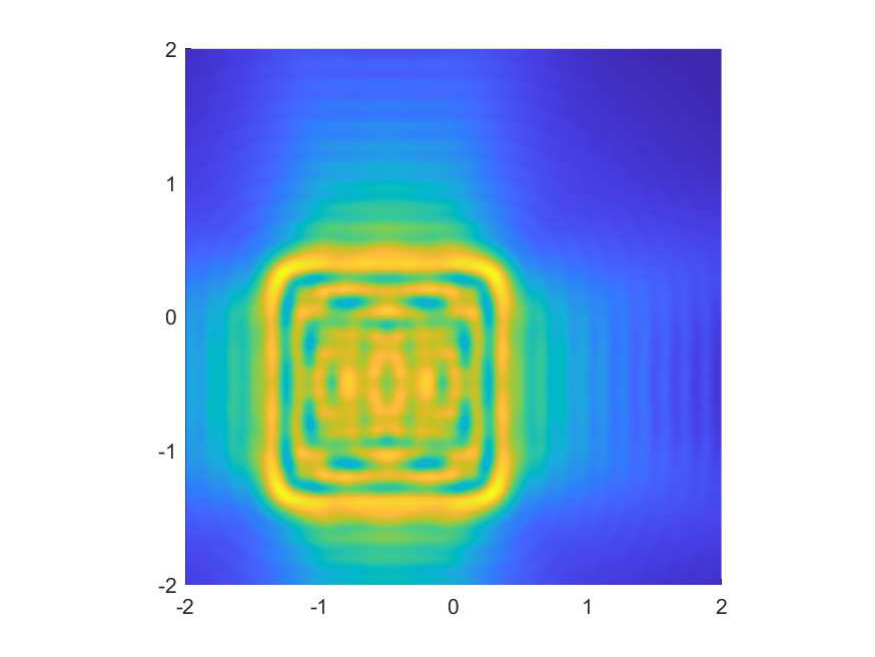}};
			\node[label] at (4.6,-4.5) {\includegraphics[width=3.4cm, keepaspectratio]{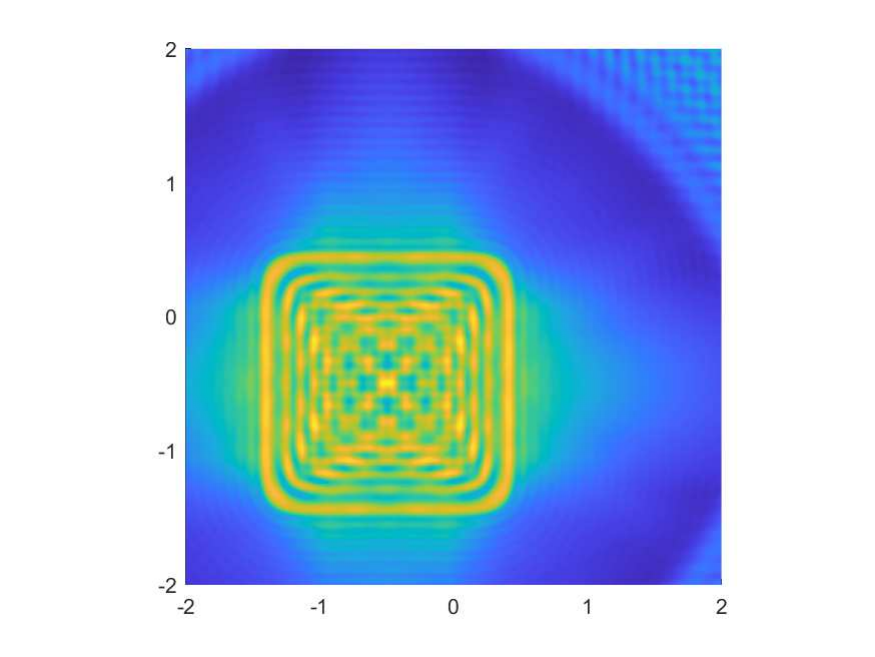}};
			\node[label] at (7.4,-4.5) {\includegraphics[width=3.4cm, keepaspectratio]{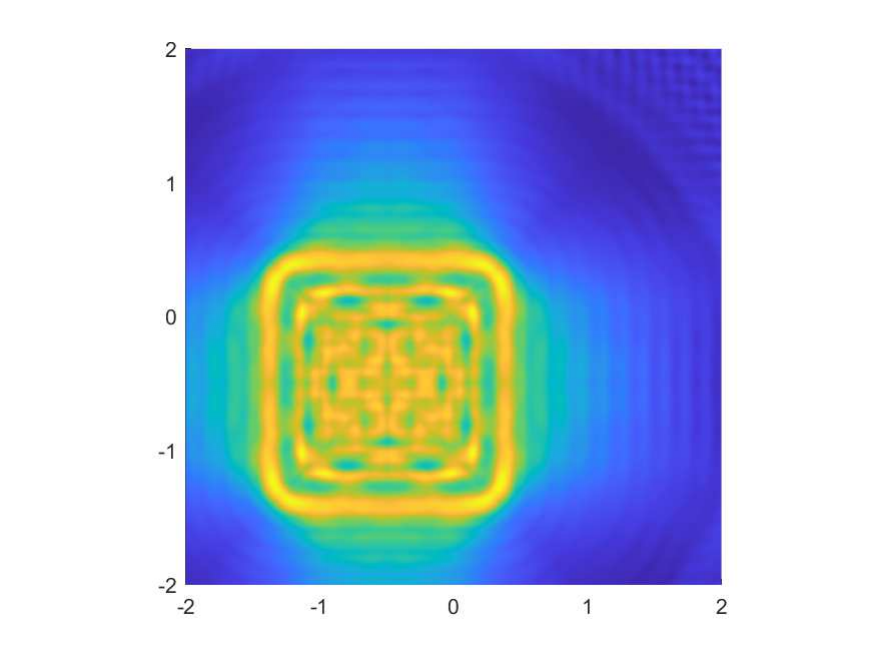}};

			\node[label] at (-1,-7) {\includegraphics[width=3.4cm, keepaspectratio]{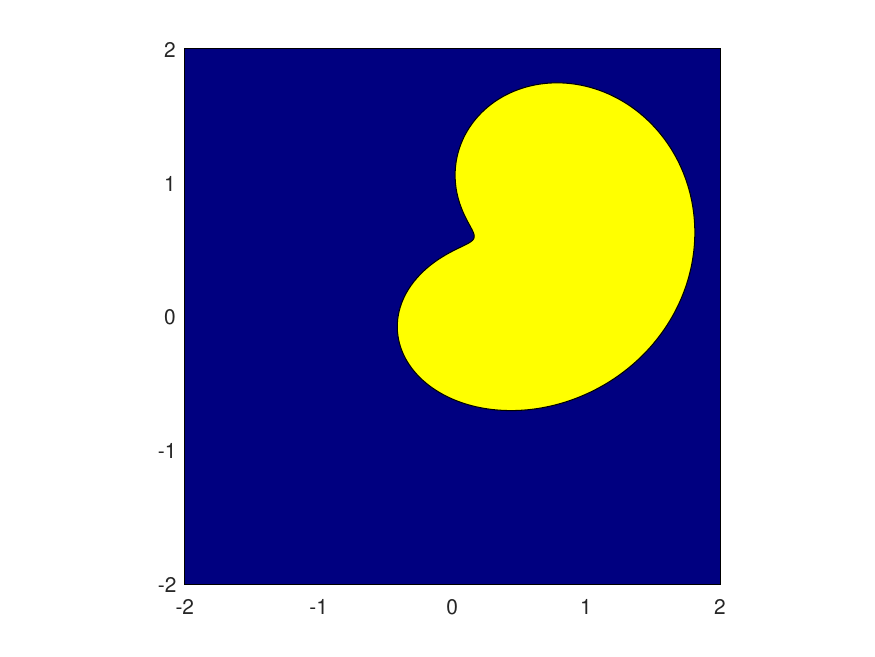}};
			\node[label] at (1.8,-7) {\includegraphics[width=3.4cm, keepaspectratio]{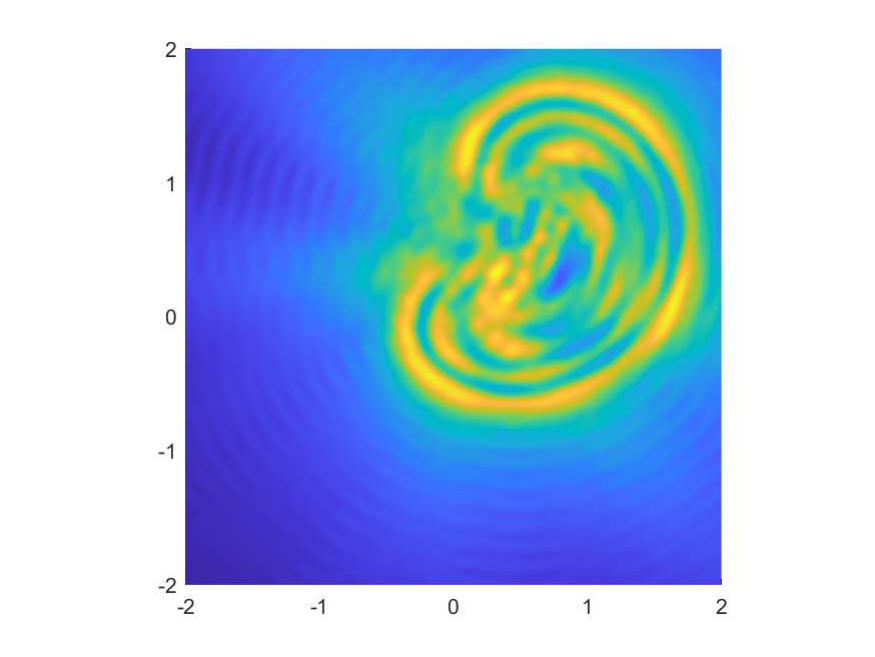}};
			\node[label] at (4.6,-7) {\includegraphics[width=3.4cm, keepaspectratio]{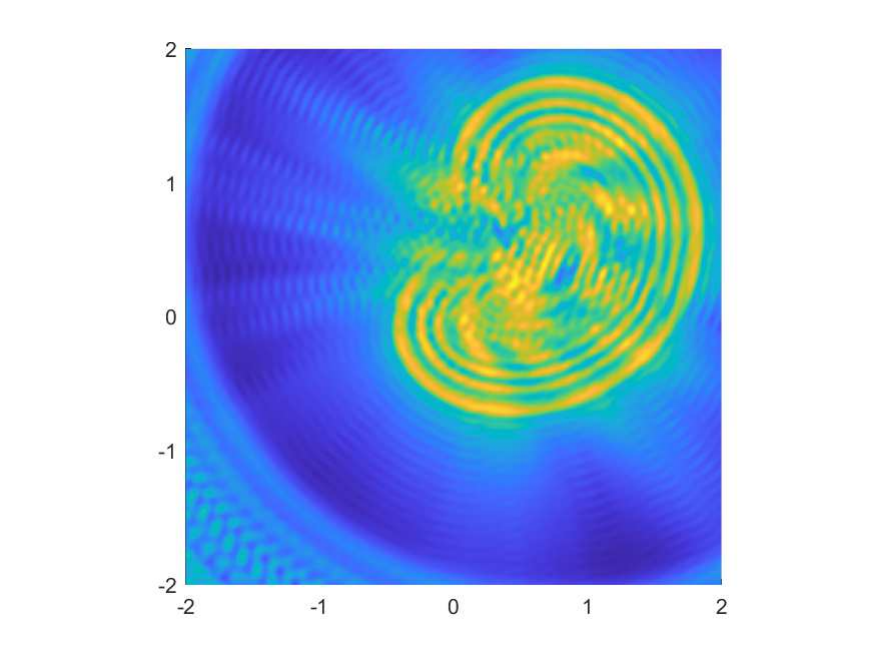}};
			\node[label] at (7.4,-7) {\includegraphics[width=3.4cm, keepaspectratio]{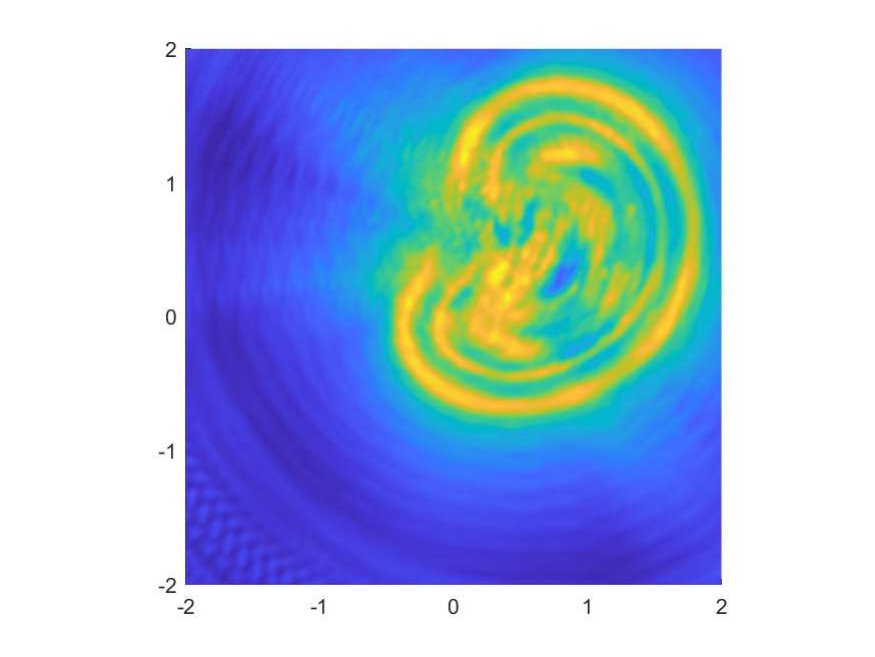}};

	\end{tikzpicture}}
	\caption{Reconstructions of pear-shaped, rectangle-shaped and apple-shaped penetrable obstacles using indicators $\pmb{I}_P$, $\pmb{I}_S$ and $\pmb{I}_F$ with \(30\%\) noise level. }\label{figure4}
\end{figure}
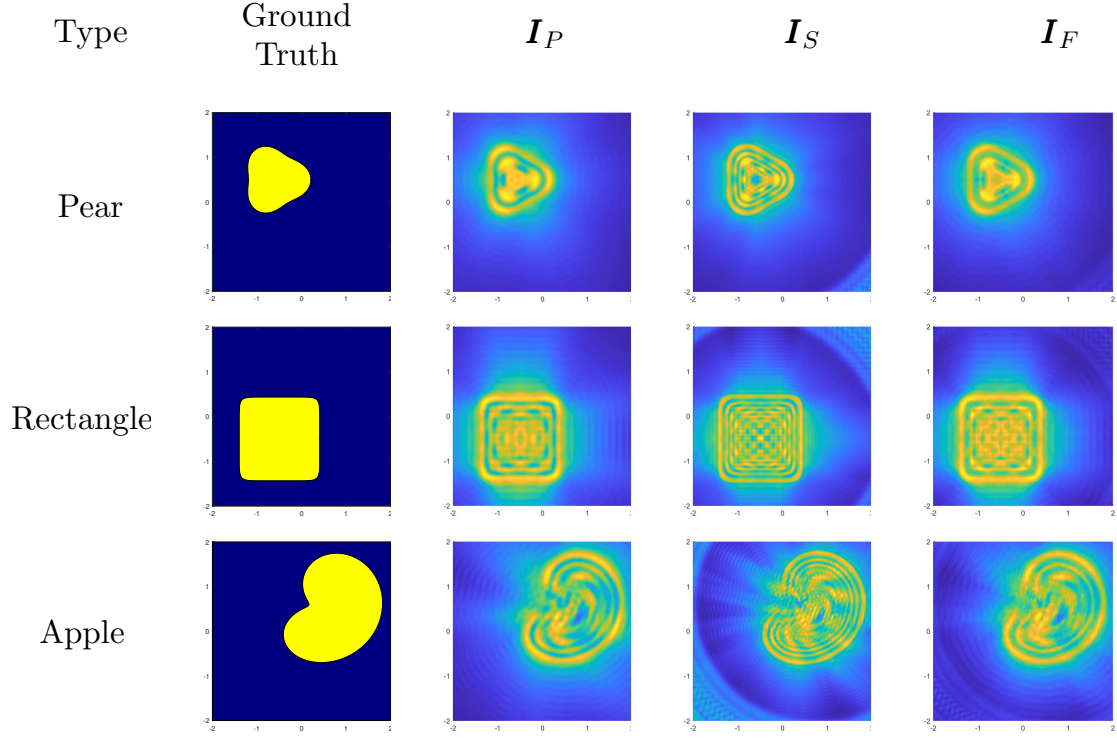

\subsection*{Example 1: Reconstructions with different noise levels.}

We consider the inverse scattering problem of reconstructing penetrable obstacles under different noise levels. Figures~\ref{figure1}--\ref{figure3} display the reconstruction results for flower-shaped, 5-leaf-shaped, kite-shaped and apple-shaped obstacles obtained by the indicators \(\pmb{I}_P\), \(\pmb{I}_S\), and \(\pmb{I}_F\), respectively.

For all four geometries, the indicators produce clear and stable imaging responses at both \(0\%\) and \(30\%\) noise levels. The indicators attain their largest values in a neighborhood of the true boundary and decrease smoothly as the sampling point departs from it. Although the presence of noise slightly blurs the fine features of the imaging function, the overall location and shape of each obstacle remain accurately identifiable, demonstrating the robustness of all three indicators.

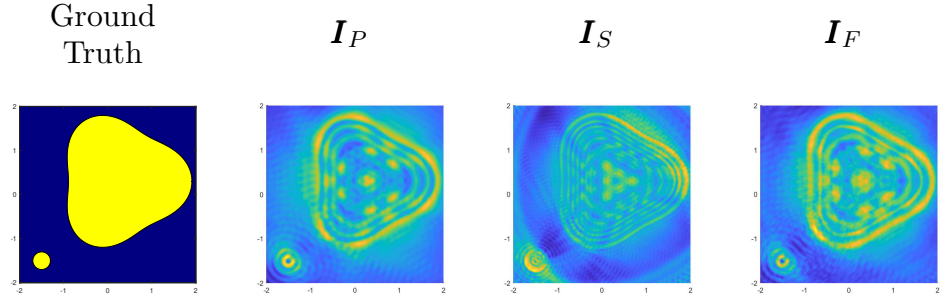
\begin{figure}[!htpb]
	\centering 
	\resizebox{0.85\textwidth}{!}{
		\begin{tikzpicture}[
			transform shape, 
			label/.style={font=\normalsize, anchor=center, align=center},
			leftlabel/.style={font=\normalsize, anchor=east, align=right}
			]
			\node[label] at (0,0){Ground \\ Truth};
			\node[label] at (3,0) {$\pmb{I}_P$};
			\node[label] at (6,0) {$\pmb{I}_S$};
			\node[label] at (9,0)  {$\pmb{I}_F$};
			
			\node[label] at (0,-2) {\includegraphics[width=3.5cm, keepaspectratio]{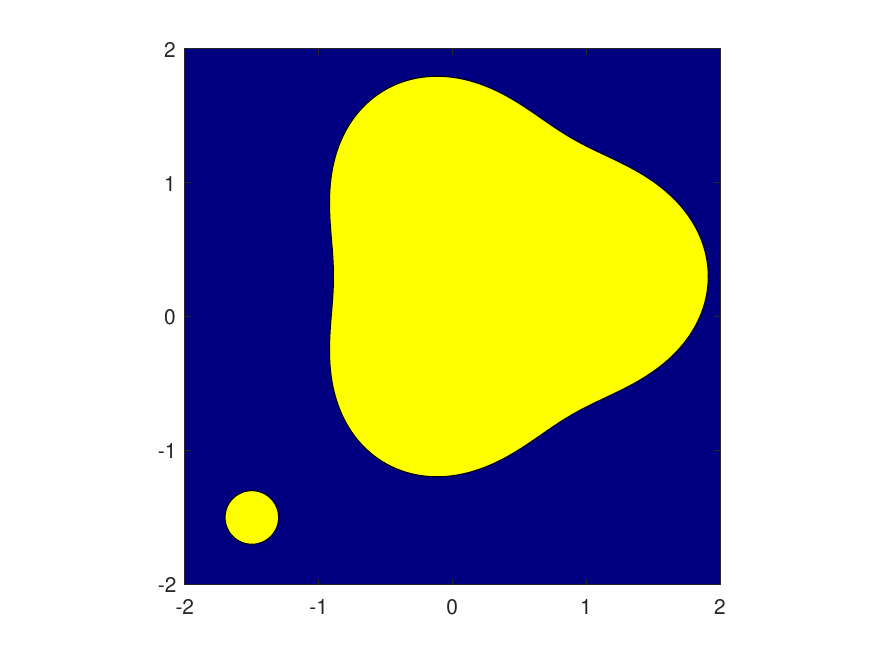}};
			\node[label] at (3,-2) {\includegraphics[width=3.5cm, keepaspectratio]{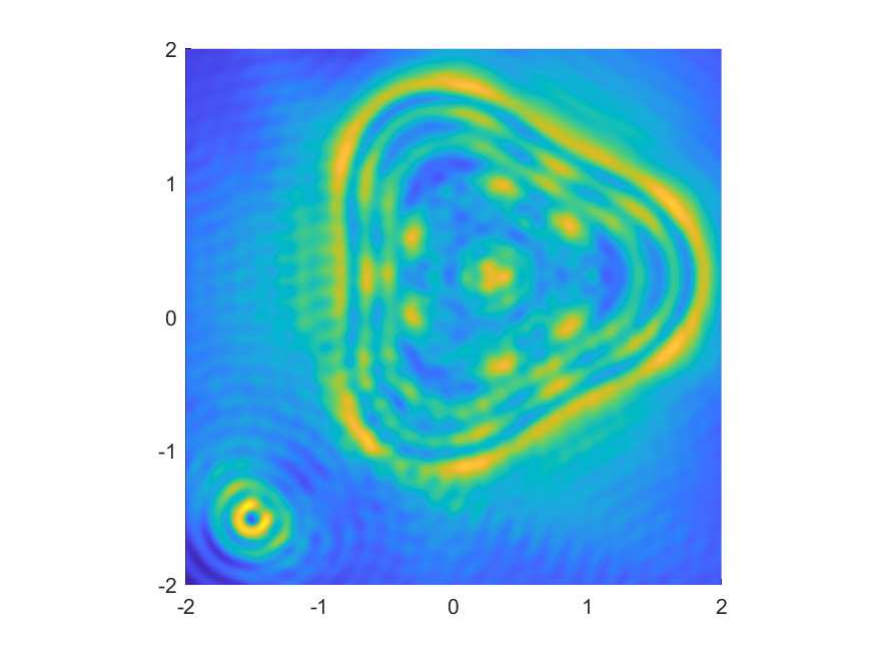}};
			\node[label] at (6,-2) {\includegraphics[width=3.5cm, keepaspectratio]{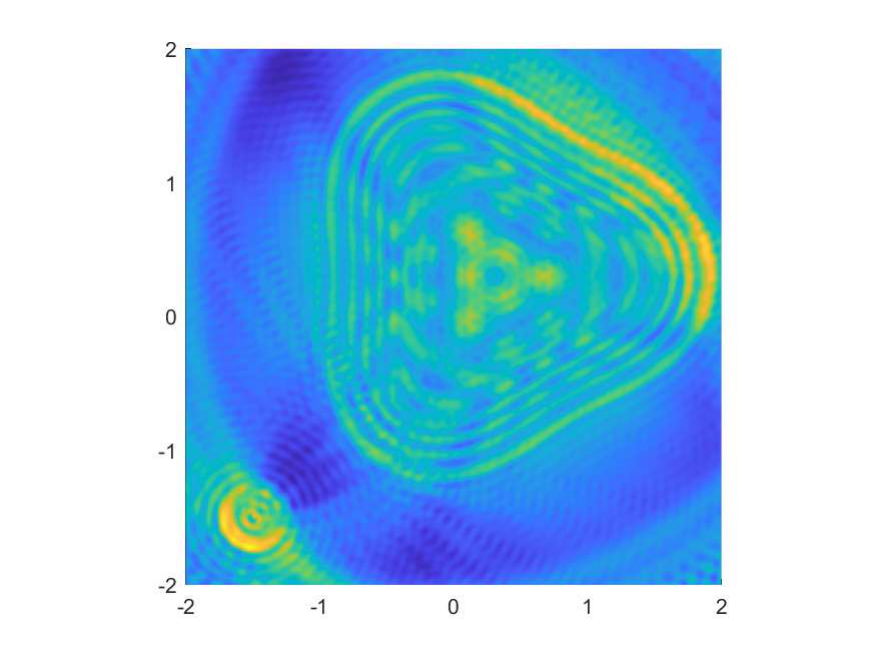}};
			\node[label] at (9,-2) {\includegraphics[width=3.5cm, keepaspectratio]{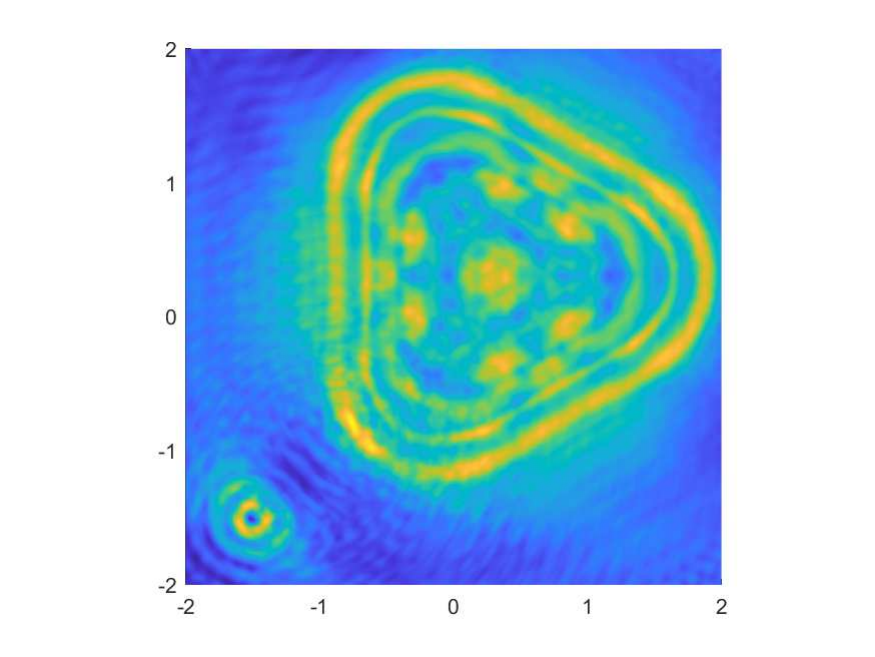}};
			
	\end{tikzpicture}}
	\caption{Reconstruction of the pear-circle obstacle using indicators	$\pmb{I}_P$, $\pmb{I}_S$ and $\pmb{I}_F$ with \(30\%\) noise level.}\label{figure5}
\end{figure}
\begin{figure}[!htpb]
	\centering 
	\resizebox{0.85\textwidth}{!}{
		\begin{tikzpicture}[
			transform shape, 
			label/.style={font=\normalsize, anchor=center, align=center},
			leftlabel/.style={font=\normalsize, anchor=east, align=right}
			]
			\node[label] at (0,0){Ground \\ Truth};
			\node[label] at (3,0) {$\pmb{I}_P$};
			\node[label] at (6,0) {$\pmb{I}_S$};
			\node[label] at (9,0)  {$\pmb{I}_F$};
			
			\node[label] at (0,-2) {\includegraphics[width=3.5cm, keepaspectratio]{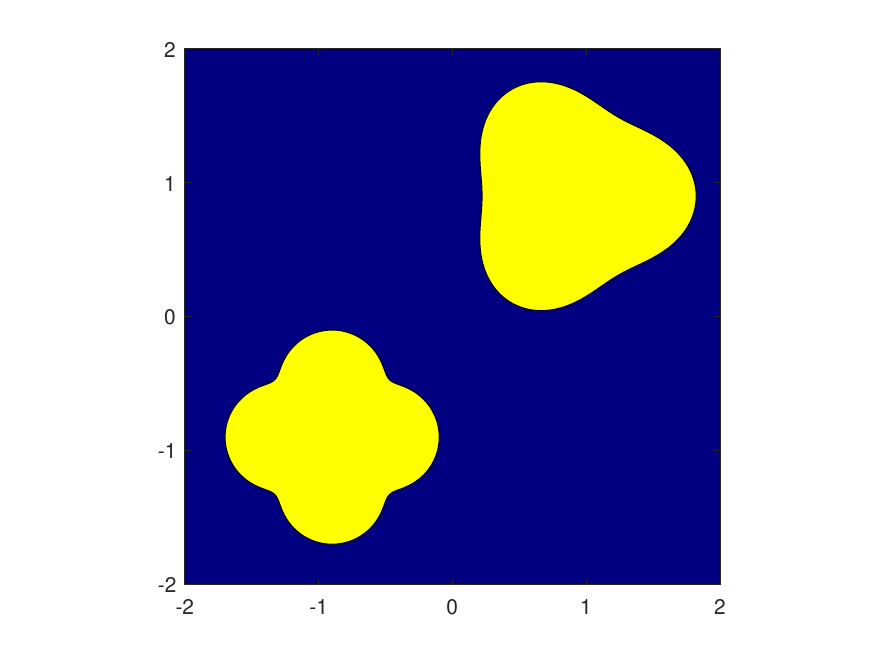}};
			\node[label] at (3,-2) {\includegraphics[width=3.5cm, keepaspectratio]{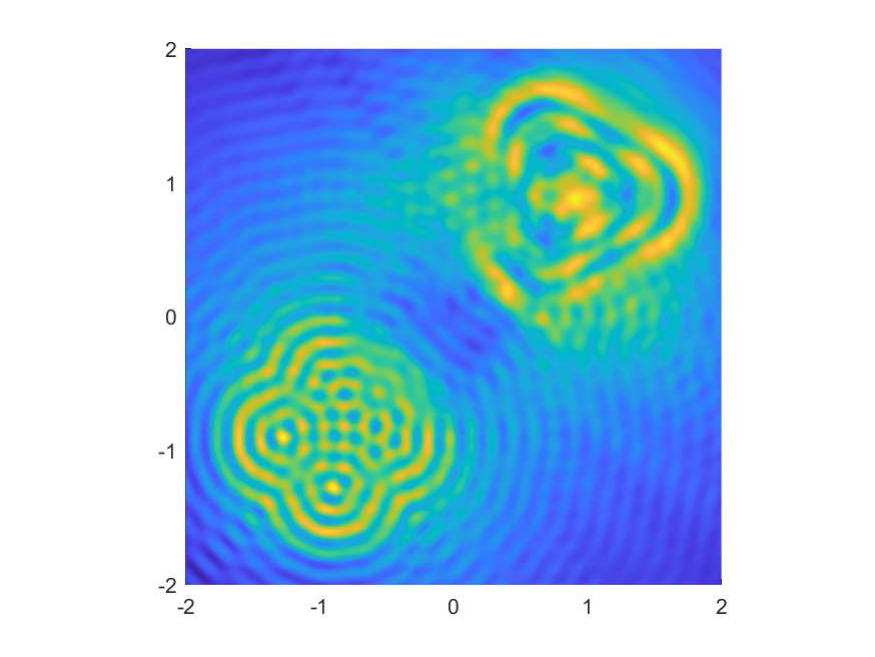}};
			\node[label] at (6,-2) {\includegraphics[width=3.5cm, keepaspectratio]{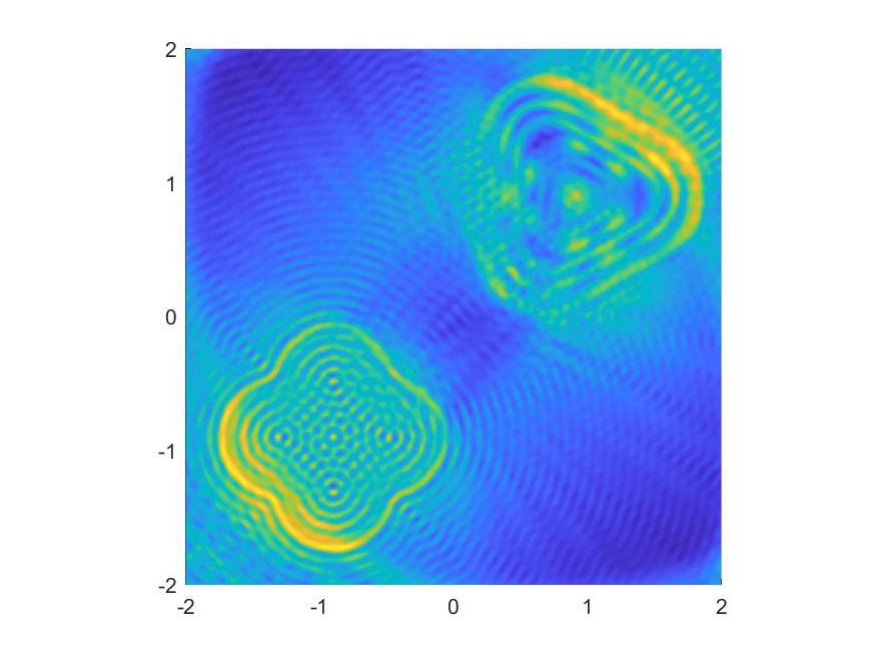}};
			\node[label] at (9,-2) {\includegraphics[width=3.5cm, keepaspectratio]{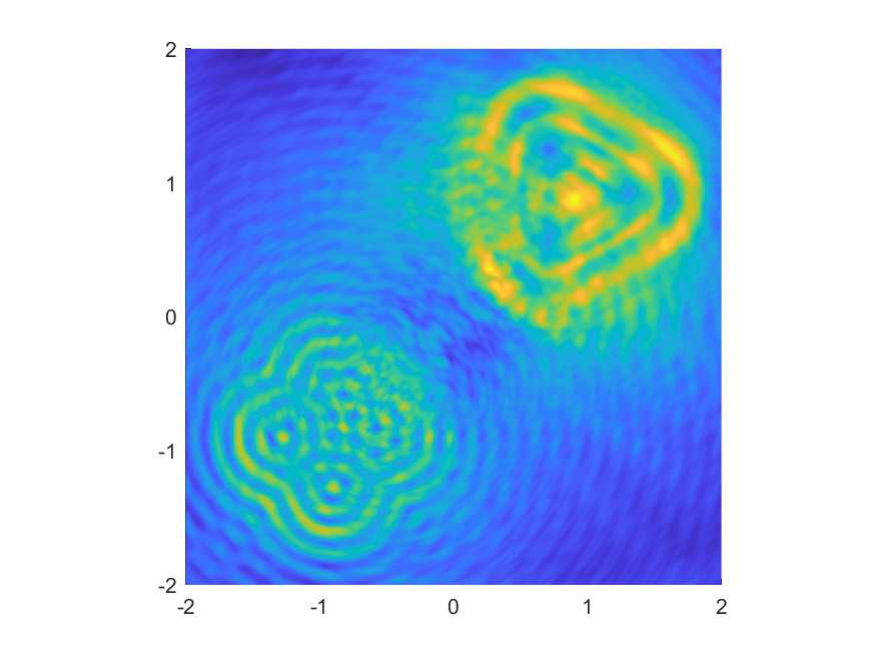}};
			
	\end{tikzpicture}}
	\caption{Reconstruction of the pear-flower obstacle using indicators	$\pmb{I}_P$, $\pmb{I}_S$ and $\pmb{I}_F$ with \(30\%\) noise level.}\label{figure6}
\end{figure}

\subsection*{Example 2: Reconstructions with asymmetric sampling region.}

Figure~\ref{figure4} presents the reconstructions of the pear-shaped, rectangle-shaped and apple-shaped obstacles obtained in an asymmetric sampling region. Even with \(30\%\) noise added to the data, the indicators \(\pmb{I}_P\), \(\pmb{I}_S\), and \(\pmb{I}_F\) successfully recover the approximate location and overall geometry of the obstacles, and each indicator exhibits its characteristic imaging behavior across the three test shapes.

\begin{figure}[!htpb]
	\centering 
	\resizebox{0.85\textwidth}{!}{
		\begin{tikzpicture}[
			transform shape, 
			label/.style={font=\normalsize, anchor=center, align=center},
			leftlabel/.style={font=\normalsize, anchor=east, align=right}
			]
			\node[label] at (0,0){Ground \\ Truth};
			\node[label] at (3,0) {$\pmb{I}_P$};
			\node[label] at (6,0) {$\pmb{I}_S$};
			\node[label] at (9,0)  {$\pmb{I}_F$};
			
			\node[label] at (0,-2) {\includegraphics[width=3.5cm, keepaspectratio]{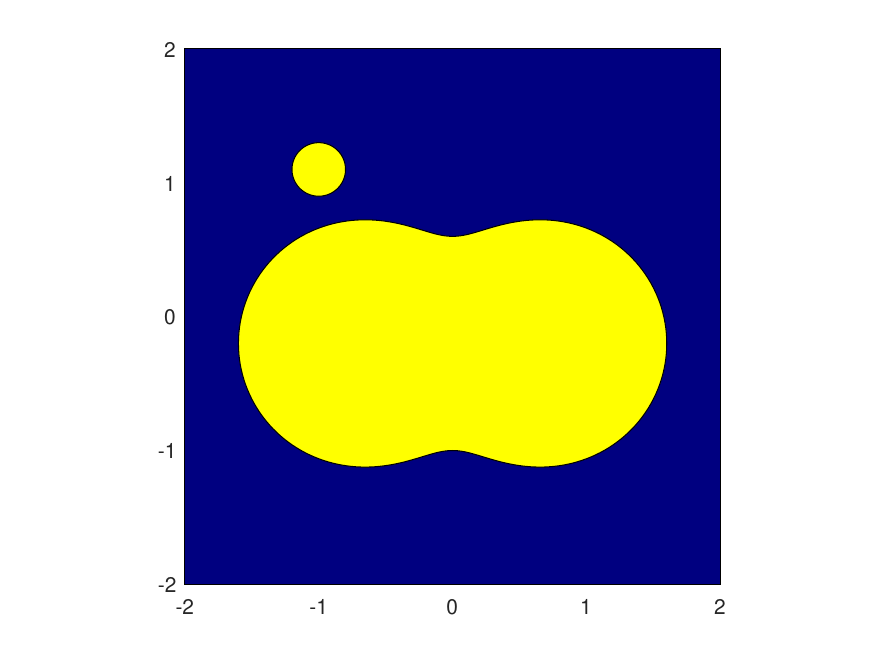}};
			\node[label] at (3,-2) {\includegraphics[width=3.5cm, keepaspectratio]{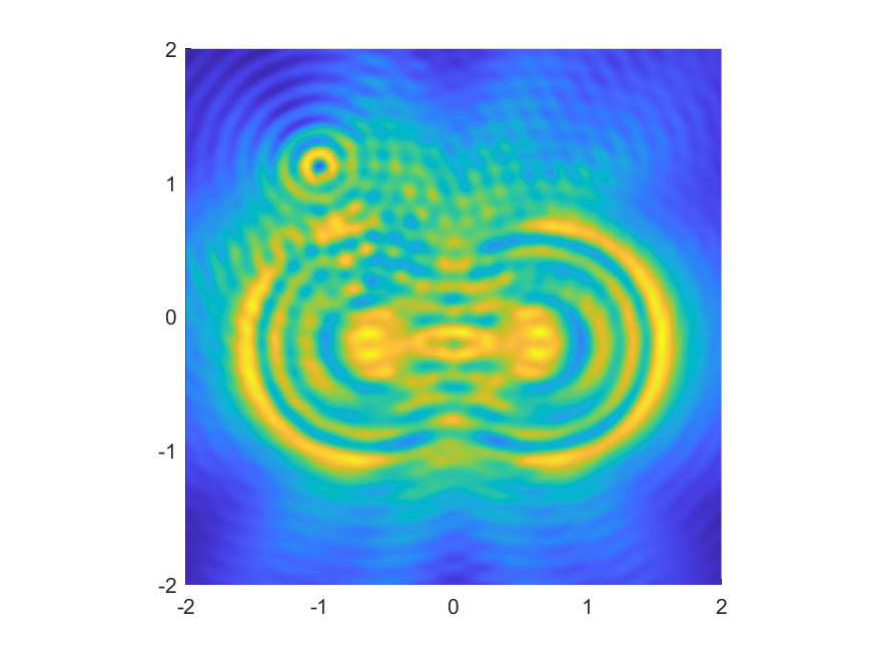}};
			\node[label] at (6,-2) {\includegraphics[width=3.5cm, keepaspectratio]{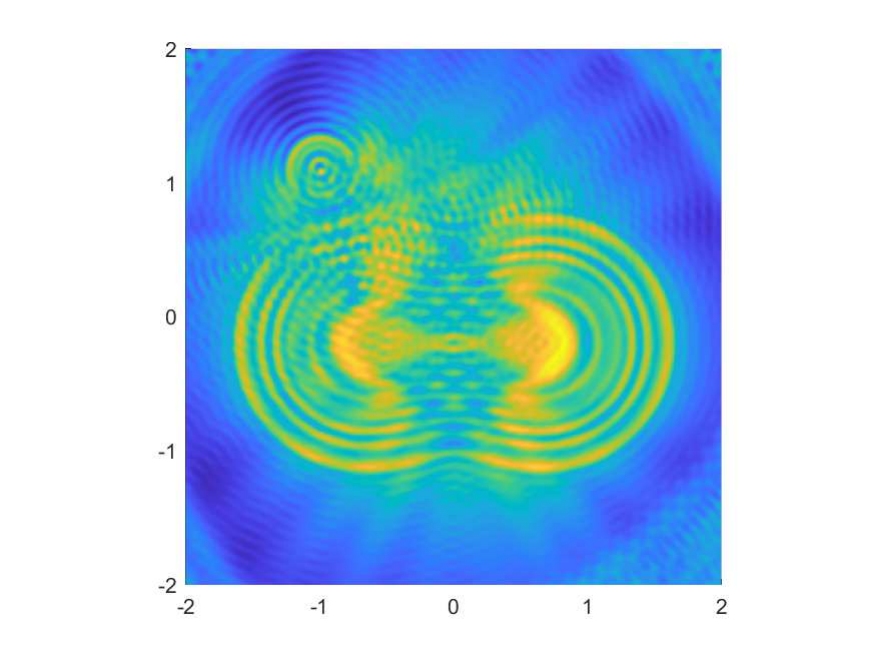}};
			\node[label] at (9,-2) {\includegraphics[width=3.5cm, keepaspectratio]{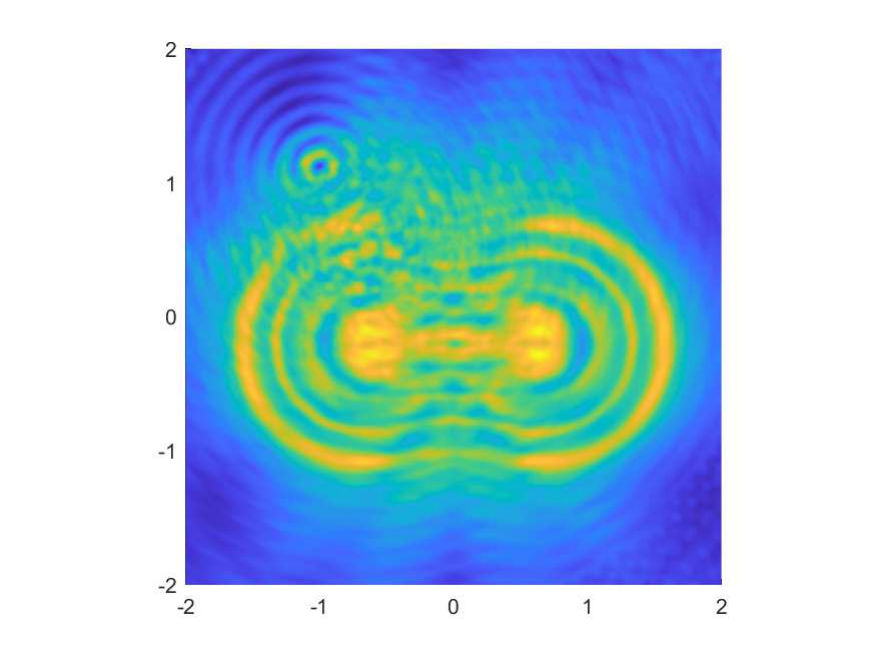}};
			
	\end{tikzpicture}}
	\caption{Reconstruction of the peanut-circle obstacle using indicators	$\pmb{I}_P$, $\pmb{I}_S$ and $\pmb{I}_F$ with \(30\%\) noise level.}\label{figure7}
\end{figure}
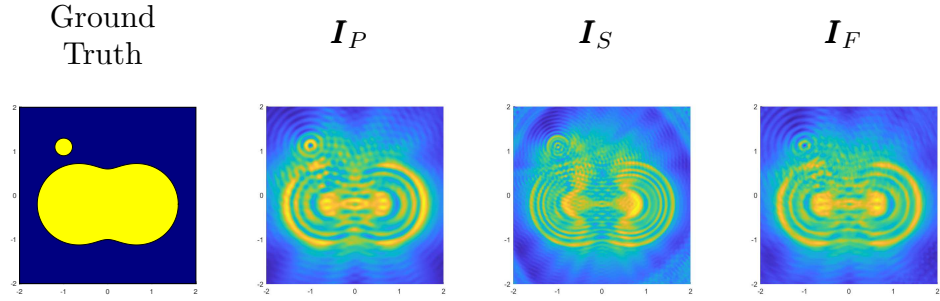

\begin{figure}[!htpb]
	\centering 
	\resizebox{0.85\textwidth}{!}{
		\begin{tikzpicture}[
			transform shape, 
			label/.style={font=\normalsize, anchor=center, align=center},
			leftlabel/.style={font=\normalsize, anchor=east, align=right}
			]
			\node[label] at (-3,0) {Incident\\Range};
			\node[label] at (0,0) {$\pmb{I}_{P,L}$};
			\node[label] at (3,0){$\pmb{I}_{S,L}$};
			\node[label] at (6,0) {$\pmb{I}_{F,L}$};
			
			\node[label] at (-3,-2) {[0,$\pi$]};
			\node[label] at (-3,-4.5) {[$\pi$,$2\pi$]};
			
			\node[label] at (0,-2) {\includegraphics[width=3.5cm, keepaspectratio]{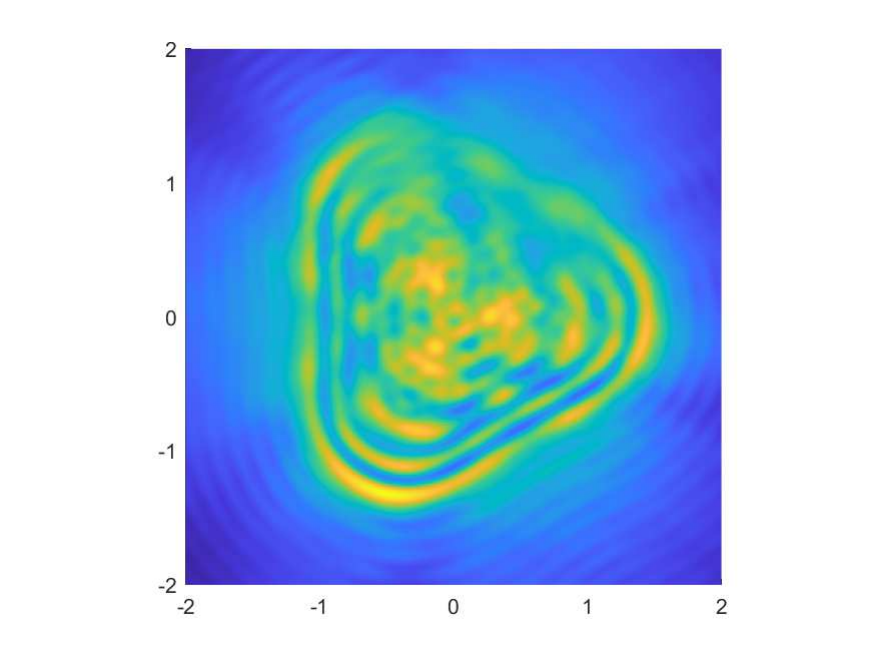}};
			\node[label] at (0,-4.5) {\includegraphics[width=3.5cm, keepaspectratio]{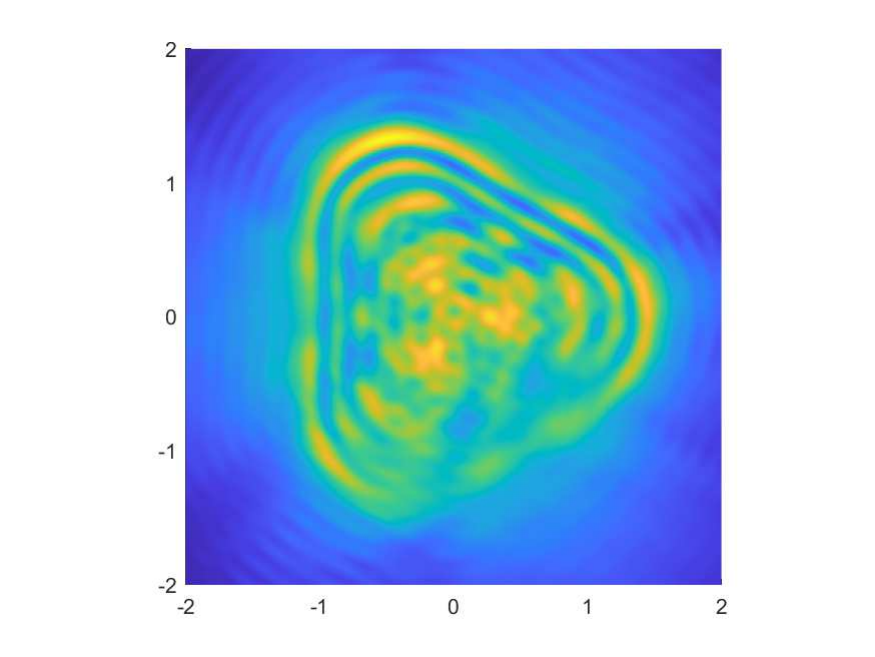}};
			
			\node[label] at (3,-2) {\includegraphics[width=3.5cm, keepaspectratio]{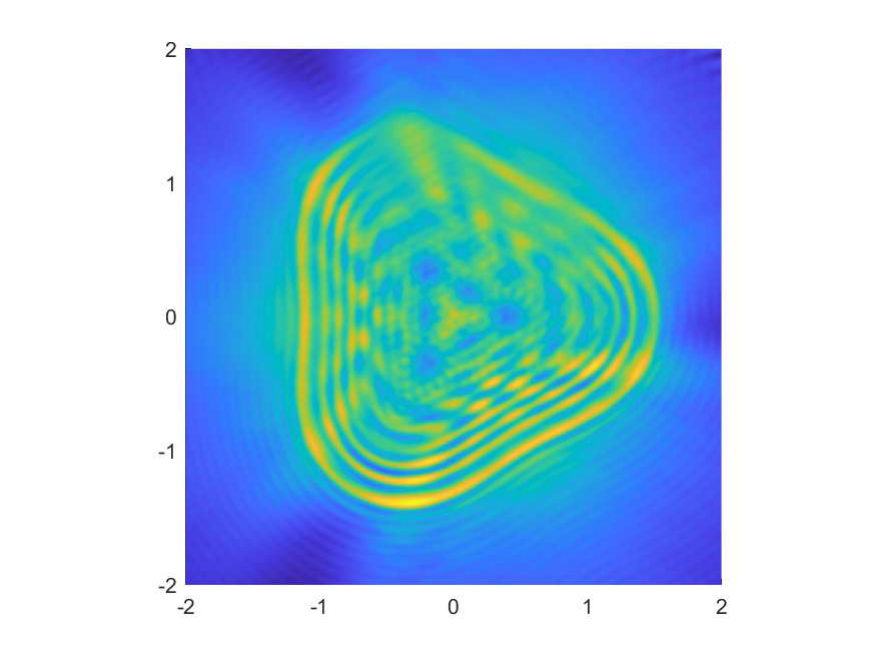}};
			\node[label] at (6,-2) {\includegraphics[width=3.5cm, keepaspectratio]{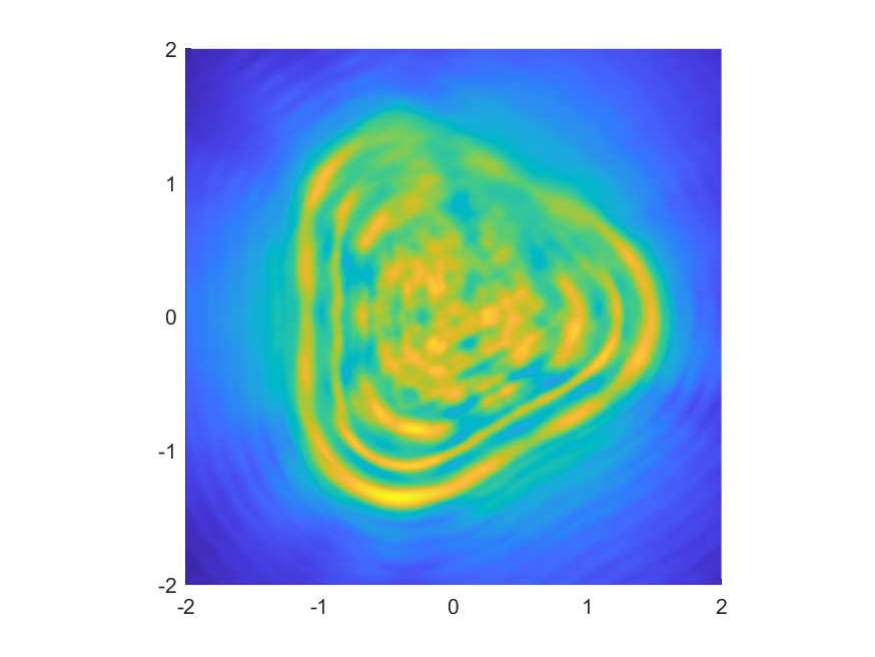}};
			
			\node[label] at (3,-4.5) {\includegraphics[width=3.5cm, keepaspectratio]{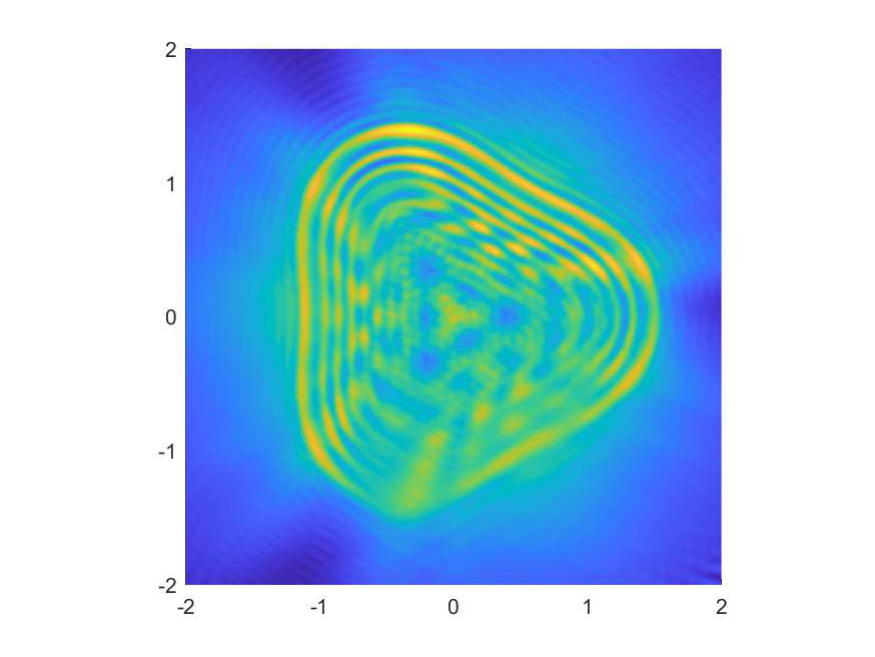}};
			\node[label] at (6,-4.5) {\includegraphics[width=3.5cm, keepaspectratio]{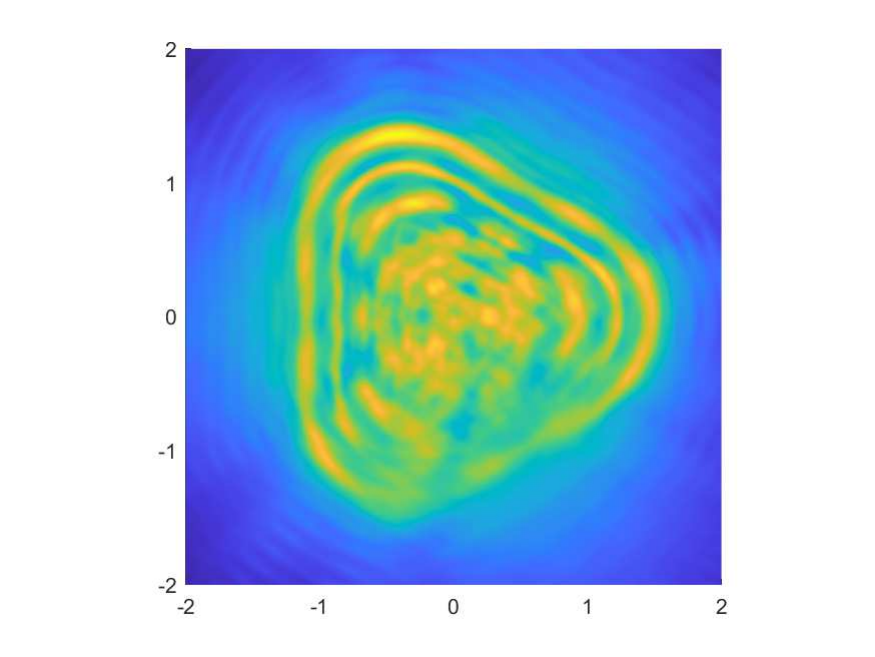}};

	\end{tikzpicture}}
	\caption{Limited-aperture reconstruction of the pear-shaped penetrable obstacle using indicators $\pmb{I}_{P,L}$, $\pmb{I}_{S,L}$ and $\pmb{I}_{F,L}$ with \(30\%\) noise level under half-circle incident direction. }\label{figure9}
\end{figure}
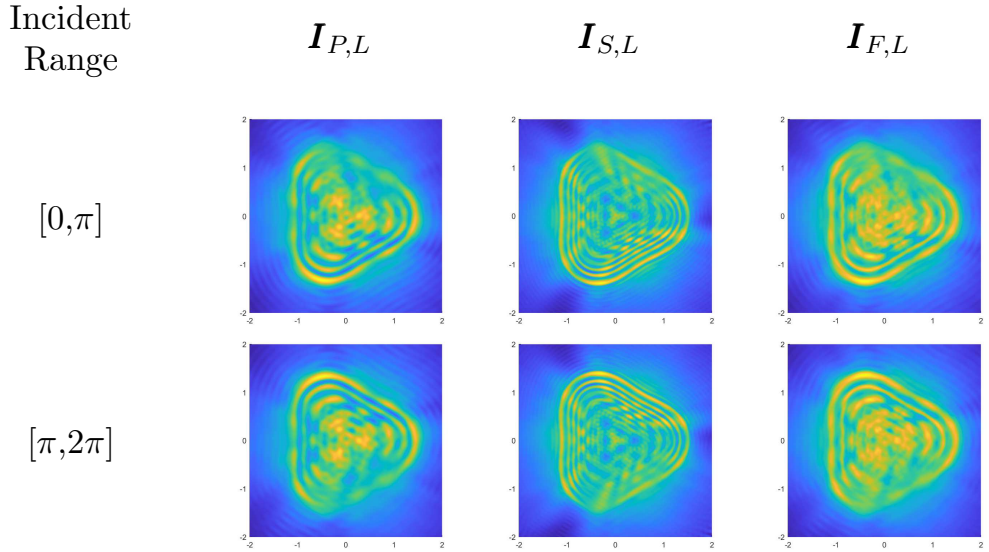

\subsection*{Example 3: Reconstructions with two disconnected obstacles.}

We now consider configurations involving two disconnected obstacles. Figure~\ref{figure5} shows the reconstruction of a pear-shaped penetrable obstacle together with a small rigid circle located at a large distance. Owing to the clear separation between the two components, all indicators \(\pmb{I}_P\), \(\pmb{I}_S\), and \(\pmb{I}_F\) successfully identify both obstacles and capture their individual shapes.

Figure~\ref{figure6} displays the case of a pear-shaped penetrable obstacle paired with a flower-shaped rigid obstacle of comparable size. The imaging results clearly distinguish the two components, and the main geometric features of each obstacle are well reproduced.

In contrast, Figure~\ref{figure7} illustrates a peanut-shaped penetrable obstacle and a rigid circle positioned in close proximity. Although the indicators still reveal the presence and approximate location of both obstacles, the reconstruction in the region where the two components approach each other becomes less accurate, and some boundary details are blurred or suppressed. Away from the interaction zone, however, the obstacle boundaries remain well resolved.

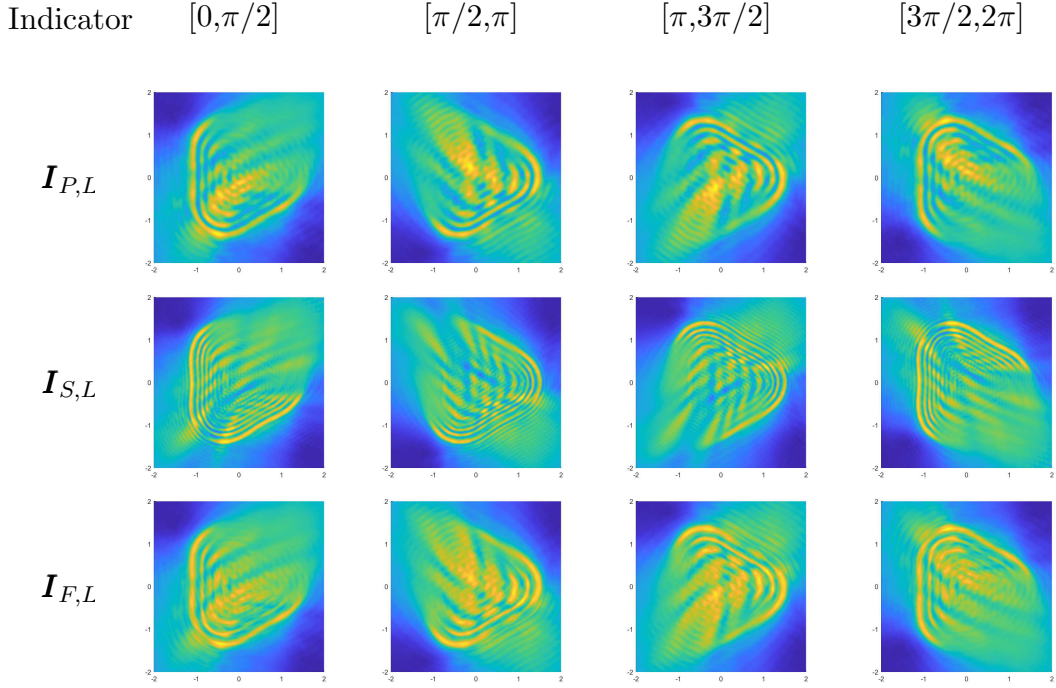
\begin{figure}[!htpb]
	\centering 
	\resizebox{0.9\textwidth}{!}{
		\begin{tikzpicture}[
			transform shape, 
			label/.style={font=\normalsize, anchor=center, align=center},
			leftlabel/.style={font=\normalsize, anchor=east, align=right}
			]
			
			\node[label] at (-1,0) {[0,$\pi/2$]};
			\node[label] at (1.9,0) {[$\pi/2$,$\pi$]};
			
			\node[label] at (4.9,0) {[$\pi$,$3\pi/2$]};
			\node[label] at (7.9,0) {[$3\pi/2$,$2\pi$]};

			\node[label] at (-3,0) {Indicator};
			\node[label] at (-3,-2) {$\pmb{I}_{P,L}$};
			\node[label] at (-3,-4.5){$\pmb{I}_{S,L}$};
			\node[label] at (-3,-7) {$\pmb{I}_{F,L}$};
			
			\node[label] at (-1,-2) {\includegraphics[width=3.4cm, keepaspectratio]{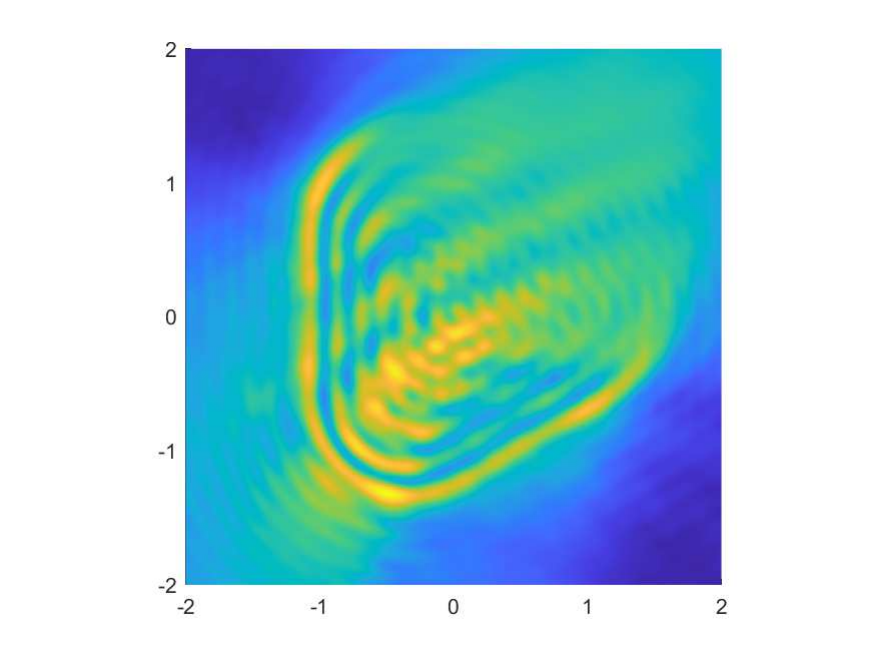}};
			\node[label] at (1.9,-2) {\includegraphics[width=3.4cm, keepaspectratio]{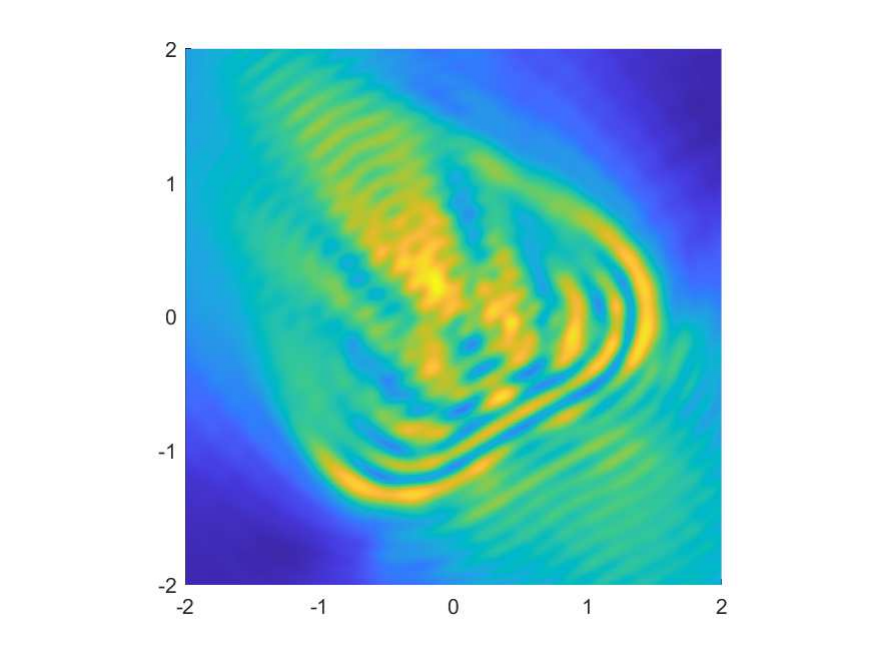}};
			\node[label] at (4.9,-2) {\includegraphics[width=3.4cm, keepaspectratio]{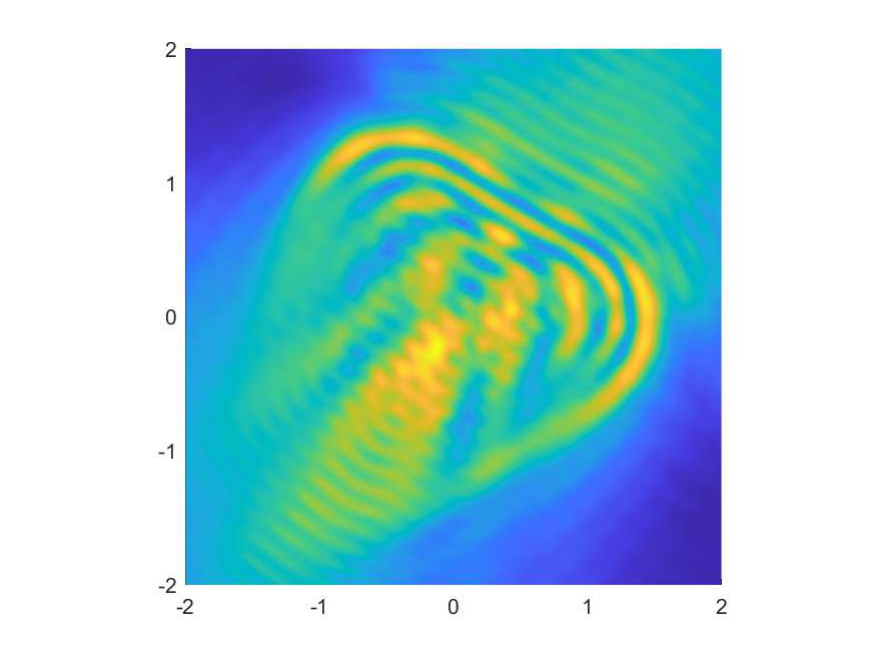}};
			\node[label] at (7.9,-2) {\includegraphics[width=3.4cm, keepaspectratio]{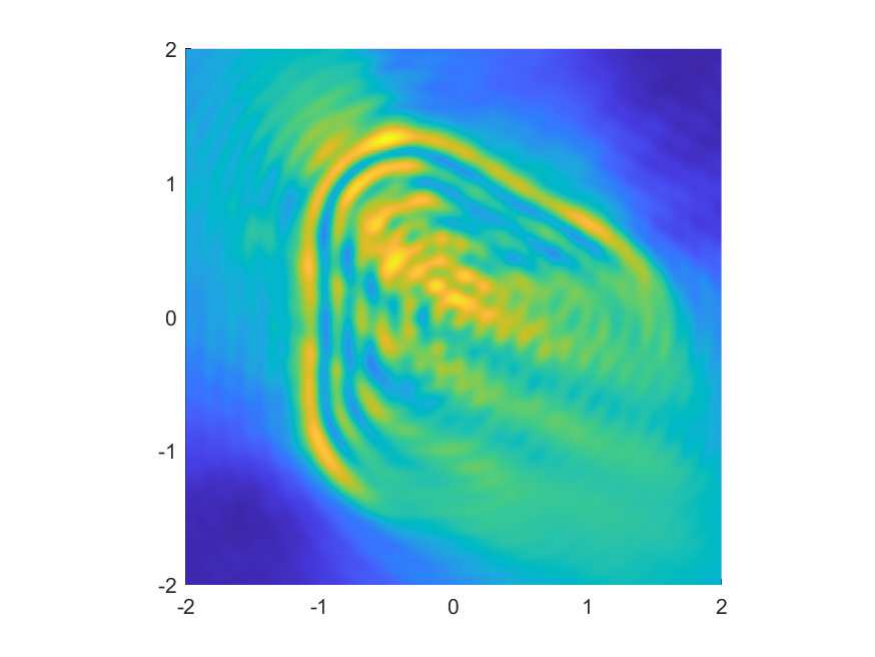}};

			\node[label] at (-1,-4.5) {\includegraphics[width=3.4cm, keepaspectratio]{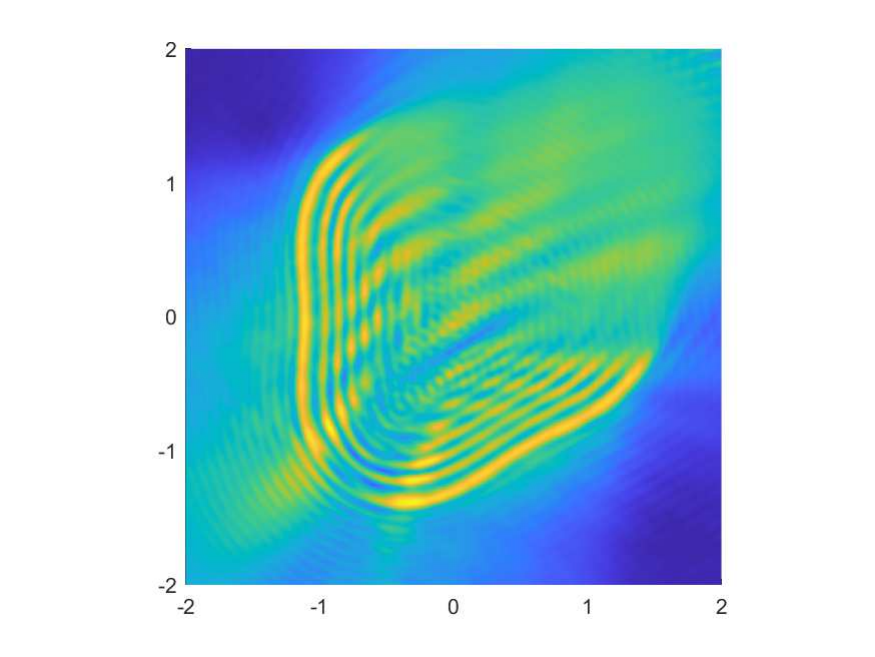}};
			\node[label] at (1.9,-4.5) {\includegraphics[width=3.4cm, keepaspectratio]{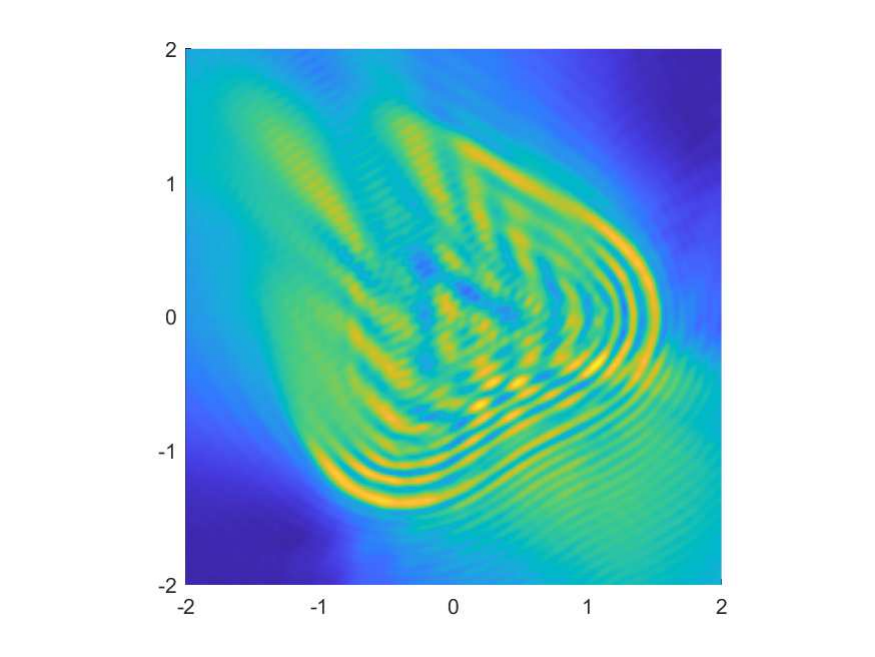}};
			\node[label] at (4.9,-4.5) {\includegraphics[width=3.4cm, keepaspectratio]{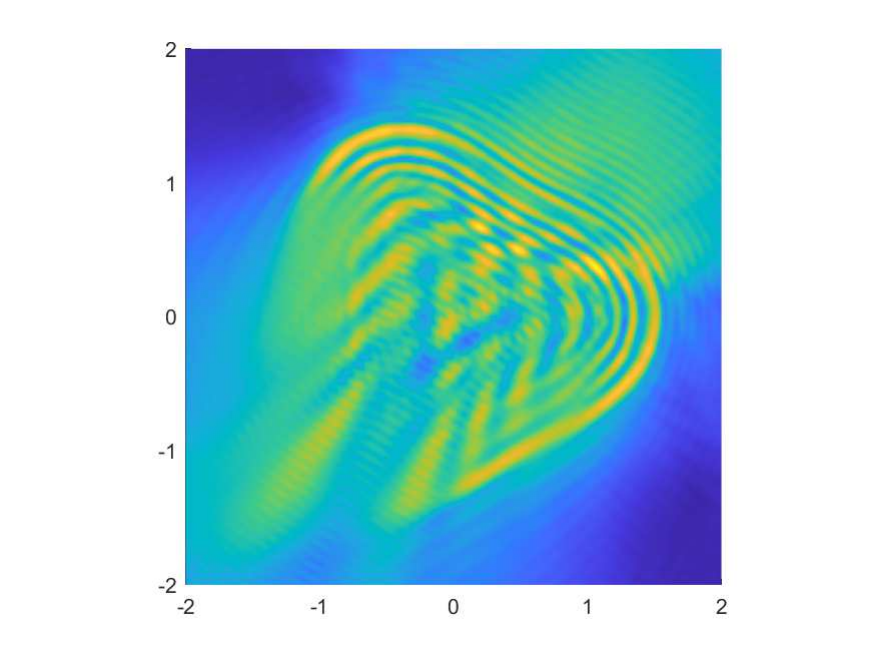}};
			\node[label] at (7.9,-4.5) {\includegraphics[width=3.4cm, keepaspectratio]{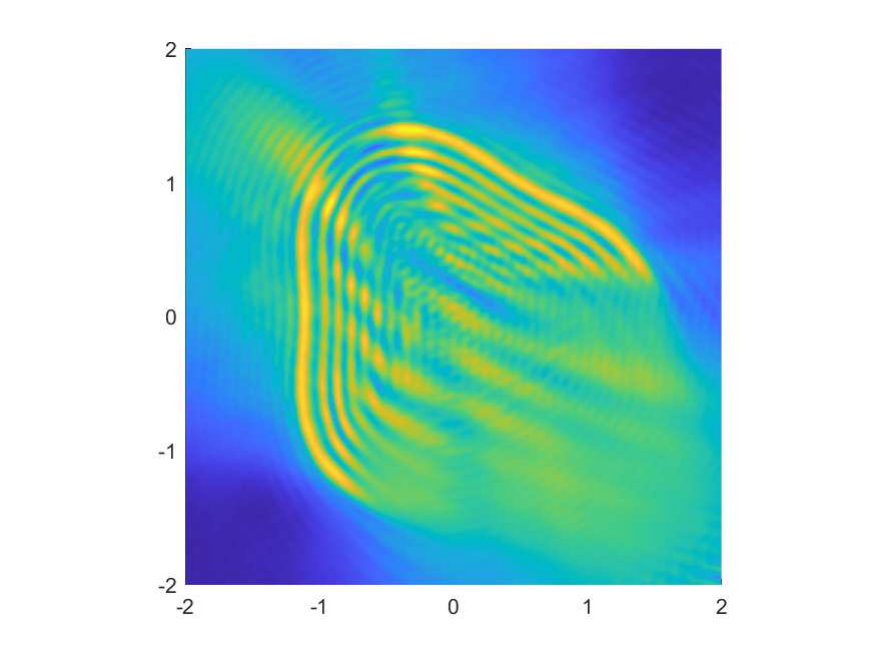}};

			\node[label] at (-1,-7) {\includegraphics[width=3.4cm, keepaspectratio]{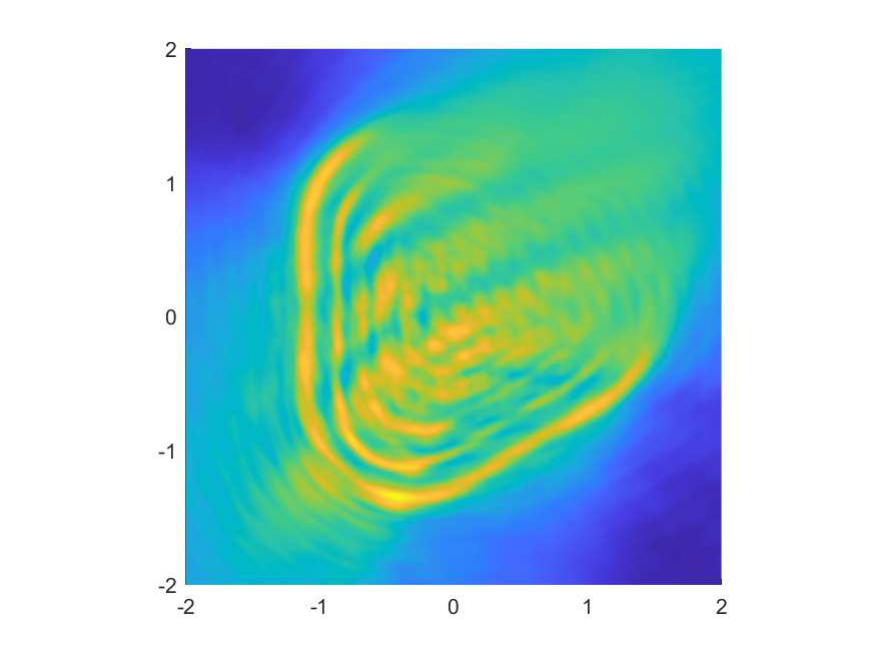}};
			\node[label] at (1.9,-7) {\includegraphics[width=3.4cm, keepaspectratio]{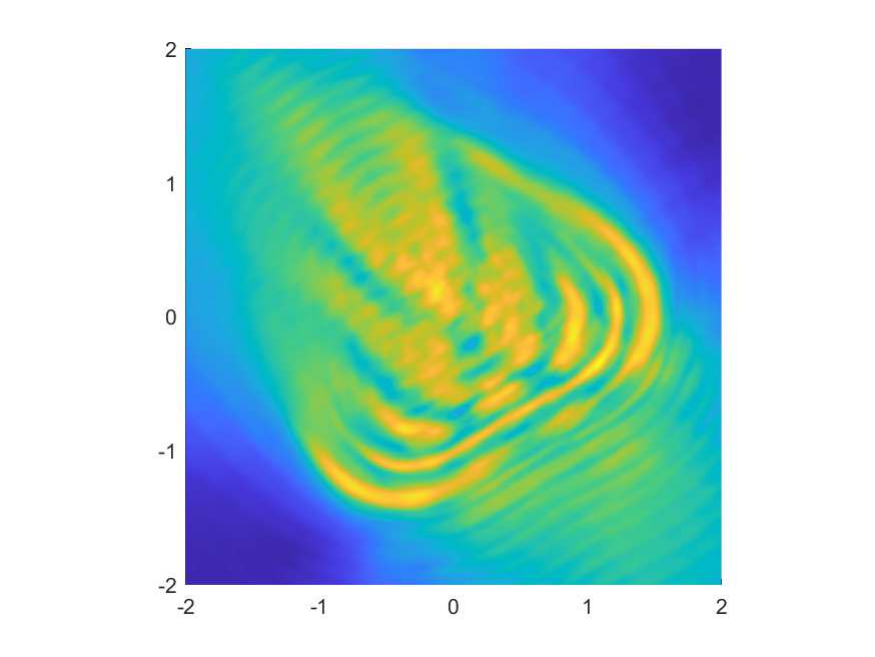}};
			\node[label] at (4.9,-7) {\includegraphics[width=3.4cm, keepaspectratio]{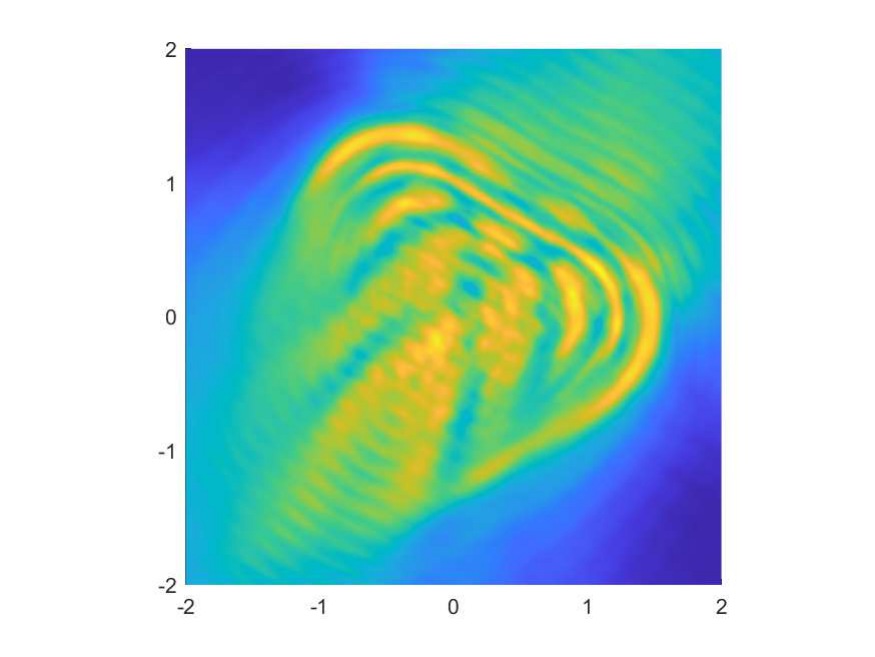}};
			\node[label] at (7.9,-7) {\includegraphics[width=3.4cm, keepaspectratio]{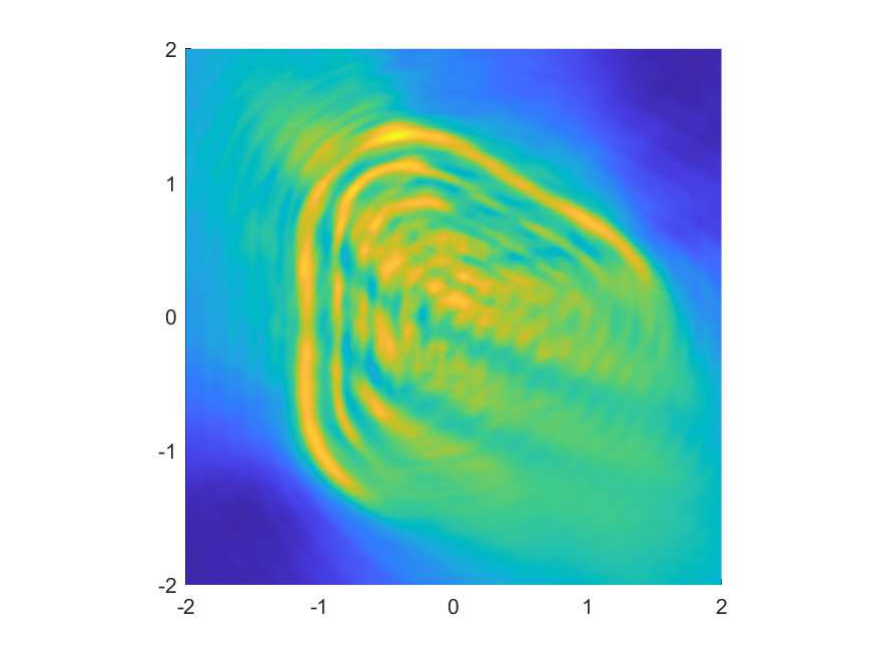}};

	\end{tikzpicture}}
	\caption{Limited-aperture reconstruction of the pear-shaped penetrable obstacle using indicators $\pmb{I}_{P,L}$, $\pmb{I}_{S,L}$ and $\pmb{I}_{F,L}$ with \(30\%\) noise level under quarter-circle incident direction. }\label{figure8}
\end{figure}

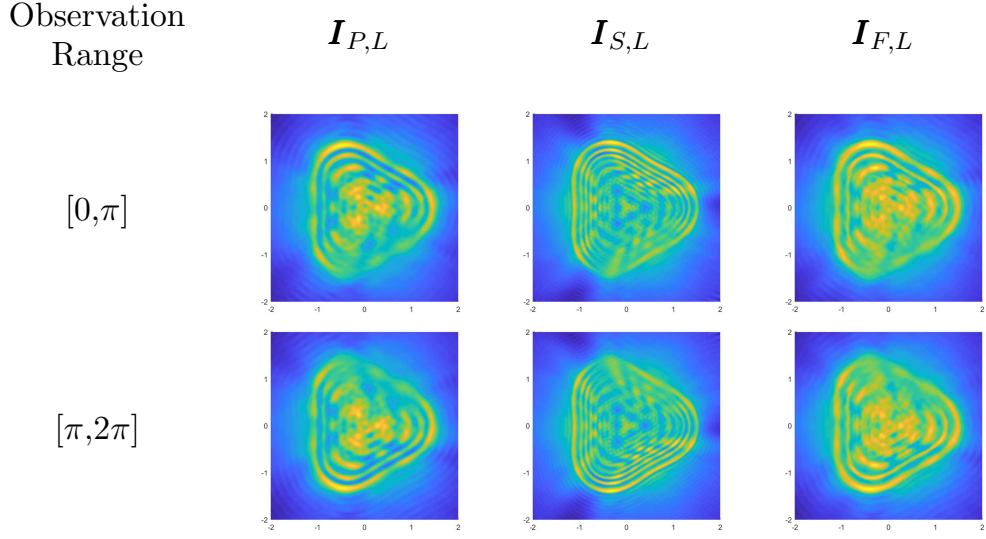
\begin{figure}[!htpb]
	\centering 
	\resizebox{0.85\textwidth}{!}{
		\begin{tikzpicture}[
			transform shape, 
			label/.style={font=\normalsize, anchor=center, align=center},
			leftlabel/.style={font=\normalsize, anchor=east, align=right}
			]
			
			\node[label] at (-3,0) {Observation\\Range};
			\node[label] at (0,0) {$\pmb{I}_{P,L}$};
			\node[label] at (3,0){$\pmb{I}_{S,L}$};
			\node[label] at (6,0) {$\pmb{I}_{F,L}$};
			
			\node[label] at (-3,-2) {[0,$\pi$]};
			\node[label] at (-3,-4.5) {[$\pi$,$2\pi$]};
			
			\node[label] at (0,-2) {\includegraphics[width=3.5cm, keepaspectratio]{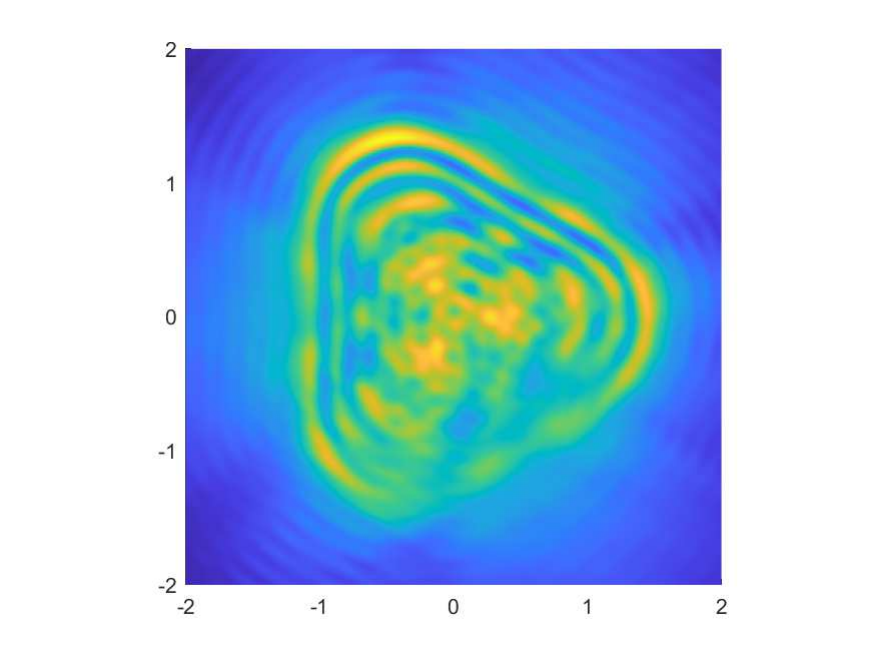}};
			\node[label] at (0,-4.5) {\includegraphics[width=3.5cm, keepaspectratio]{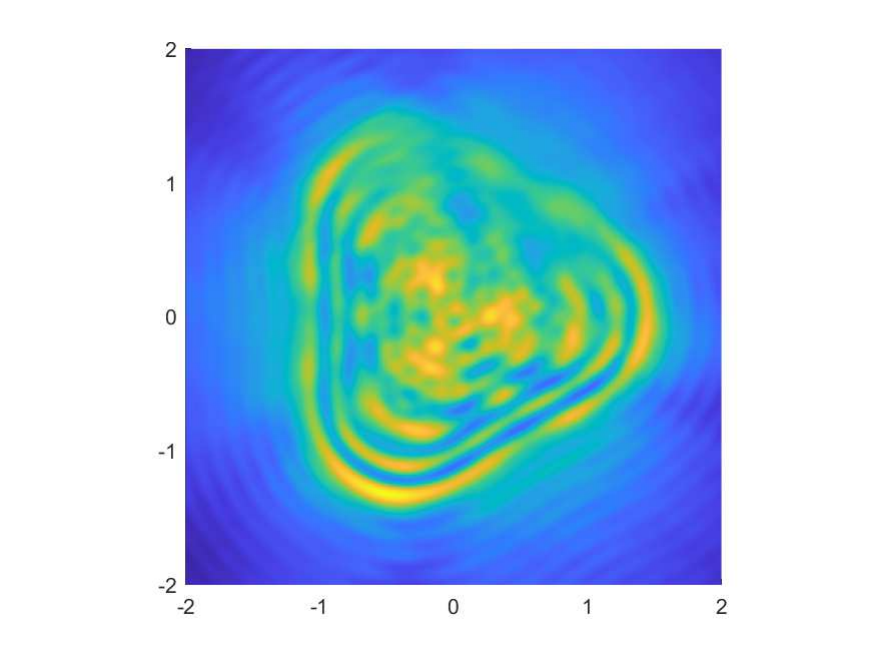}};
			
			\node[label] at (3,-2) {\includegraphics[width=3.5cm, keepaspectratio]{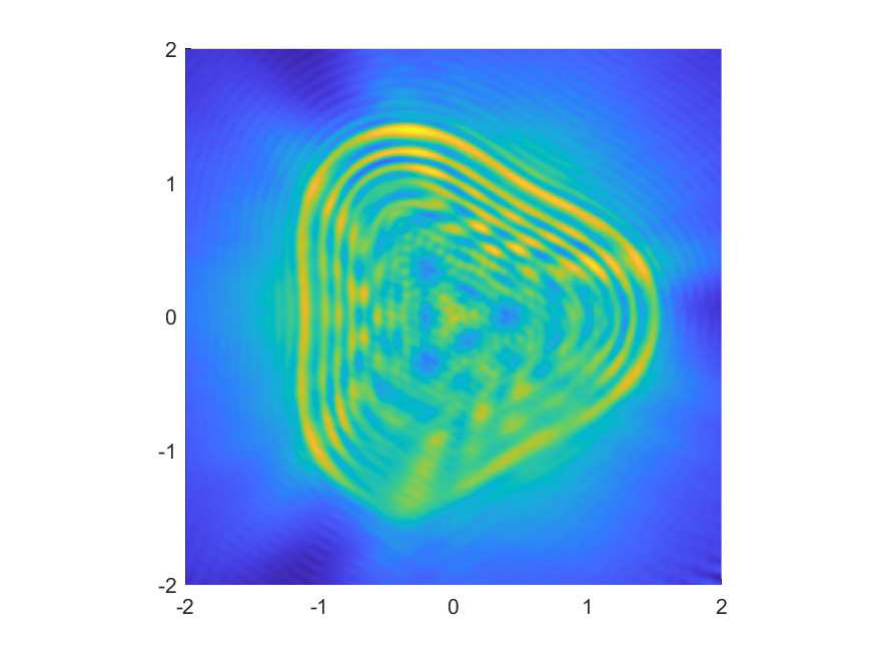}};
			\node[label] at (6,-2) {\includegraphics[width=3.5cm, keepaspectratio]{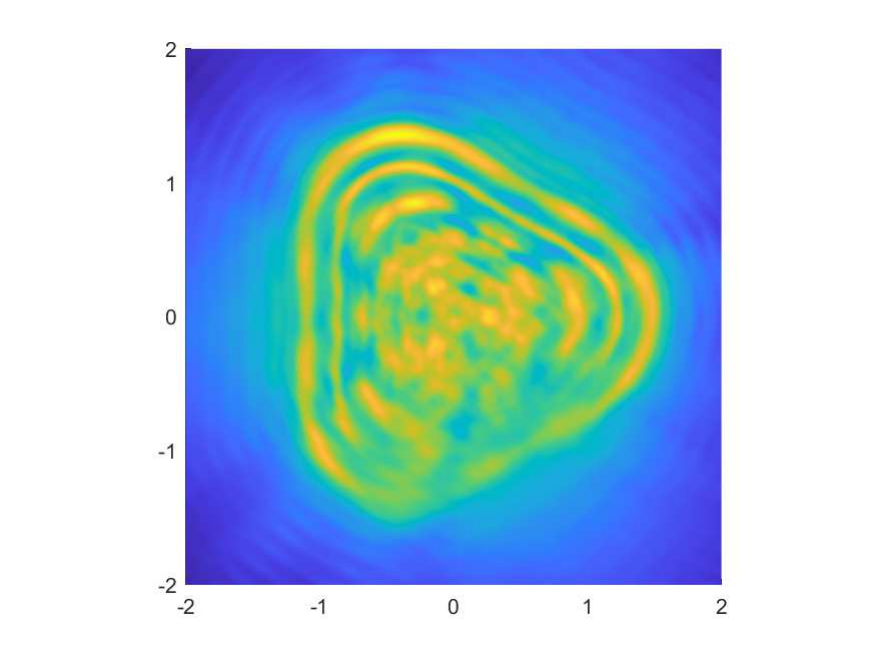}};
			
			\node[label] at (3,-4.5) {\includegraphics[width=3.5cm, keepaspectratio]{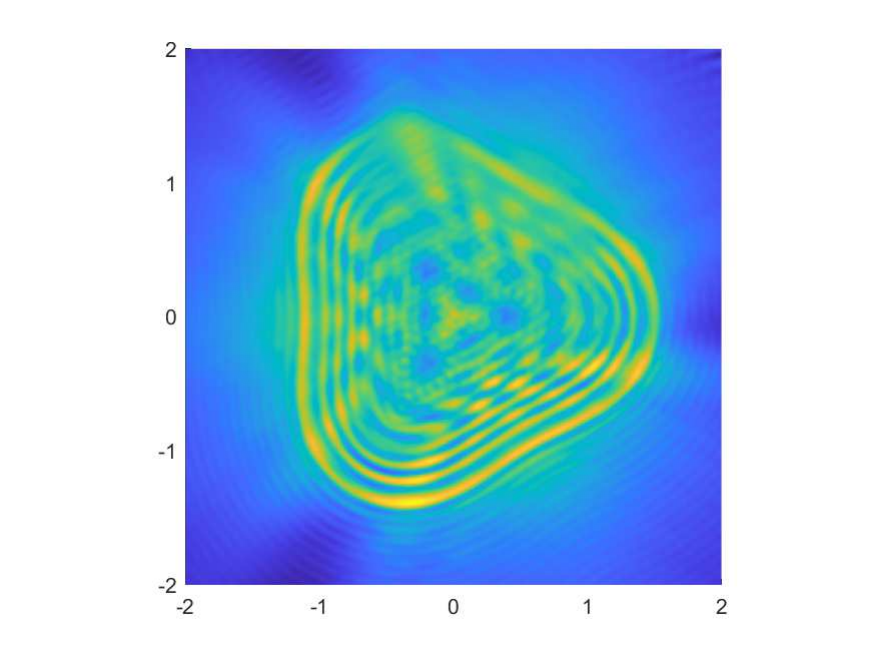}};
			\node[label] at (6,-4.5) {\includegraphics[width=3.5cm, keepaspectratio]{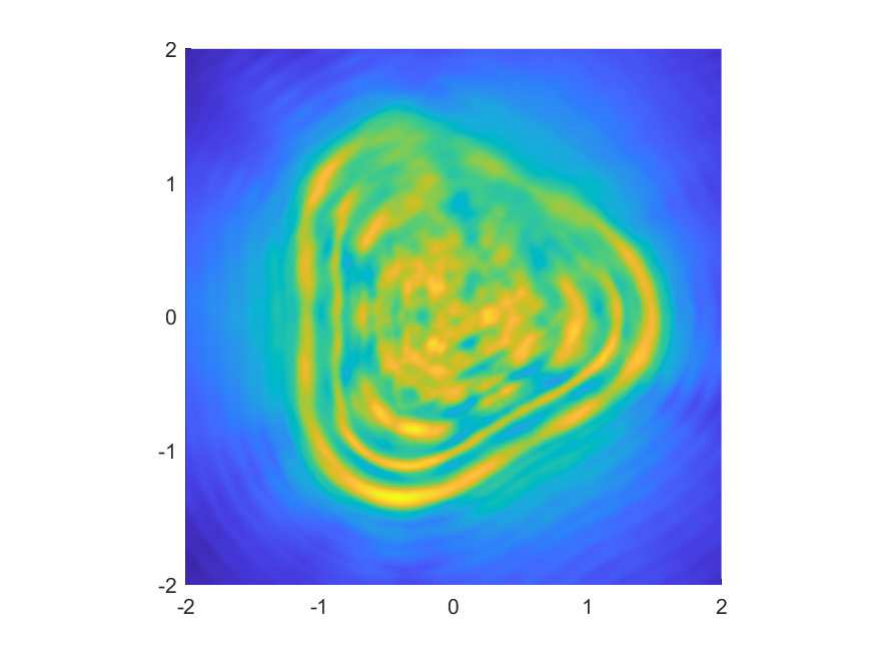}};
			
	\end{tikzpicture}}
	\caption{Limited-aperture reconstruction of the pear-shaped penetrable obstacle using indicators $\pmb{I}_{P,L}$, $\pmb{I}_{S,L}$ and $\pmb{I}_{F,L}$ with 30$\%$ noise level under half-circle observation direction. }\label{figure11}
\end{figure}

\subsection*{Example 4: Reconstructions with limited-aperture data.}

In this experiment, we investigate the inverse problem with limited-aperture data. We consider two scenarios: restricted incidence directions and restricted observation directions. 

First, with the observation directions fixed over the full range \([0, 2\pi]\), the incidence directions are restricted to specific intervals. Figure~\ref{figure9} presents reconstructions using the limited-aperture indicators \(\pmb{I}_{P,L}\), \(\pmb{I}_{S,L}\), and \(\pmb{I}_{F,L}\) over the incidence ranges \([0, \pi]\) and \([\pi, 2\pi]\). Figure~\ref{figure8} further shows reconstructions for narrower incidence ranges: \([0, \pi/2]\), \([\pi/2, \pi]\), \([\pi, 3\pi/2]\), and \([3\pi/2, 2\pi]\).

Second, for incident directions fixed over the full range \([0, 2\pi]\), figure~\ref{figure11} and figure~\ref{figure10} present reconstructions over a half-circle and a quarter-circle of observation directions, respectively, using the limited-aperture indicators \(\pmb{I}_{P,L}\), \(\pmb{I}_{S,L}\), and \(\pmb{I}_{F,L}\).

The results from both scenarios demonstrate that as the aperture decreases, the overall reconstruction of the obstacle becomes less precise. However, the boundary segments aligned with the available directions can be well recovered.

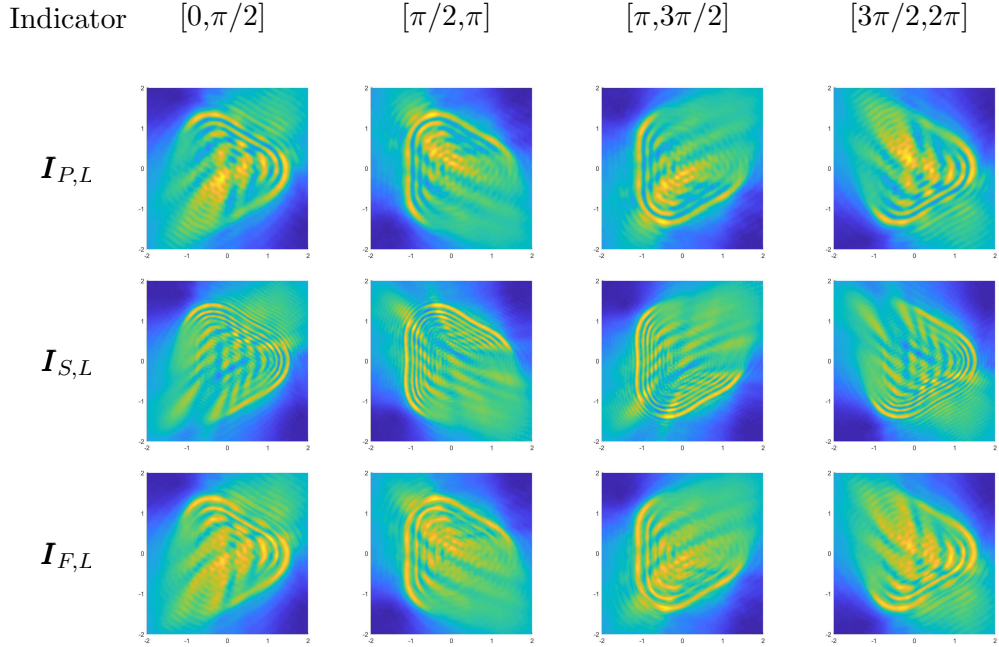
\begin{figure}[!htpb]
	\centering 
	\resizebox{0.85\textwidth}{!}{
		\begin{tikzpicture}[
			transform shape, 
			label/.style={font=\normalsize, anchor=center, align=center},
			leftlabel/.style={font=\normalsize, anchor=east, align=right}
			]
			
			\node[label] at (-1,0) {[0,$\pi/2$]};
			\node[label] at (1.9,0) {[$\pi/2$,$\pi$]};
			
			\node[label] at (4.9,0) {[$\pi$,$3\pi/2$]};
			\node[label] at (7.9,0) {[$3\pi/2$,$2\pi$]};

			\node[label] at (-3,0) {Indicator};
			\node[label] at (-3,-2) {$\pmb{I}_{P,L}$};
			\node[label] at (-3,-4.5){$\pmb{I}_{S,L}$};
			\node[label] at (-3,-7) {$\pmb{I}_{F,L}$};
			
			\node[label] at (-1,-2) {\includegraphics[width=3.4cm, keepaspectratio]{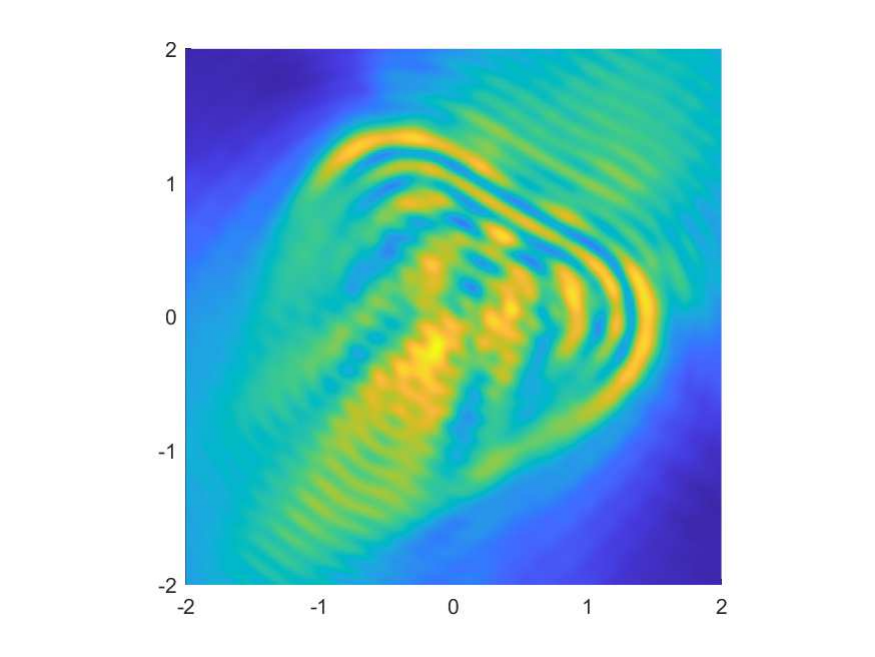}};
			\node[label] at (1.9,-2) {\includegraphics[width=3.4cm, keepaspectratio]{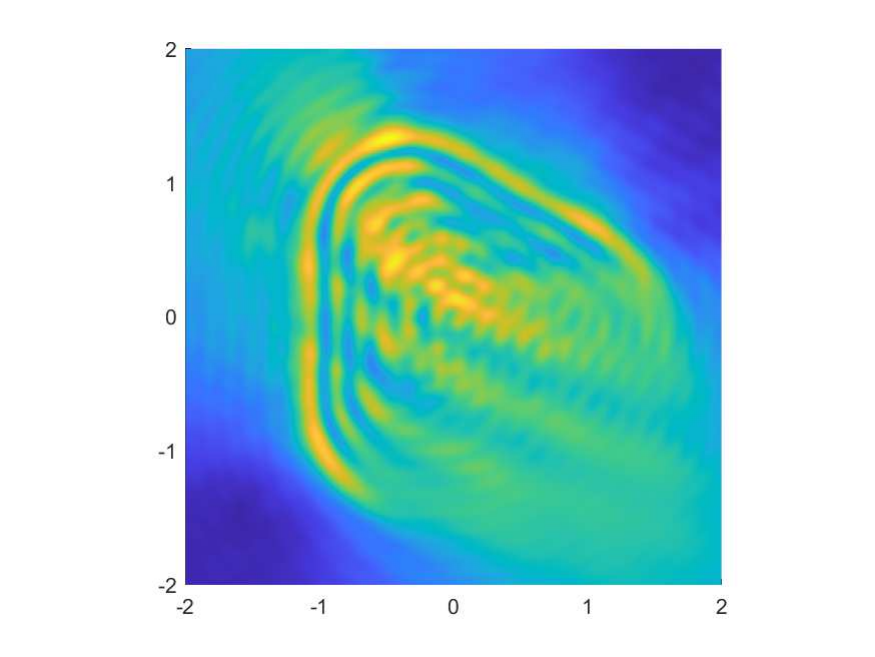}};
			\node[label] at (4.9,-2) {\includegraphics[width=3.4cm, keepaspectratio]{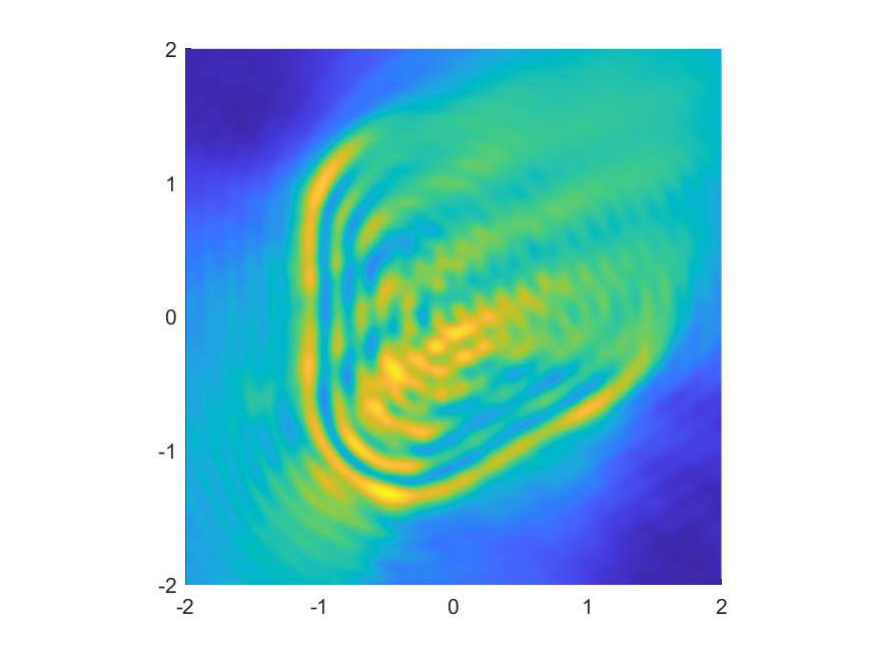}};
			\node[label] at (7.9,-2) {\includegraphics[width=3.4cm, keepaspectratio]{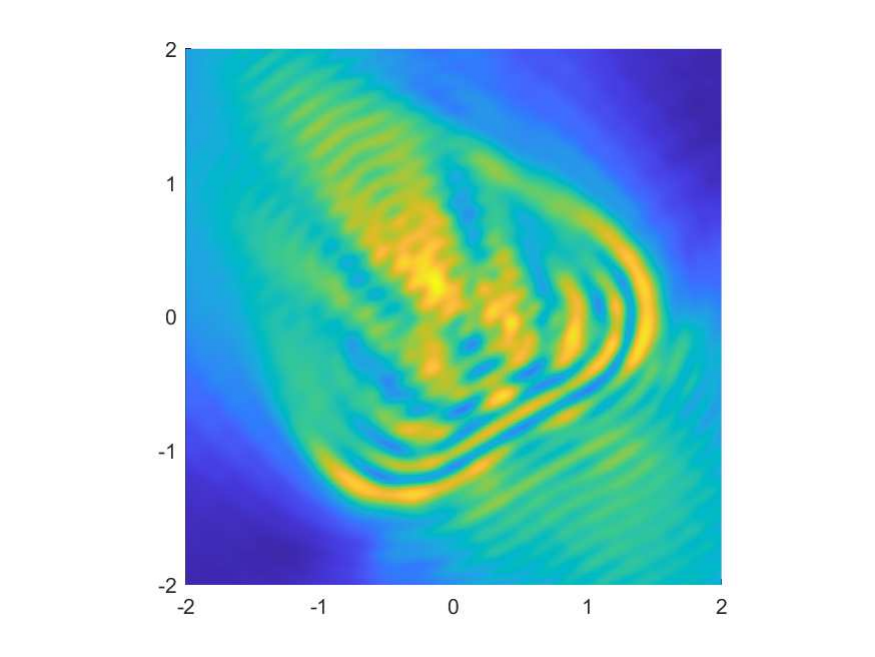}};

			\node[label] at (-1,-4.5) {\includegraphics[width=3.4cm, keepaspectratio]{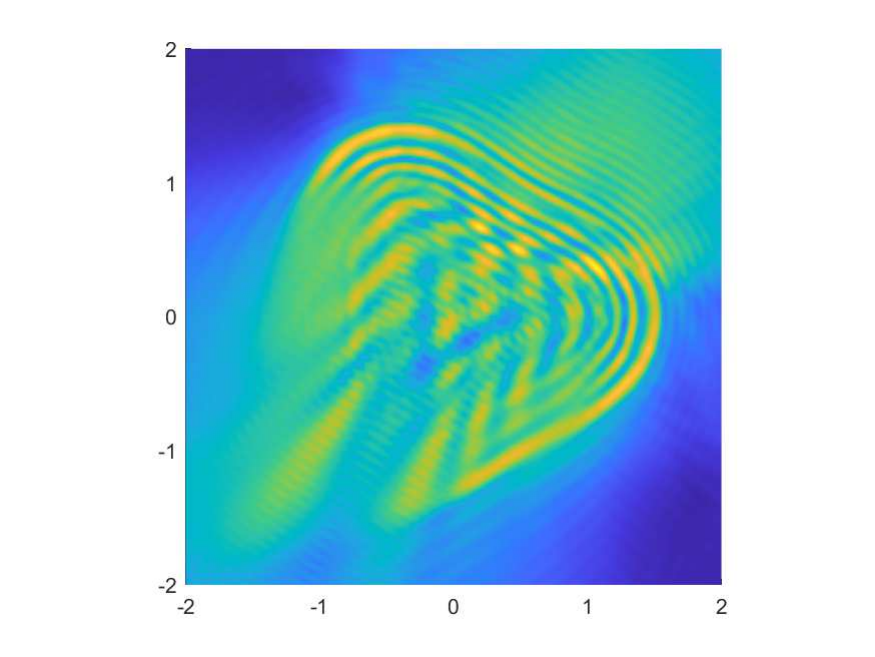}};
			\node[label] at (1.9,-4.5) {\includegraphics[width=3.4cm, keepaspectratio]{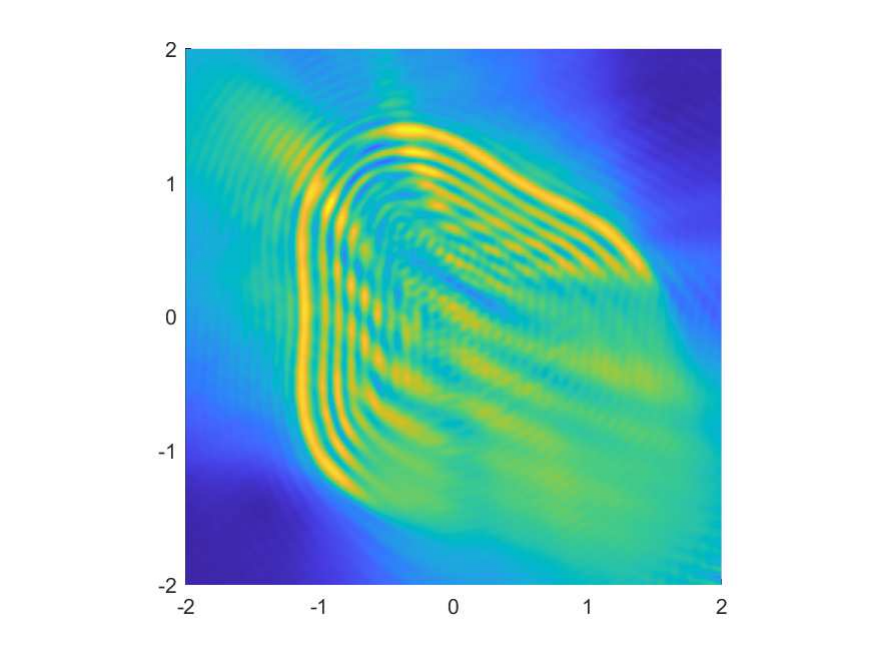}};
			\node[label] at (4.9,-4.5) {\includegraphics[width=3.4cm, keepaspectratio]{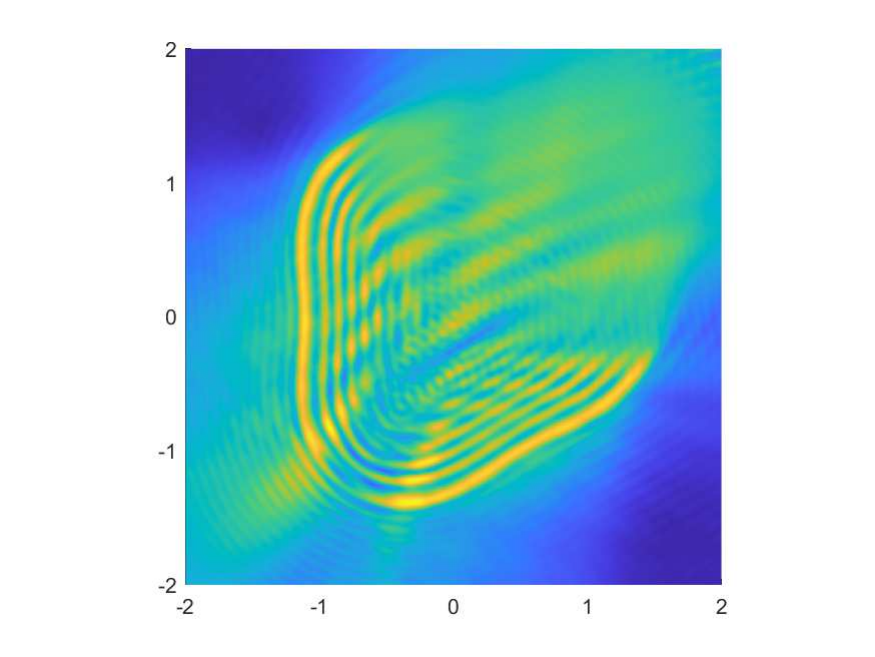}};
			\node[label] at (7.9,-4.5) {\includegraphics[width=3.4cm, keepaspectratio]{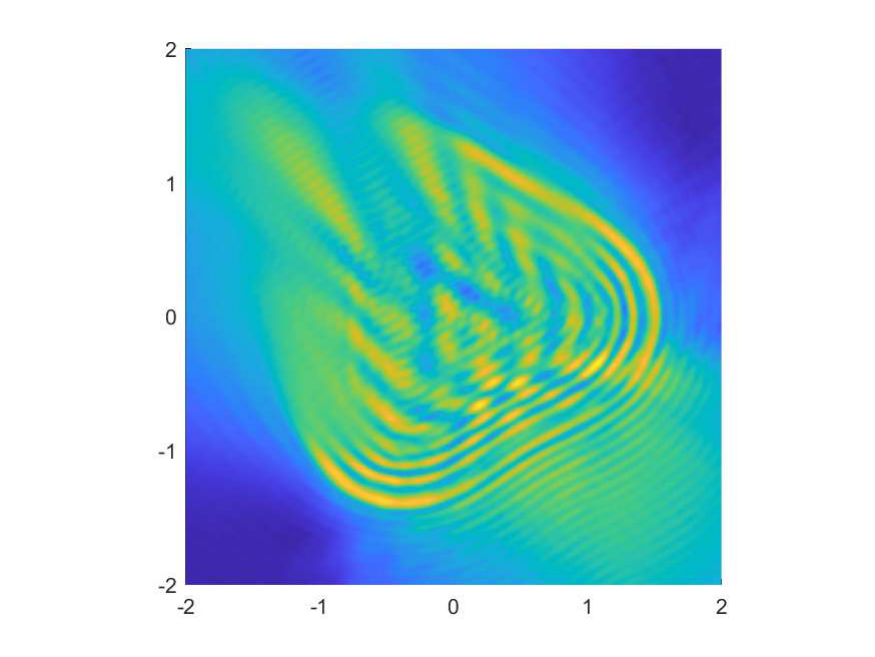}};

			\node[label] at (-1,-7) {\includegraphics[width=3.4cm, keepaspectratio]{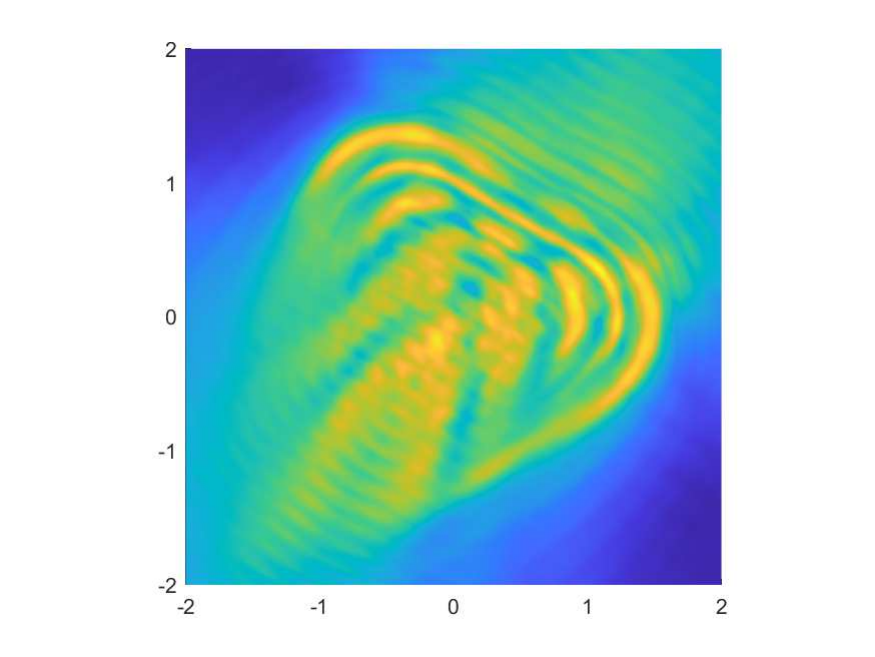}};
			\node[label] at (1.9,-7) {\includegraphics[width=3.4cm, keepaspectratio]{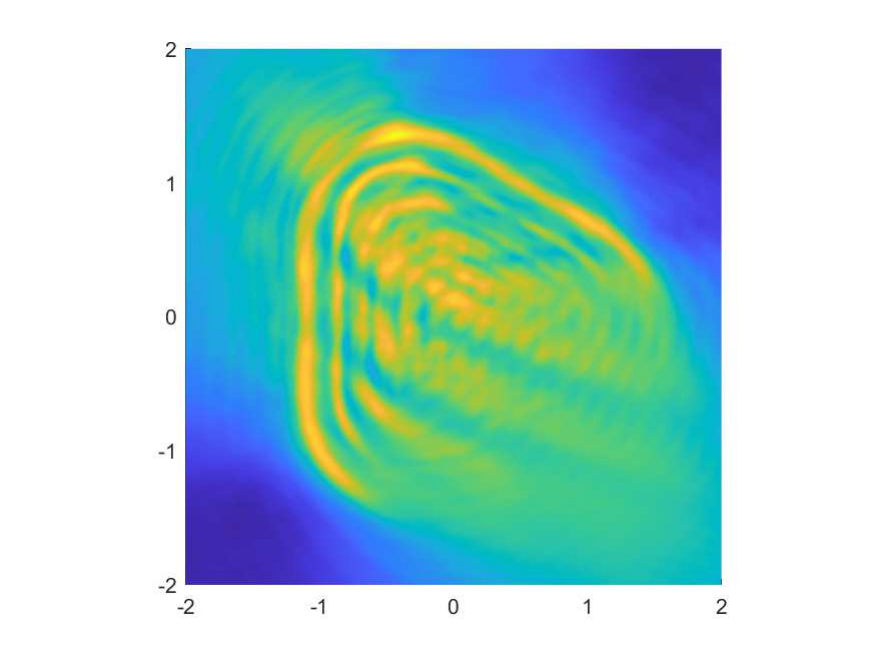}};
			\node[label] at (4.9,-7) {\includegraphics[width=3.4cm, keepaspectratio]{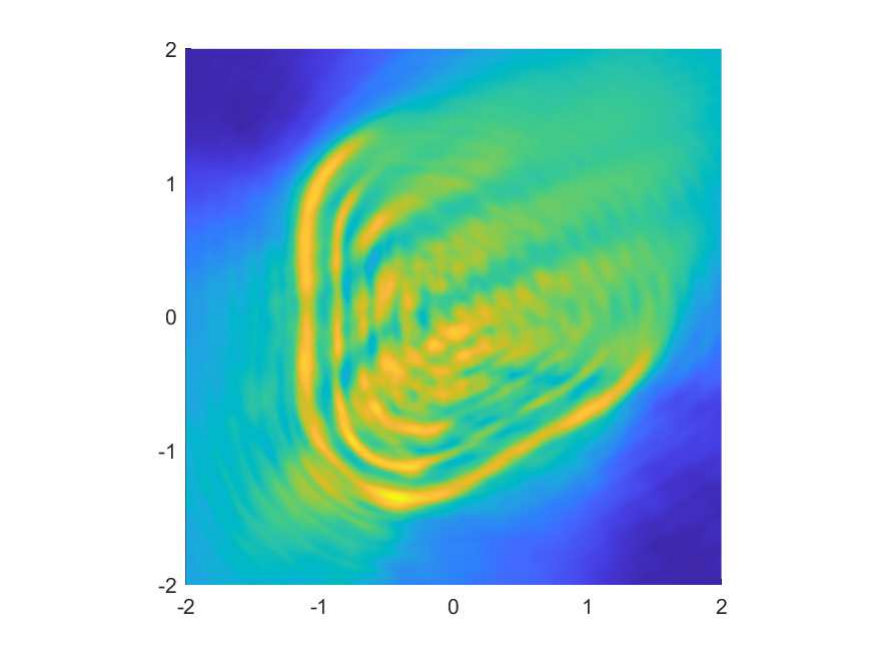}};
			\node[label] at (7.9,-7) {\includegraphics[width=3.4cm, keepaspectratio]{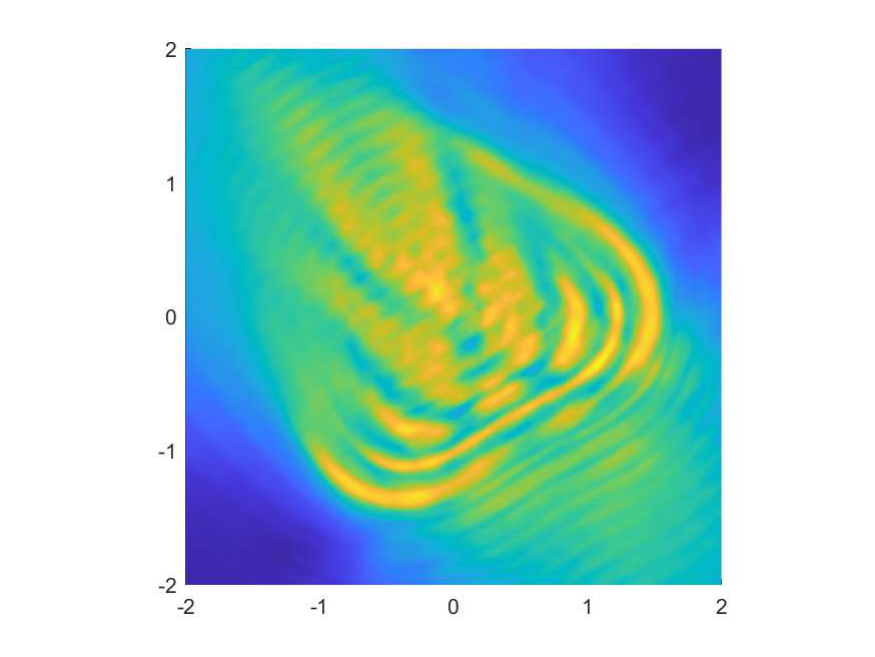}};

	\end{tikzpicture}}
	\caption{Limited-aperture reconstruction of the pear-shaped penetrable obstacle using indicators $\pmb{I}_{P,L}$, $\pmb{I}_{S,L}$ and $\pmb{I}_{F,L}$ with \(30\%\) noise level under quarter-circle observation direction. }\label{figure10}
\end{figure}

\section{Conclusions}

This work has considered the two-dimensional inverse elastic scattering problem of reconstructing a penetrable obstacle from far-field data. For the forward problem, the coupled Navier equations are reformulated as a system of scalar Helmholtz equations via the Helmholtz decomposition. The uniqueness of the associated coupled Helmholtz boundary value problem and of the corresponding boundary integral equation system is established, and an efficient Nystr\"{o}m-type discretization method is developed for the numerical solution of the resulting system. For the inverse problem, by leveraging the relation between the compressional and shear far-field patterns of the Navier system and those of the coupled Helmholtz system, we employ three efficient and robust indicators to reconstruct the location and shape of the penetrable obstacle. Furthermore, we derive the decay properties of these indicators for sampling points away from the obstacle and establish corresponding stability estimates. Numerical experiments demonstrate the effectiveness and robustness of the proposed method.

\section*{Acknowledgements}
The work of L.~Zhao is supported in part by the NSFC Grant 12401565 and CAUC Supporting Grant 3122025PT08.

%
%
%

\appendix
\counterwithin*{equation}{section}
\renewcommand{\theequation}{\thesection\arabic{equation}}

\section{Coupled Boundary Integral Equations for a Two-Obstacle Configuration}\label{appendix a}

Let $D_1, D_2 \subset \mathbb{R}^2$ be two disjoint bounded domains with smooth boundaries
$\partial D_1$ and $\partial D_2$, where $D_1$ is a penetrable elastic obstacle and
$D_2$ is a rigid obstacle. Let
\[
\pmb{U}_1 \in (C^2(D_1))^2 \cap (C^1(\overline{D_1}))^2,
\qquad
\pmb{U}_2 \in \bigl(C^2(\mathbb{R}^2\setminus \overline{D_1\cup D_2})\bigr)^2
\cap \bigl(C^1(\mathbb{R}^2\setminus (D_1\cup D_2))\bigr)^2
\]
be the displacement fields in $D_1$ and $\mathbb{R}^2\setminus \overline{D_1\cup D_2}$, respectively.
Then the boundary conditions are given by
\begin{equation}\label{bc2}
\begin{cases}
\pmb{U}_1=\pmb{U}_2
& \text{on } \partial D_1,\\[6pt]
T_{\pmb \nu_{D_1},1}(\pmb{U}_1)=T_{\pmb \nu_{D_1},2}(\pmb{U}_2)
& \text{on } \partial D_1,\\[6pt]
\pmb{U}_2=0
& \text{on } \partial D_2,
\end{cases}
\end{equation}
where $\pmb{\nu}_{D_1}$ and $\pmb{\tau}_{D_1}$ denote the unit exterior normal vector and the corresponding unit tangent vector on $\partial D_1$, respectively. Moreover, $\pmb{\nu}_{D_2}$ and $\pmb{\tau}_{D_2}$ denote the unit exterior normal vector and the corresponding unit tangent vector on $\partial D_2$, respectively.

Let $\pmb{v}:=\pmb{U}_2-\pmb{u}$ be the scattered field in the exterior domain.
Applying the Helmholtz decomposition to $\pmb{U}_1$ and $\pmb{v}$, we write
\[
\pmb{U}_1=\nabla\phi_1+\pmb{\operatorname{curl}}\psi_{1},
\qquad
\pmb v=\nabla \phi_2+\pmb{\operatorname{curl}}\psi_{2}.
\]
Substituting these representations into the boundary conditions \eqref{bc2}, then taking the inner products of the first two equations in \eqref{bc2} with $\pmb{\nu}_{D_1}$ and $\pmb{\tau}_{D_1}$ on $\partial D_1$, respectively, and taking the inner products of the third equation in \eqref{bc2} with $\pmb{\nu}_{D_2}$ and $\pmb{\tau}_{D_2}$ on $\partial D_2$, we obtain the following coupled boundary value problem:
\begin{equation}\label{bc3}
\left\{
\begin{array}{@{}l l@{}}
\partial_{\pmb{\nu}_{D_1}}\phi_{1} + \partial_{\pmb{\tau}_{D_1}}\psi_{1}
- \partial_{\pmb{\nu}_{D_1}}\phi_{2} - \partial_{\pmb{\tau}_{D_1}}\psi_{2}
= f_{1},
& \qquad \text{on } \partial D_1,\\[3pt]

\partial_{\pmb{\tau}_{D_1}}\phi_{1} - \partial_{\pmb{\nu}_{D_1}}\psi_{1}
- \partial_{\pmb{\tau}_{D_1}}\phi_{2} + \partial_{\pmb{\nu}_{D_1}}\psi_{2}
= f_{2},
& \qquad \text{on } \partial D_1,\\[3pt]
2\bigl(\mu_{1}\pmb{\nu}_{D_1}\cdot\partial_{\pmb{\nu}_{D_1}}\nabla\phi_{1}
+\mu_{1}\pmb{\nu}_{D_1}\cdot\partial_{\pmb{\nu}_{D_1}}\pmb{\operatorname{curl}}\psi_{1}\bigr)
-\lambda_{1}\kappa_{a}^{2}\phi_{1}\\[3pt]
\qquad
-2\bigl(\mu_{2}\pmb{\nu}_{D_1}\cdot\partial_{\pmb{\nu}_{D_1}}\nabla\phi_{2}
+\mu_{2}\pmb{\nu}_{D_1}\cdot\partial_{\pmb{\nu}_{D_1}}\pmb{\operatorname{curl}}\psi_{2}\bigr)
+\lambda_{2}\kappa_{p}^{2}\phi_{2}
= f_{3}, &\qquad \text{on } \partial D_1,\\[3pt]
2\bigl(\mu_{1}\pmb{\tau}_{D_1}\cdot\partial_{\pmb{\nu}_{D_1}}\nabla\phi_{1}
+\mu_{1}\pmb{\tau}_{D_1}\cdot\partial_{\pmb{\nu}_{D_1}}\pmb{\operatorname{curl}}\psi_{1}\bigr)
-\mu_{1}\kappa_{b}^{2}\psi_{1}\\[3pt]
\qquad
-2\bigl(\mu_{2}\pmb{\tau}_{D_1}\cdot\partial_{\pmb{\nu}_{D_1}}\nabla\phi_{2}
+\mu_{2}\pmb{\tau}_{D_1}\cdot\partial_{\pmb{\nu}_{D_1}}\pmb{\operatorname{curl}}\psi_{2}\bigr)
+\mu_{2}\kappa_{s}^{2}\psi_{2}
= f_{4}, & \qquad \text{on } \partial D_1,\\[3pt]

\partial_{\pmb{\nu}_{D_2}}\phi_{2}+\partial_{\pmb{\tau}_{D_2}}\psi_{2}=f_{5},
& \qquad \text{on } \partial D_2,\\[3pt]

\partial_{\pmb{\tau}_{D_2}}\phi_{2}-\partial_{\pmb{\nu}_{D_2}}\psi_{2}=f_{6},
& \qquad \text{on } \partial D_2.
\end{array}
\right.
\end{equation}
where
\[
f_{1}=\pmb{\nu}_{D_1}\cdot\pmb{u},\qquad
f_{2}=\pmb{\tau}_{D_1}\cdot\pmb{u},\qquad
f_{3}=\pmb{\nu}_{D_1}\cdot T_{\pmb \nu_{D_1},2}(\pmb{u}),\qquad
f_{4}=\pmb{\tau}_{D_1}\cdot T_{\pmb \nu_{D_1},2}(\pmb{u}),
\]
\[
f_{5}=-\pmb{\nu}_{D_2}\cdot\pmb{u},\qquad
f_{6}=-\pmb{\tau}_{D_2}\cdot\pmb{u}.
\]

We seek solutions of \eqref{bc3} in the following single-layer potential form:
\begin{align}
&\phi_1(\pmb{x})
= \int_{\partial D_1}\varPhi(\pmb{x},\pmb{y};\kappa_{a})g_{1,D_1}(\pmb{y})\mathrm{d}s(\pmb{y}),
\qquad \pmb{x}\in D_1,\notag\\
\label{single3}
&\psi_1(\pmb{x})
= \int_{\partial D_1}\varPhi(\pmb{x},\pmb{y};\kappa_{b})g_{2,D_1}(\pmb{y})\mathrm{d}s(\pmb{y}),
\qquad \pmb{x}\in D_1,\\
&\phi_2(\pmb{x})
= \sum_{\sigma}\int_{\partial \sigma}\varPhi(\pmb{x},\pmb{y};\kappa_{p})g_{3,\sigma}(\pmb{y})\mathrm{d}s(\pmb{y}),
\qquad \pmb{x}\in\mathbb{R}^2\setminus\overline{D_1\cup D_2},\notag\\
\label{single4}
&\psi_2(\pmb{x})
= \sum_{\sigma}\int_{\partial \sigma}\varPhi(\pmb{x},\pmb{y};\kappa_{s})g_{4,\sigma}(\pmb{y})\mathrm{d}s(\pmb{y}),
\qquad \pmb{x}\in\mathbb{R}^2\setminus\overline{D_1\cup D_2},
\end{align}
where the summation is taken over $\sigma\in\{D_1,D_2\}$. The densities satisfy
$g_{i,D_1}\in C^{1,\alpha}(\partial D_1)$ for $i=1,\ldots,4$ and
$g_{j,D_2}\in C^{1,\alpha}(\partial D_2)$ for $j=3,4$, with $0<\alpha<1$.

The boundary integral operators are defined by
\begin{equation*}
\begin{aligned}
S_{\kappa}^{\sigma,\chi}[g](\pmb{x})
&=\int_{\partial \sigma}\varPhi(\pmb{x},\pmb{y};\kappa)g(\pmb{y})\mathrm{d}s(\pmb{y}),
\qquad \pmb{x}\in\partial \chi,\\
N_{\kappa}^{\sigma,\chi}[g](\pmb{x})
&=\int_{\partial \sigma}\frac{\partial}{\partial\pmb{\nu}_{\chi}(\pmb{x})}
\varPhi(\pmb{x},\pmb{y};\kappa)g(\pmb{y})\mathrm{d}s(\pmb{y}),
\qquad \pmb{x}\in\partial \chi,\\
T_{\kappa}^{\sigma,\chi}[g](\pmb{x})
&=\int_{\partial \sigma}\frac{\partial}{\partial\pmb{\tau}_{\chi}(\pmb{x})}
\varPhi(\pmb{x},\pmb{y};\kappa)g(\pmb{y})\mathrm{d}s(\pmb{y}),
\qquad \pmb{x}\in\partial \chi,
\end{aligned}
\end{equation*}
where $\chi\in\{D_1,D_2\}$.

Next, we take the limit in \eqref{single3} as $\pmb{x}$ approaches $\partial D_1$ from the interior of $D_1$,
and in \eqref{single4} as $\pmb{x}$ approaches $\partial D_1$ and $\partial D_2$ from the exterior domain
$\mathbb{R}^2\setminus\overline{D_1\cup D_2}$.
Using the jump relations together with the boundary conditions in \eqref{bc2},
we arrive at the following coupled boundary integral equations on $\partial D_1$ and $\partial D_2$:

\begin{align*}
f_1 &= g_{1,D_1}/2 + N_{\kappa_{a}}^{D_1,D_1}[g_{1,D_1}]
+ T_{\kappa_{b}}^{D_1,D_1}[g_{2,D_1}]
+ g_{3,D_1}/2 - \sum_{\sigma}N_{\kappa_{p}}^{\sigma,D_1}[g_{3,\sigma}]
- \sum_{\sigma}T_{\kappa_{s}}^{\sigma,D_1}[g_{4,\sigma}],\\
f_2 &= T_{\kappa_{a}}^{D_1,D_1}[g_{1,D_1}] - g_{2,D_1}/2
- N_{\kappa_{b}}^{D_1,D_1}[g_{2,D_1}]
- \sum_{\sigma}T_{\kappa_{p}}^{\sigma,D_1}[g_{3,\sigma}]
- g_{4,D_1}/2 + \sum_{\sigma}N_{\kappa_{s}}^{\sigma,D_1}[g_{4,\sigma}],\\
f_3 &= 2\bigl( -\mu_1\kappa_{a}^2\pmb{\nu}_{D_1}^\top S_{\kappa_{a}}^{D_1,D_1}[\pmb{\nu}_{D_1}\pmb{\nu}_{D_1}^\top g_{1,D_1}]\pmb{\nu}_{D_1}
+ \mu_1\pmb{\nu}_{D_1}^\top N_{\kappa_{a}}^{D_1,D_1}[\pmb{\tau}_{D_1}\partial_{\pmb{\tau}_{D_1}} g_{1,D_1} + g_{1,D_1}\partial_{\pmb{\tau}_{D_1}}\pmb{\tau}_{D_1}]\\[2pt]
&\quad - \mu_1\pmb{\nu}_{D_1}^\top T_{\kappa_{a}}^{D_1,D_1}[\pmb{\nu}_{D_1}\partial_{\pmb{\tau}_{D_1}} g_{1,D_1} + g_{1,D_1}\partial_{\pmb{\tau}_{D_1}}\pmb{\nu}_{D_1}]
+ \mu_1\kappa_{b}^2\pmb{\nu}_{D_1}^\top S_{\kappa_{b}}^{D_1,D_1}[\pmb{\tau}_{D_1}\pmb{\nu}_{D_1}^\top g_{2,D_1}]\pmb{\nu}_{D_1}\\[2pt]
&\quad + \mu_1\pmb{\nu}_{D_1}^\top N_{\kappa_{b}}^{D_1,D_1}[\pmb{\nu}_{D_1}\partial_{\pmb{\tau}_{D_1}} g_{2,D_1} + g_{2,D_1}\partial_{\pmb{\tau}_{D_1}}\pmb{\nu}_{D_1}]
+ \mu_1\pmb{\nu}_{D_1}^\top T_{\kappa_{b}}^{D_1,D_1}[\pmb{\tau}_{D_1}\partial_{\pmb{\tau}_{D_1}} g_{2,D_1} + g_{2,D_1}\partial_{\pmb{\tau}_{D_1}}\pmb{\tau}_{D_1}] \notag\\[2pt]
&\quad + \mu_2\kappa_{p}^2\pmb{\nu}_{D_1}^\top \sum_{\sigma}S_{\kappa_{p}}^{\sigma,D_1}[\pmb{\nu}_{D_1}\pmb{\nu}_{D_1}^\top g_{3,\sigma}]\pmb{\nu}_{D_1}
- \mu_2\pmb{\nu}_{D_1}^\top \sum_{\sigma}N_{\kappa_{p}}^{\sigma,D_1}[\pmb{\tau}_{D_1}\partial_{\pmb{\tau}_{D_1}} g_{3,\sigma} + g_{3,\sigma}\partial_{\pmb{\tau}_{D_1}}\pmb{\tau}_{D_1}]\\
&\quad + \mu_2\pmb{\nu}_{D_1}^\top \sum_{\sigma}T_{\kappa_{p}}^{\sigma,D_1}[\pmb{\nu}_{D_1}\partial_{\pmb{\tau}_{D_1}} g_{3,\sigma} + g_{3,\sigma}\partial_{\pmb{\tau}_{D_1}}\pmb{\nu}_{D_1}]
- \mu_2\kappa_{s}^2\pmb{\nu}_{D_1}^\top \sum_{\sigma}S_{\kappa_{s}}^{\sigma,D_1}[\pmb{\tau}_{D_1}\pmb{\nu}_{D_1}^\top g_{4,\sigma}]\pmb{\nu}_{D_1}\\
&\quad - \mu_2\pmb{\nu}_{D_1}^\top \sum_{\sigma}N_{\kappa_{s}}^{\sigma,D_1}[\pmb{\nu}_{D_1}\partial_{\pmb{\tau}_{D_1}} g_{4,\sigma} + g_{4,\sigma}\partial_{\pmb{\tau}_{D_1}}\pmb{\nu}_{D_1}]
- \mu_2\pmb{\nu}_{D_1}^\top \sum_{\sigma}T_{\kappa_{s}}^{\sigma,D_1}[\pmb{\tau}_{D_1}\partial_{\pmb{\tau}_{D_1}} g_{4,\sigma} + g_{4,\sigma}\partial_{\pmb{\tau}_{D_1}}\pmb{\tau}_{D_1}] \bigr) \notag\\
&\quad - \lambda_1\kappa_{a}^2 S_{\kappa_{a}}^{D_1,D_1}[g_{1,D_1}]
+ \lambda_2\kappa_{p}^2 \sum_{\sigma}S_{\kappa_{p}}^{\sigma,D_1}[g_{3,\sigma}]
+ \mu_1(\pmb{\nu}_{D_1}\cdot\partial_{\pmb{\tau}_{D_1}}\pmb{\tau}_{D_1}) g_{1,D_1}
+ \mu_1(\pmb{\nu}_{D_1}\cdot\partial_{\pmb{\tau}_{D_1}}\pmb{\nu}_{D_1}) g_{2,D_1}\notag\\
&\quad + \mu_1\partial_{\pmb{\tau}_{D_1}} g_{2,D_1} + \mu_2(\pmb{\nu}_{D_1}\cdot\partial_{\pmb{\tau}_{D_1}}\pmb{\tau}_{D_1}) g_{3,D_1}
+ \mu_2(\pmb{\nu}_{D_1}\cdot\partial_{\pmb{\tau}_{D_1}}\pmb{\nu}_{D_1}) g_{4,D_1}
+ \mu_2\partial_{\pmb{\tau}_{D_1}} g_{4,D_1},\\[2pt]
f_4 &= 2\bigl( -\mu_1\kappa_{a}^2\pmb{\tau}_{D_1}^\top S_{\kappa_{a}}^{D_1,D_1}[\pmb{\nu}_{D_1}\pmb{\nu}_{D_1}^\top g_{1,D_1}]\pmb{\nu}_{D_1}
+ \mu_1\pmb{\tau}_{D_1}^\top N_{\kappa_{a}}^{D_1,D_1}[\pmb{\tau}_{D_1}\partial_{\pmb{\tau}_{D_1}} g_{1,D_1} + g_{1,D_1}\partial_{\pmb{\tau}_{D_1}}\pmb{\tau}_{D_1}]\\[2pt]
&\quad- \mu_1\pmb{\tau}_{D_1}^\top T_{\kappa_{a}}^{D_1,D_1}[\pmb{\nu}_{D_1}\partial_{\pmb{\tau}_{D_1}} g_{1,D_1} + g_{1,D_1}\partial_{\pmb{\tau}_{D_1}}\pmb{\nu}_{D_1}] 
 + \mu_1\kappa_{b}^2\pmb{\tau}_{D_1}^\top S_{\kappa_{b}}^{D_1,D_1}[\pmb{\tau}_{D_1}\pmb{\nu}_{D_1}^\top g_{2,D_1}]\pmb{\nu}_{D_1}\\[2pt]
&\quad+ \mu_1\pmb{\tau}_{D_1}^\top N_{\kappa_{b}}^{D_1,D_1}[\pmb{\nu}_{D_1}\partial_{\pmb{\tau}_{D_1}} g_{2,D_1} + g_{2,D_1}\partial_{\pmb{\tau}_{D_1}}\pmb{\nu}_{D_1}]
+ \mu_1\pmb{\tau}_{D_1}^\top T_{\kappa_{b}}^{D_1,D_1}[\pmb{\tau}_{D_1}\partial_{\pmb{\tau}_{D_1}} g_{2,D_1} + g_{2,D_1}\partial_{\pmb{\tau}_{D_1}}\pmb{\tau}_{D_1}] \notag\\[2pt]
&\quad + \mu_2\kappa_{p}^2\pmb{\tau}_{D_1}^\top \sum_{\sigma}S_{\kappa_{p}}^{\sigma,D_1}[\pmb{\nu}_{D_1}\pmb{\nu}_{D_1}^\top g_{3,\sigma}]\pmb{\nu}_{D_1}
- \mu_2\pmb{\tau}_{D_1}^\top \sum_{\sigma}N_{\kappa_{p}}^{\sigma,D_1}[\pmb{\tau}_{D_1}\partial_{\pmb{\tau}_{D_1}} g_{3,\sigma} + g_{3,\sigma}\partial_{\pmb{\tau}_{D_1}}\pmb{\tau}_{D_1}]\\
&\quad + \mu_2\pmb{\tau}_{D_1}^\top \sum_{\sigma}T_{\kappa_{p}}^{\sigma,D_1}[\pmb{\nu}_{D_1}\partial_{\pmb{\tau}_{D_1}} g_{3,\sigma} + g_{3,\sigma}\partial_{\pmb{\tau}_{D_1}}\pmb{\nu}_{D_1}]
- \mu_2\kappa_{s}^2\pmb{\tau}_{D_1}^\top \sum_{\sigma}S_{\kappa_{s}}^{\sigma,D_1}[\pmb{\tau}_{D_1}\pmb{\nu}_{D_1}^\top g_{4,\sigma}]\pmb{\nu}_{D_1}\\
&\quad- \mu_2\pmb{\tau}_{D_1}^\top \sum_{\sigma}N_{\kappa_{s}}^{\sigma,D_1}[\pmb{\nu}_{D_1}\partial_{\pmb{\tau}_{D_1}} g_{4,\sigma} + g_{4,\sigma}\partial_{\pmb{\tau}_{D_1}}\pmb{\nu}_{D_1}]
- \mu_2\pmb{\tau}_{D_1}^\top \sum_{\sigma}T_{\kappa_{s}}^{\sigma,D_1}[\pmb{\tau}_{D_1}\partial_{\pmb{\tau}_{D_1}} g_{4,\sigma} + g_{4,\sigma}\partial_{\pmb{\tau}_{D_1}}\pmb{\tau}_{D_1}] \bigr) \notag\\
&\quad - \mu_1\kappa_{b}^2 S_{\kappa_{b}}^{D_1,D_1}[g_{2,D_1}]
+ \mu_2\kappa_{s}^2 \sum_{\sigma}S_{\kappa_{s}}^{\sigma,D_1}[g_{4,\sigma}]
+ \mu_1(\pmb{\tau}_{D_1}\cdot\partial_{\pmb{\tau}_{D_1}}\pmb{\tau}_{D_1}) g_{1,D_1}
+ \mu_1\partial_{\pmb{\tau}_{D_1}} g_{1,D_1}
\notag\\
&\quad +
\mu_1(\pmb{\tau}_{D_1}\cdot\partial_{\pmb{\tau}_{D_1}}\pmb{\nu}_{D_1}) g_{2,D_1}  + \mu_2(\pmb{\tau}_{D_1}\cdot\partial_{\pmb{\tau}_{D_1}}\pmb{\tau}_{D_1}) g_{3,D_1}
+ \mu_2\partial_{\pmb{\tau}_{D_1}} g_{3,D_1}
+ \mu_2(\pmb{\tau}_{D_1}\cdot\partial_{\pmb{\tau}_{D_1}}\pmb{\nu}_{D_1}) g_{4,D_1},\\[2pt]
f_5 &= -g_{3,D_2}/2 + \sum_{\sigma}N_{\kappa_{p}}^{\sigma,D_2}[g_{3,\sigma}]
+ \sum_{\sigma}T_{\kappa_{s}}^{\sigma,D_2}[g_{4,\sigma}],\\
f_6 &= \sum_{\sigma}T_{\kappa_{p}}^{\sigma,D_2}[g_{3,\sigma}]
+ \sum_{\sigma}N_{\kappa_{s}}^{\sigma,D_2}[g_{4,\sigma}] + g_{4,D_2}/2,
\end{align*}
where $f_1$, $f_2$, $f_3$, $f_4$ are defined on $\partial D_1$, while $f_5$, $f_6$ are defined on $\partial D_2$.

\section{Full Discretization}\label{appendix b}
The Appendix presents the fully discretized form of equation \eqref{par1}-\eqref{par4}, obtained by applying the Nystr\"{o}m method in line with the approaches in \cite{Kress} and \cite{Drule}. The parametrized kernels \(\widetilde{S}\) and \(\widetilde{N}\) for the single-layer and normal-derivative operators admit the decompositions
\begin{align*}
	\widetilde {S}(\xi,\varsigma;\kappa)
	&=\widetilde {S}_{1}(\xi,\varsigma;\kappa)\ln\bigl(4\sin^{2}\frac{\xi-\varsigma}{2}\bigr) 
	+\widetilde {S}_{2}(\xi,\varsigma;\kappa),\\
	\widetilde N(\xi,\varsigma;\kappa)
	&=\widetilde N_{1}(\xi,\varsigma;\kappa)\ln\bigl(4\sin^{2}\frac{\xi-\varsigma}{2}\bigr)
	+ \widetilde N_{2}(\xi,\varsigma;\kappa),
\end{align*}
where
\begin{align*}
	\widetilde {S}_{1}(\xi,\varsigma;\kappa)
	&=-\frac{1}{4\pi}J_{0}\bigl(\kappa\lvert \pmb p(\xi)-\pmb p(\varsigma)\rvert\bigr),\\
	\widetilde {S}_{2}(\xi,\varsigma;\kappa)
	&=\widetilde {S}(\xi,\varsigma;\kappa)
	-\widetilde {S}_{1}(\xi,\varsigma;\kappa)\ln\bigl(4\sin^{2}\frac{\xi-\varsigma}{2}\bigr),\\
	\widetilde N_{1}(\xi,\varsigma;\kappa)
	&=\frac{\kappa}{4\pi}
	\pmb{n}(\xi)  \cdot  \bigl[\pmb p(\xi)-\pmb p(\varsigma)\bigr]
	\frac{J_{1}\bigl(\kappa\lvert \pmb p(\xi)-\pmb p(\varsigma)\rvert\bigr)}
	{\lvert \pmb p(\xi)-\pmb p(\varsigma)\rvert},\\
	\widetilde N_{2}(\xi,\varsigma;\kappa)
	&=\widetilde N(\xi,\varsigma;\kappa)
	-\widetilde N_{1}(\xi,\varsigma;\kappa)\ln\bigl(4\sin^{2}\frac{\xi-\varsigma}{2}\bigr),
\end{align*}
with $J_0$ and $J_1$ denoting the Bessel functions of the first kind with order zero and one, respectively. Moreover, the diagonal terms are
\begin{align*}
	\widetilde{S}_1(\xi,\xi;\kappa)&=-\frac{1}{4\pi}, &&\widetilde{S}_2(\xi,\xi;\kappa)=\frac{\mathrm{i}}{4}-\frac{E_c}{2\pi}-\frac{1}{2\pi}\ln(\frac{\kappa}{2}\left\vert \pmb p'(\xi)\right\vert),\\
	\widetilde{N}_1(\xi,\xi;\kappa)&=0, &&\widetilde{N}_2(\xi,\xi;\kappa)=\frac{1}{4\pi}\frac{\pmb{n}(\xi) \cdot \pmb p''(\xi)}{{\left\vert \pmb p'(\xi)\right\vert}^2},
\end{align*}
where \(E_{c}=0.57721\ldots\) is the Euler constant.

The kernel \(\widetilde T\) for the tangential-derivative operator admits the decomposition
\[
\widetilde T(\xi,\varsigma;\kappa)
= \widetilde T_1(\xi,\varsigma)\cot\frac{\varsigma-\xi}{2}
+ \widetilde T_2(\xi,\varsigma;\kappa)\ln  \Bigl(4\sin^2\frac{\xi-\varsigma}{2}\Bigr)
+ \widetilde T_3(\xi,\varsigma;\kappa),
\]
where
\[
\begin{aligned}
	\widetilde T_1(\xi,\varsigma)
	&=\frac{1}{2\pi}\pmb{n}(\xi)^{\perp}  \cdot  \bigl[\pmb p(\varsigma)-\pmb p(\xi)\bigr]
	\frac{\tan\tfrac{\varsigma-\xi}{2}}{\lvert \pmb p(\xi)-\pmb p(\varsigma)\rvert^2},\\
	\widetilde T_2(\xi,\varsigma;\kappa)
	&=\frac{\kappa}{4\pi}\pmb{n}(\xi)^{\perp}  \cdot  \bigl[\pmb p(\xi)-\pmb p(\varsigma)\bigr]
	\frac{J_{1}\bigl(\kappa\lvert \pmb p(\xi)-\pmb p(\varsigma)\rvert\bigr)}
	{\lvert \pmb p(\xi)-\pmb p(\varsigma)\rvert},\\
	\widetilde T_3(\xi,\varsigma;\kappa)
	&=\widetilde T(\xi,\varsigma;\kappa)
	- \widetilde T_1(\xi,\varsigma;\kappa)\cot\tfrac{\varsigma-\xi}{2}
	- \widetilde T_2(\xi,\varsigma;\kappa)\ln  \Bigl(4\sin^2\tfrac{\xi-\varsigma}{2}\Bigr),
\end{aligned}
\]
and the diagonal values are
\[
\widetilde T_1(\xi,\xi) = \frac{1}{4\pi},\quad
\widetilde T_2(\xi,\xi;\kappa) = 0,\quad
\widetilde T_3(\xi,\xi;\kappa) = 0.
\]

A set of equidistant quadrature nodes is defined as $\varsigma_{j}^{(n)}:=\pi j/n, j=0,\ldots,2n-1$. 
For the smooth integrals,the trapezoidal rule gives
\begin{align*}
	\int_0^{2\pi}K(\xi,\varsigma)f(\varsigma)\mathrm{d}\varsigma\approx\frac{\pi}{n}\sum_{j=0}^{2n-1}K(\xi,\varsigma_j^{(n)})f(\varsigma_j^{(n)}).
\end{align*}
To compute the singular integrals, we refer to the methods proposed in \cite{Kress} and \cite{Drule}, employing the following quadrature rules 
\begin{align*}
	\frac{1}{2\pi}\int_0^{2\pi}\cot\tfrac{\varsigma-\xi}{2}K(\xi,\varsigma)f(\varsigma)\mathrm{d}\varsigma&\approx\sum_{j=0}^{2n-1}U_j^{(n)}(\xi)K(\xi,\varsigma_j^{(n)})f(\varsigma_j^{(n)}),\\
	\int_0^{2\pi}\ln\left(4\sin^2\tfrac{\xi-\varsigma}{2}\right)K(\xi,\varsigma)f(\varsigma)\mathrm{d}\varsigma&\approx\sum_{j=0}^{2n-1}R_j^{(n)}(\xi)K(\xi,\varsigma_j^{(n)})f(\varsigma_j^{(n)}),
\end{align*}
where the function $K$ is continuous. The quadrature weights are given by
\begin{align*}
	U_j^{(n)}(\xi)&=\frac{1}{2n}\big[1-\cos n(\varsigma_j^{(n)}-\xi)\big]\cot\tfrac{\varsigma_j^{(n)}-\xi}{2},\\
	R_j^{(n)}(\xi)&=-\frac{2\pi}{n}\sum_{m=1}^{n-1}\frac{1}{m}\cos\left[m(\xi-\varsigma_j^{(n)})\right]-\frac{\pi}{n^2}\cos\left[n(\xi-\varsigma_j^{(n)})\right],
\end{align*}
and for integrals involving derivative terms, with the help of Lagrange bases
\begin{align*}
	\mathcal{L}_m(\varsigma)=\frac{1}{2n}\big\{1+2\sum_{k=1}^{n-1}\cos k(\varsigma-\varsigma_m^{(n)})+\cos n(\varsigma-\varsigma_m^{(n)})\big\}
\end{align*}
we obtain
\begin{align*}
	\int_{0}^{2\pi}K(\xi,\varsigma)f^{\prime}(\varsigma)\mathrm{d}\varsigma&\approx\frac{\pi}{n}\sum_{j=0}^{2n-1}\sum_{m=0}^{2n-1}d_{m-j}^{(n)}K(\xi,\varsigma_{m}^{(n)})f(\varsigma_{j}^{(n)}),\\
	\frac{1}{2\pi}\int_0^{2\pi}\cot\tfrac{\varsigma-\xi}{2}K(\xi,\varsigma)f^{\prime}(\varsigma)\mathrm{d}\varsigma&\approx\sum_{j=0}^{2n-1}\sum_{m=0}^{2n-1}d_{m-j}^{(n)}U_m^{(n)}(\xi)K(\xi,\varsigma_m^{(n)})f(\varsigma_j^{(n)}),\\
	\int_{0}^{2\pi}\ln\left(4\sin^{2}\tfrac{\xi-\varsigma}{2}\right)K(\xi,\varsigma)f^{\prime}(\varsigma)\mathrm{d}\varsigma&\approx\sum_{j=0}^{2n-1}\sum_{m=0}^{2n-1}d_{m-j}^{(n)}R_{m}^{(n)}(\xi)K(\xi,\varsigma_{m}^{(n)})f(\varsigma_{j}^{(n)}),
\end{align*}
where $d_{m-j}^{(n)}=\mathcal{L}^{\prime}_m(\varsigma_m^{(n)})$, the quadrature weight are given by
\begin{align*}
	d_{j}^{(n)}&=\begin{cases}\frac{(-1)^{j}}{2}\cot\frac{{j}\pi}{2n}&\quad j=\pm1,\cdots,\pm2n-1,\\
		0&\quad j=0.\end{cases}
\end{align*}
Following the approach in \cite{Djump}, we derive the Nystr\"{o}m discretization through the following combination
\begin{align*}
	\tilde{\varphi}_i^{(n)}(\varsigma)&=\sum_{j=0}^{2n-1}\varPhi_j^{(i)}\mathcal{L}_j(\varsigma),\\
	\tilde{\varphi}_i^{{\prime}(n)}(\varsigma)&=\sum_{j=0}^{2n-1}\varPhi_j^{(i)}\mathcal{L}_j^{\prime}(\varsigma),
\end{align*}
where $\tilde{\varphi}_i^{(n)}$ denotes a finite-dimensional representation of the densities $\tilde{\varphi}_i$, and $\varPhi_j^{(i)}=\tilde{\varphi}_i^{(n)}(\varsigma_j),$ $i=1,\ldots,4$. 

Therefore, the fully discretized scheme of \eqref{par1}-\eqref{par4} is
\begin{align*}
	w_{1,i}^{(n)}=&\sum_{j=0}^{2n-1}X_{i,j;a}^{(n)}J_{\pmb{p}}^2(\varsigma_j^{(n)})\tilde{\varrho}_{1,j}^{(n)}+\sum_{j=0}^{2n-1}Y_{i,j;b}^{(n)}J_{\pmb{p}}^2(\varsigma_j^{(n)})\tilde{\varrho}_{2,j}^{(n)}-\sum_{j=0}^{2n-1}X_{i,j;p}^{(n)}J_{\pmb{p}}^2(\varsigma_j^{(n)})\tilde{\varrho}_{3,j}^{(n)}\\
	&\quad -\sum_{j=0}^{2n-1}Y_{i,j;s}^{(n)}J_{\pmb{p}}^2(\varsigma_i^{(n)})\tilde{\varrho}_{4,j}^{(n)}+J_{\pmb{p}}(\varsigma_i^{(n)})/2\tilde{\varrho}_{1,i}^{(n)}+J_{\pmb{p}}(\varsigma_i^{(n)})/2\tilde{\varrho}_{3,i}^{(n)},\\
	w_{2,i}^{(n)}=&\sum_{j=0}^{2n-1}Y_{i,j;a}^{(n)}J_{\pmb{p}}^2(\varsigma_j^{(n)})\tilde{\varrho}_{1,j}^{(n)}-\sum_{j=0}^{2n-1}X_{i,j;b}^{(n)}J_{\pmb{p}}^2(\varsigma_j^{(n)})\tilde{\varrho}_{2,j}^{(n)}-\sum_{j=0}^{2n-1}Y_{i,j;p}^{(n)}J_{\pmb{p}}^2(\varsigma_j^{(n)})\tilde{\varrho}_{3,j}^{(n)}\\
	&\quad +\sum_{j=0}^{2n-1}X_{i,j;s}^{(n)}J_{\pmb{p}}^2(\varsigma_j^{(n)})\tilde{\varrho}_{4,j}^{(n)}-J_{\pmb{p}}(\varsigma_i^{(n)})/2\tilde{\varrho}_{2,i}^{(n)}-J_{\pmb{p}}(\varsigma_i^{(n)})/2\tilde{\varrho}_{4,i}^{(n)},\\
	w_{3,i}^{(n)}=&\sum_{j=0}^{2n-1}H_{i,j;a}^{(n)}\tilde{\varrho}_{1,j}^{(n)}+\sum_{j=0}^{2n-1}H_{i,j;b}^{(n)}\tilde{\varrho}_{2,j}^{(n)}+\sum_{j=0}^{2n-1}H_{i,j;p}^{(n)}\tilde{\varrho}_{3,j}^{(n)}+\sum_{j=0}^{2n-1}H_{i,j;s}^{(n)}\tilde{\varrho}_{4,j}^{(n)}\\
	&\quad +\mu_1m_3(\varsigma_{i}^{(n)},\varsigma_{i}^{(n)})/J_{\pmb{p}}(\varsigma_i^{(n)})\tilde{\varrho}_{1,i}^{(n)}+\mu_1m_4(\varsigma_{i}^{(n)},\varsigma_{i}^{(n)})/J_{\pmb{p}}(\varsigma_i^{(n)})\tilde{\varrho}_{2,i}^{(n)}\\
	&+\mu_1\sum_{m=0}^{2n-1}d_{m-i}^{(n)}\tilde{\varrho}_{2,m}^{(n)}+\mu_2m_3(\varsigma_{i}^{(n)},\varsigma_{i}^{(n)})/J_{\pmb{p}}(\varsigma_i^{(n)})\tilde{\varrho}_{3,i}^{(n)}\\
	&\quad +\mu_2m_4(\varsigma_{i}^{(n)},\varsigma_{i}^{(n)})/J_{\pmb{p}}(\varsigma_i^{(n)})\tilde{\varrho}_{4,i}^{(n)}+\mu_2\sum_{m=0}^{2n-1}d_{m-i}^{(n)}\tilde{\varrho}_{4,m}^{(n)},\\
	w_{4,i}^{(n)}=&\sum_{j=0}^{2n-1}M_{i,j;a}^{(n)}\tilde{\varrho}_{1,j}^{(n)}+\sum_{j=0}^{2n-1}M_{i,j;b}^{(n)}\tilde{\varrho}_{2,j}^{(n)}+\sum_{j=0}^{2n-1}M_{i,j;p}^{(n)}\tilde{\varrho}_{3,j}^{(n)}+\sum_{j=0}^{2n-1}M_{i,j;s}^{(n)}\tilde{\varrho}_{4,j}^{(n)}\\
	&\quad +\mu_1m_7(\varsigma_{i}^{(n)},\varsigma_{i}^{(n)})/J_{\pmb{p}}(\varsigma_i^{(n)})\tilde{\varrho}_{1,i}^{(n)}+\mu_1\sum_{m=0}^{2n-1}d_{m-i}^{(n)}\tilde{\varrho}_{1,m}^{(n)}\\
	&+\mu_1m_8(\varsigma_{i}^{(n)},\varsigma_{i}^{(n)})/J_{\pmb{p}}(\varsigma_i^{(n)})\tilde{\varrho}_{2,i}^{(n)}+\mu_2m_7(\varsigma_{i}^{(n)},\varsigma_{i}^{(n)})/J_{\pmb{p}}(\varsigma_i^{(n)})\tilde{\varrho}_{3,i}^{(n)}\\
	&\quad +\mu_2\sum_{m=0}^{2n-1}d_{m-i}^{(n)}\tilde{\varrho}_{3,m}^{(n)}+\mu_2m_8(\varsigma_{i}^{(n)},\varsigma_{i}^{(n)})/J_{\pmb{p}}(\varsigma_i^{(n)})\tilde{\varrho}_{4,i}^{(n)},
\end{align*}
where
\begin{align*}
	&m_1(\varsigma_{i}^{(n)},\varsigma_{j}^{(n)})=\tilde{\pmb{\nu}}^{\top}(\varsigma_{i}^{(n)})\pmb{n}(\varsigma_{j}^{(n)}),\quad &&m_2(\varsigma_{i}^{(n)},\varsigma_{j}^{(n)})=\tilde{\pmb{\nu}}^{\top}(\varsigma_{i}^{(n)})\pmb{n}^{\perp}(\varsigma_{j}^{(n)}),\\
	&m_3(\varsigma_{i}^{(n)},\varsigma_{j}^{(n)})=\tilde{\pmb{\nu}}^{\top}(\varsigma_{i}^{(n)})\pmb{n}^{\perp '}(\varsigma_{j}^{(n)}),\quad &&m_4(\varsigma_{i}^{(n)},\varsigma_{j}^{(n)})=\tilde{\pmb{\nu}}^{\top}(\varsigma_{i}^{(n)})\pmb{n}'(\varsigma_{j}^{(n)}),\\
	&m_5(\varsigma_{i}^{(n)},\varsigma_{j}^{(n)})=\tilde{\pmb{\tau}}^{\top}(\varsigma_{i}^{(n)})\pmb{n}(\varsigma_{j}^{(n)}),\quad &&m_6(\varsigma_{i}^{(n)},\varsigma_{j}^{(n)})=\tilde{\pmb{\tau}}^{\top}(\varsigma_{i}^{(n)})\pmb{n}^{\perp}(\varsigma_{j}^{(n)}),\\
	&m_7(\varsigma_{i}^{(n)},\varsigma_{j}^{(n)})=\tilde{\pmb{\tau}}^{\top}(\varsigma_{i}^{(n)})\pmb{n}^{\perp '}(\varsigma_{j}^{(n)}),\quad &&m_8(\varsigma_{i}^{(n)},\varsigma_{j}^{(n)})=\tilde{\pmb{\tau}}^{\top}(\varsigma_{i}^{(n)})\pmb{n}'(\varsigma_{j}^{(n)}),
\end{align*}
and
\begin{align*}
	X_{i,j;l}^{(n)}=& \big(R_{j}^{(n)}(\varsigma_{i}^{(n)})\widetilde{N}_1(\varsigma_{i}^{(n)},\varsigma_{j}^{(n)};\kappa_l)+\frac{\pi}{n}\widetilde{N}_2(\varsigma_{i}^{(n)},\varsigma_j^{(n)};\kappa_l)\big)/J_{\pmb{p}}(\varsigma_{i}^{(n)}),\\
	Y_{i,j;l}^{(n)}=& \big(2\pi U_j^{(n)}(\varsigma_{i}^{(n)})\widetilde{T}_1(\varsigma_{i}^{(n)},\varsigma_j^{(n)})+R_{j}^{(n)}\widetilde{T}_2(\varsigma_{i}^{(n)},\varsigma_j^{(n)};\kappa_l)+\frac{\pi}{n}\widetilde{T}_3(\varsigma_{i}^{(n)},\varsigma_j^{(n)};\kappa_l)\big)/J_{\pmb{p}}(\varsigma_{i}^{(n)}),\\
	Z_{i,j;l}^{(n)}=&
	R_{j}^{(n)}(\varsigma_{i}^{(n)})\widetilde{S}_1(\varsigma_{i}^{(n)},\varsigma_{j}^{(n)};\kappa_l)+\frac{\pi}{n}\widetilde{S}_2(\varsigma_{i}^{(n)},\varsigma_j^{(n)};\kappa_l),\\
	H_{i,j;a}^{(n)}=&-2\mu_1\Big[ \kappa_{a}^2Z_{i,j;a}^{(n)}m_1^2(\varsigma_{i}^{(n)},\varsigma_{j}^{(n)})-X_{i,j;a}^{(n)}m_3(\varsigma_{i}^{(n)},\varsigma_{j}^{(n)})+Y_{i,j;a}^{(n)}m_4(\varsigma_{i}^{(n)},\varsigma_{j}^{(n)})\\
	&\quad -\sum_{m=0}^{2n-1}d_{m-j}^{(n)} \Big(X_{i,m;a}^{(n)}m_2(\varsigma_{i}^{(n)},\varsigma_{m}^{(n)})-Y_{i,m;a}^{(n)}m_1(\varsigma_{i}^{(n)},\varsigma_{m}^{(n)})\Big)\Big]\\
	&-\lambda_1\kappa_{a}^2Z_{i,j;a}^{(n)}J_{\pmb{p}}^2(\varsigma_j^{(n)}),\\
	H_{i,j;b}^{(n)}=&2\mu_1\Big[\kappa_{b}^2Z_{i,j;b}^{(n)}m_2(\varsigma_{i}^{(n)},\varsigma_{j}^{(n)})m_1(\varsigma_{i}^{(n)},\varsigma_{j}^{(n)})+X_{i,j;b}^{(n)}m_4(\varsigma_{i}^{(n)},\varsigma_{j}^{(n)})\\
	&\quad +\sum_{m=0}^{2n-1}d_{m-j}^{(n)}\Big(X_{i,m;b}^{(n)}m_1(\varsigma_{i}^{(n)},\varsigma_{m}^{(n)})+Y_{i,m;b}^{(n)}m_2(\varsigma_{i}^{(n)},\varsigma_{m}^{(n)})\Big)\\
	&+Y_{i,j;b}^{(n)}m_3(\varsigma_{i}^{(n)},\varsigma_{j}^{(n)})\Big],\\
	H_{i,j;p}^{(n)}=&2\mu_2\Big[\kappa_{p}^2Z_{i,j;p}^{(n)}m_1^2(\varsigma_{i}^{(n)},\varsigma_{j}^{(n)})-X_{i,j;p}^{(n)}m_3(\varsigma_{i}^{(n)},\varsigma_{j}^{(n)})+Y_{i,j;p}^{(n)}m_4(\varsigma_{i}^{(n)},\varsigma_{j}^{(n)})\\
	&\quad -\sum_{m=0}^{2n-1}d_{m-j}^{(n)} \Big(X_{i,m;p}^{(n)}m_2(\varsigma_{i}^{(n)},\varsigma_{m}^{(n)})-Y_{i,m;p}^{(n)}m_1(\varsigma_{i}^{(n)},\varsigma_{m}^{(n)})\Big)\Big]\\
	&+\lambda_2\kappa_{p}^2Z_{i,j;p}^{(n)}J_{\pmb{p}}^2(\varsigma_j^{(n)}),\\
	H_{i,j;s}^{(n)}=&-2\mu_2\Big[\kappa_{s}^2Z_{i,j;s}^{(n)}m_2(\varsigma_{i}^{(n)},\varsigma_{j}^{(n)})m_1(\varsigma_{i}^{(n)},\varsigma_{j}^{(n)})+X_{i,j;s}^{(n)}m_4(\varsigma_{i}^{(n)},\varsigma_{j}^{(n)})\\
	&\quad +\sum_{m=0}^{2n-1}d_{m-j}^{(n)}\Big(X_{i,m;s}^{(n)}m_1(\varsigma_{i}^{(n)},\varsigma_{m}^{(n)})+Y_{i,m;s}^{(n)}m_2(\varsigma_{i}^{(n)},\varsigma_{m}^{(n)})\Big)\\
	&+Y_{i,j;s}^{(n)}m_3(\varsigma_{i}^{(n)},\varsigma_{j}^{(n)})\Big],
\end{align*}
\begin{align*}
	M_{i,j;a}^{(n)}=&-2\mu_1\Big[ \kappa_{a}^2Z_{i,j;a}^{(n)}m_5(\varsigma_{i}^{(n)},\varsigma_{j}^{(n)})m_1(\varsigma_{i}^{(n)},\varsigma_{j}^{(n)})-X_{i,j;a}^{(n)}m_7(\varsigma_{i}^{(n)},\varsigma_{j}^{(n)})\\
	&\quad -\sum_{m=0}^{2n-1}d_{m-j}^{(n)} \Big(X_{i,m;a}^{(n)}m_6(\varsigma_{i}^{(n)},\varsigma_{m}^{(n)})-Y_{i,m;a}^{(n)}m_5(\varsigma_{i}^{(n)},\varsigma_{m}^{(n)})\Big)\\
	&+Y_{i,j;a}^{(n)}m_8(\varsigma_{i}^{(n)},\varsigma_{j}^{(n)})\Big],\\
	M_{i,j;b}^{(n)}=&2\mu_1\Big[\kappa_{b}^2Z_{i,j;b}^{(n)}m_6(\varsigma_{i}^{(n)},\varsigma_{j}^{(n)})m_1(\varsigma_{i}^{(n)},\varsigma_{j}^{(n)})+X_{i,j;b}^{(n)}m_8(\varsigma_{i}^{(n)},\varsigma_{j}^{(n)})\\
	&\quad +\sum_{m=0}^{2n-1}d_{m-j}^{(n)}\Big(X_{i,m;b}^{(n)}m_5(\varsigma_{i}^{(n)},\varsigma_{m}^{(n)})+Y_{i,m;b}^{(n)}m_6(\varsigma_{i}^{(n)},\varsigma_{m}^{(n)})\Big)\\
	&+Y_{i,j;b}^{(n)}m_7(\varsigma_{i}^{(n)},\varsigma_{j}^{(n)})\Big]-\mu_1\kappa_{b}^2Z_{i,j;b}^{(n)}J_{\pmb{p}}^2(\varsigma_j^{(n)}),\\
	M_{i,j;p}^{(n)}=&2\mu_2\Big[\kappa_{p}^2Z_{i,j;p}^{(n)}m_5(\varsigma_{i}^{(n)},\varsigma_{j}^{(n)})m_1(\varsigma_{i}^{(n)},\varsigma_{j}^{(n)})-X_{i,j;p}^{(n)}m_7(\varsigma_{i}^{(n)},\varsigma_{j}^{(n)})\\
	&\quad -\sum_{m=0}^{2n-1}d_{m-j}^{(n)} \Big(X_{i,m;p}^{(n)}m_6(\varsigma_{i}^{(n)},\varsigma_{m}^{(n)})-Y_{i,m;p}^{(n)}m_5(\varsigma_{i}^{(n)},\varsigma_{m}^{(n)})\Big)\\
	&+Y_{i,j;p}^{(n)}m_8(\varsigma_{i}^{(n)},\varsigma_{j}^{(n)})\Big],\\
	M_{i,j;s}^{(n)}=&-2\mu_2\Big[\kappa_{s}^2Z_{i,j;s}^{(n)}m_6(\varsigma_{i}^{(n)},\varsigma_{j}^{(n)})m_1(\varsigma_{i}^{(n)},\varsigma_{j}^{(n)})+X_{i,j;s}^{(n)}m_8(\varsigma_{i}^{(n)},\varsigma_{j}^{(n)})\\
	&\quad +\sum_{m=0}^{2n-1}d_{m-j}^{(n)}\Big(X_{i,m;s}^{(n)}m_5(\varsigma_{i}^{(n)},\varsigma_{m}^{(n)})+Y_{i,m;s}^{(n)}m_6(\varsigma_{i}^{(n)},\varsigma_{m}^{(n)})\Big)\\
	&+Y_{i,j;s}^{(n)}m_7(\varsigma_{i}^{(n)},\varsigma_{j}^{(n)})\Big]+\mu_2\kappa_{s}^2Z_{i,j;s}^{(n)}J_{\pmb{p}}^2(\varsigma_j^{(n)}).
\end{align*}

\end{document}